\documentclass[preprint,11pt]{elsarticle}

\usepackage{graphicx}
\usepackage{subfigure}
\usepackage{float}
\usepackage{amssymb}
\usepackage{amsmath, nccmath}
\usepackage{xcolor}
\usepackage{booktabs} 
\usepackage{makecell}

\newtheorem{theorem}{Theorem}
\newtheorem{lemma}{Lemma}
\newtheorem{rem}{Remark}
\newtheorem{example}{Example}[section]
\newproof{proof}{Proof}
\numberwithin{equation}{section}

\newcommand{\vx}{\textbf{x}}
\newcommand{\vy}{\textbf{y}}
\newcommand{\vp}{\textbf{p}}
\newcommand{\vF}{\textbf{F}}

\newcommand{\vz}{\textbf{z}}

\newcommand{\vW}{\textbf{W}}

\newcommand{\be}{\begin{equation}}
\newcommand{\ee}{\end{equation}}
\begin{document}

\begin{frontmatter}


\title{Accuracy and Architecture Studies of Residual Neural Network solving Ordinary Differential Equations}


\author[label1]{Changxin Qiu}
\ead{cxqiu@iastate.edu}
\author[label1]{Aaron Bendickson\fnref{label4}}
\ead{atb@iastate.edu}
\author[label2]{Joshua Kalyanapu\fnref{label4}}
\ead{kalyan99@iastate.edu}
\author[label1]{Jue Yan\corref{cor1}\fnref{label3}}
\ead{jyan@iastate.edu}

\cortext[cor1]{Corresponding author}
\fntext[label3]{Research work of the author is supported by National Science Foundation grant DMS-1620335 and Simons Foundation grant 637716.}
\fntext[label4]{Research work of the authors are partially supported by National Science Foundation grant DMS-1457443.}
\address[label1]{Department of Mathematics, Iowa State University, Ames, 50011, USA}
\address[label2]{Department of Electrical and Computer Engineering, Iowa State University, Ames, 50011, USA}

\begin{abstract}
In this paper we consider utilizing a residual neural network (ResNet) to solve ordinary differential equations. Stochastic gradient descent method is applied to obtain the optimal parameter set of weights and biases of the network. We apply forward Euler, Runge-Kutta2 and Runge-Kutta4 finite difference methods to generate three sets of targets training the ResNet and carry out the target study. The well trained ResNet behaves just as its counterpart of the corresponding one-step finite difference method. In particular, we carry out (1) the architecture study in terms of number of hidden layers and neurons per layer to find the optimal ResNet structure; (2) the target study to verify the ResNet solver behaves as accurate as its finite difference method counterpart; (3) solution trajectory simulation. Even the ResNet solver looks like and is implemented in a way similar to forward Euler scheme, its accuracy can be as high as any one step method. A sequence of numerical examples are presented to demonstrate the performance of the ResNet solver.
\end{abstract}

\begin{keyword}
Deep neural network \sep Residual network \sep Ordinary differential equations


\end{keyword}

\end{frontmatter}

%\linenumbers

%% main text%%%%%%%%%%%%%%%%%%%%%%%%%%%%%%%%%%%%%%%%%%%%%%
\section{Introduction}
\label{S:1}

In recent years, there is a rapid growth in the study of theory and applications of machine learning with neural networks. This growth has been driven by advances in deep learning, which have achieved tremendous success such as large-scale image classification, text, videos and speech recognition, see \cite{LeCun-Bengio-1995,Bengio-2009,Krizhevsky-Sutskever-Hinton-2012,LeCun-Bengio-Hinton-2015,Wang-Osher2020}. In the last few years, exciting new works explore the connection between differential equations and machine learning. For example, the works of \cite{E-2017,Chaudharip-Osher2017,Haber-Ruthotto2018,Chang-Meng-Haber-Ruthotto-Begert-Holtham-2018} relate deep learning problems for general data to ordinary differential equations; we have the partial differential equations (PDEs) motivated deep neural networks architecture study of \cite{Ruthotto-Haber2020}; multi step method motivated architecture study of \cite{Lu-Zhong-Li-Dong-2017} and multi-grid method motivated convolutional neural network of \cite{He-Xu2019}.

Very recently, neural networks have also been explored to numerically solve PDEs. One popular class is to use neural network to represent the solution and take advantage of the approximation power of neural networks, for which we refer to classical results of \cite{Cybenko-1989,HORNIK-1990,Barron-1993,Pinkus-1999}. Here we list the early works of  \cite{Lagaris-Likas-Fortiadis1998,RUDD2015}, the popular  physics-informed methods of \cite{Raissi-Perdikaris-Karniadakis-2019} and recent works of \cite{Sirignano-Spiliopoulos-2018,Long-Lu-Dong-2019,Winovich-Ramani-Lin2019}, etc. They have the neural network input vector as the independent variables $(t,x)$ and output vector as the approximation of PDEs solution $u(t,x)$. Loss function usually involve the mean square error of sample collocation points satisfying the PDEs and have initial value and boundary conditions errors included also to increase the accuracy. Such methods have advantages of automatic differentiation and mess free, thus can be adapted to solve high dimensional PDEs \cite{Sirignano-Spiliopoulos-2018,Beck-E-Jentzen-2019}. 

On the other hand, time $t$ and spatial $x$ variables are not treated differently in the neural networks and the data are trained over a specific window of time and spatial domain. Thus the neural network solution can not be used to evolve approximation at a later time that is out of the domain of training data. Such methods are more suitable for elliptic type PDEs, see \cite{Fan-Lin-Ying-Zepeda-2019,khoo-lu-ying-2020,Li-Lu-Mao-2020}. For time dependent parabolic and hyperbolic PDEs with solution evolution in time, neural networks may be designed on or relate to traditional numerical methods such as finite difference and finite element methods. For numerical methods solving time dependent PDEs, we usually apply the method of lines and consider discretization in space first, then discretize in time after to obtain an explicit or implicit method. We are interested in applying neural networks solving PDEs. In this paper, we first explore neural networks to solve ordinary differential equations (ODE) without spatial variables.

In this article, we consider following autonomous ordinary differential equation system initial value problem
\begin{equation}\label{intro:eq-ODE}
    \frac{d \vx}{d t}=\vF(\vx(t)), ~~\vx(t_0)=\vx_0, 
\end{equation}
where $\vx \in R^n$ are the state variables. Integrate the dynamic system (\ref{intro:eq-ODE}) from $t_0$ to $t_0+\Delta$, where $\Delta$ is an increment or time lag, we obtain the following integral format of the ODE system
\begin{equation}\label{intro:ODE-integral}
    \vx(t_0+\Delta)-\vx(t_0)=\int_{t_0}^{t_0+\Delta} \vF(\vx(t)) ~dt.
\end{equation}
We consider $\vx(t_0)$ as the neural network input and have the  network output to approximate $\vx(t_0+\Delta)$. And we apply a simple feed-forward network to approximate the integral $\int_{t_0}^{t_0+\Delta} \vF(\vx(t)) ~dt$ or the integral of the dynamics of the ODE system. The format exactly matches the well known residual neural networks (ResNets) of \cite{He-Zhang-Ren-Sun-2016}. In a word, we train the network to obtain an optimal set of weight matrices and biases such that we can repeatedly call the ResNet solver to generate a discrete set of points values $\{\vp^k \approx \vx(t_k)\}_{k=0}^m$, which approximates the solution curve $\vx(t)$ or trace just as a finite difference method of forward Euler or Runge-Kutta scheme. 

Even the format of (\ref{intro:ODE-integral}) directly connects to the forward Euler scheme. It turns out applying Resnet to solve ODE system (\ref{intro:eq-ODE}) can be as accurate as any high order methods. For example, with $\Delta=0.1$ we observe the error of one step ResNet can be as small as $O(\Delta^5)\approx 10^{-5}$, even the ResNet solver is implemented in the way similar to first order forward Euler scheme. Numerical tests show the accuracy of the ResNet solver is more related to the quality of the target training the network.  

Ever since its introduction, residual networks become more and more popular and are considered as state-of-the-art models in numerous machine learning tasks, e.g. we have the Wide ResNets \cite{Zagoruyko-Komodakis-2016}, the DenseNets \cite{Huang-Liu-Van-Weinberger-2017}, the ResNeXt \cite{Xie-Girshick-2017} and many others.  %\cite{Targ-Almedia-Lyman-2016,Hardt-Ma-2016,Veit-Wilber-Belongie-2016,He-Gkioxari-Dollar-Girshick-2017,Toderici-2017,Xiong-2017,Chang-Meng-Haber-Ruthotto-Begert-Holtham-2018,Gomez-Ren-Urtasun-Grosse-2017}. 
With the development of ResNet and its variants on various applications, several attempts are carried out to explain ResNets through theoretical analysis and empirical results %\cite{Veit-Wilber-Belongie-2016, Huang-Sun-Liu-2016,Greff-Srivastava-Schmidhuber-2016,Jastrzebski-2017,
\cite{Haber-Ruthotto-Holtham-2017,E-2017}. Compared to the practical success of ResNets, there is little discussion on the architecture studies of ResNets. It remains a mystery whether there is a general principle to guide the design of number of hidden layers and neurons per layer of an effective network.

In this paper, we consider simple feed forward network to approximate the integral $\int_{t_0}^{t_0+\Delta} \vF(\vx(t)) ~dt$ of (\ref{intro:ODE-integral}) in terms of rectangular arrangements with width as the number of hidden layers and height as the number of neurons per layer. One objective is to vary the number of layers and neurons per layer and numerically find out the optimal architecture setting, for which the measurement is based on the error between ResNet output and the reference solution. In Figure \ref{intro:fig} we list the architecture study of a linear ODE system (subfigure (a)) and a nonlinear ODE system (subfigure (c)). Based on the mean $L_2$ error, we pick the setting of one hidden layer and six neurons as the optimal architecture for the linear nodal sink ODE and pick the setting of two hidden layers and forty neurons per layer as the optimal architecture for the nonlinear damping pendulum ODE.
\begin{figure}[htbp]
\centering
\subfigure[\tiny{Architecture study of nodal sink linear ODE}]{
\includegraphics[width=0.33\linewidth]{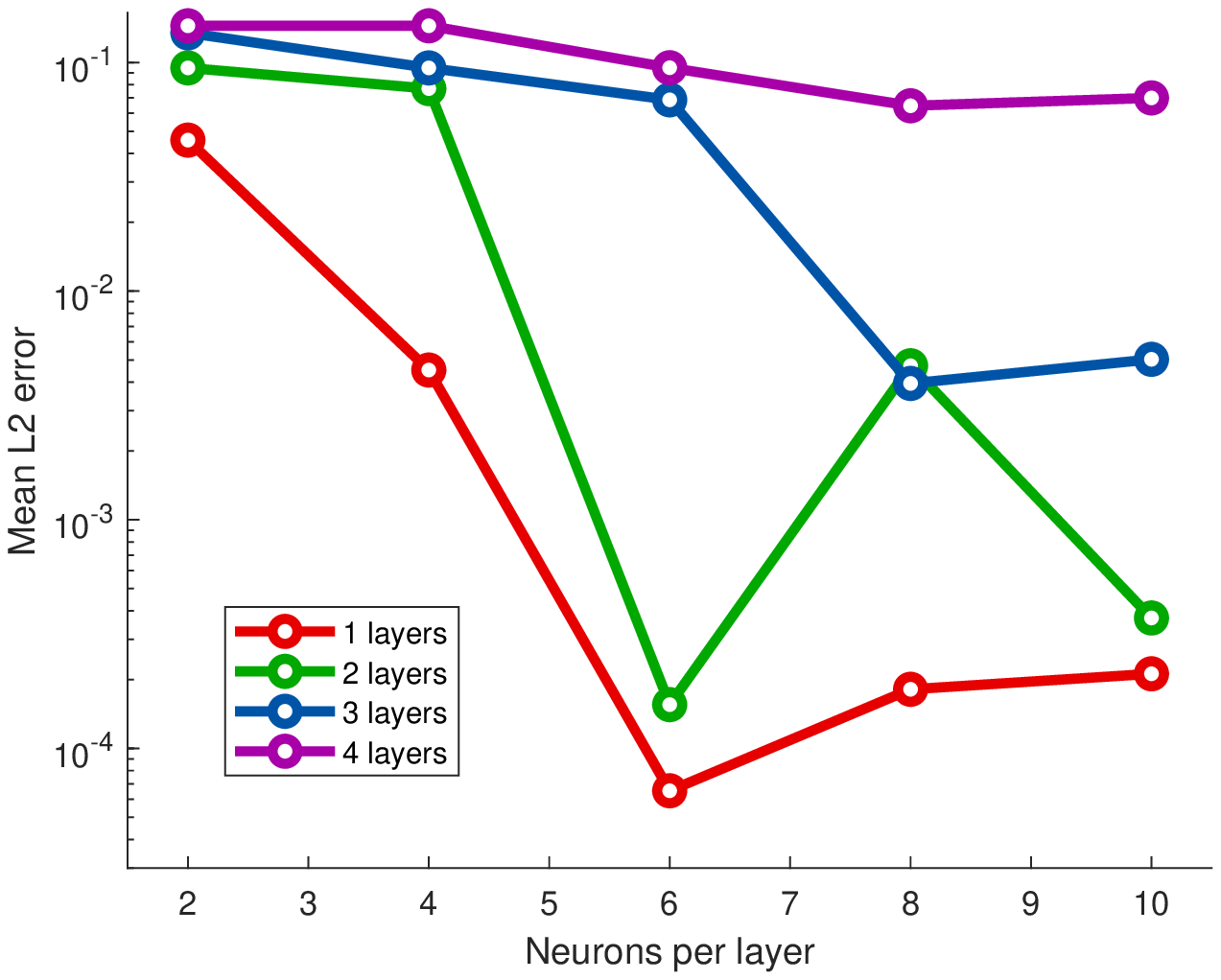}
%\caption{pic1}
}
\subfigure[\tiny{Target study of nodal sink linear ODE}]{
\includegraphics[width=0.33\linewidth]{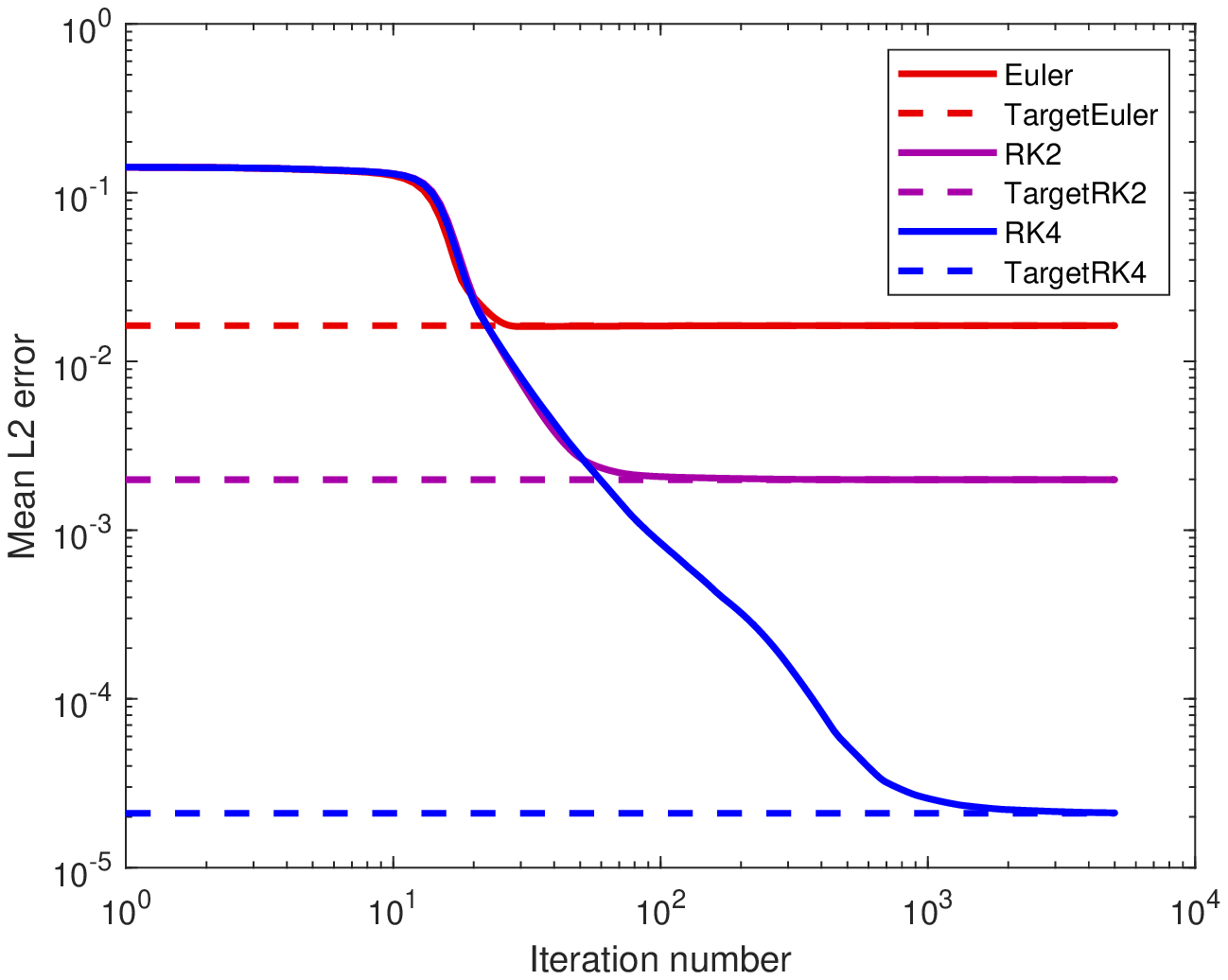}
%\caption{}
}\quad\quad
\subfigure[\tiny{Architecture study of damping pendulum}]{
\includegraphics[width=0.33\linewidth]{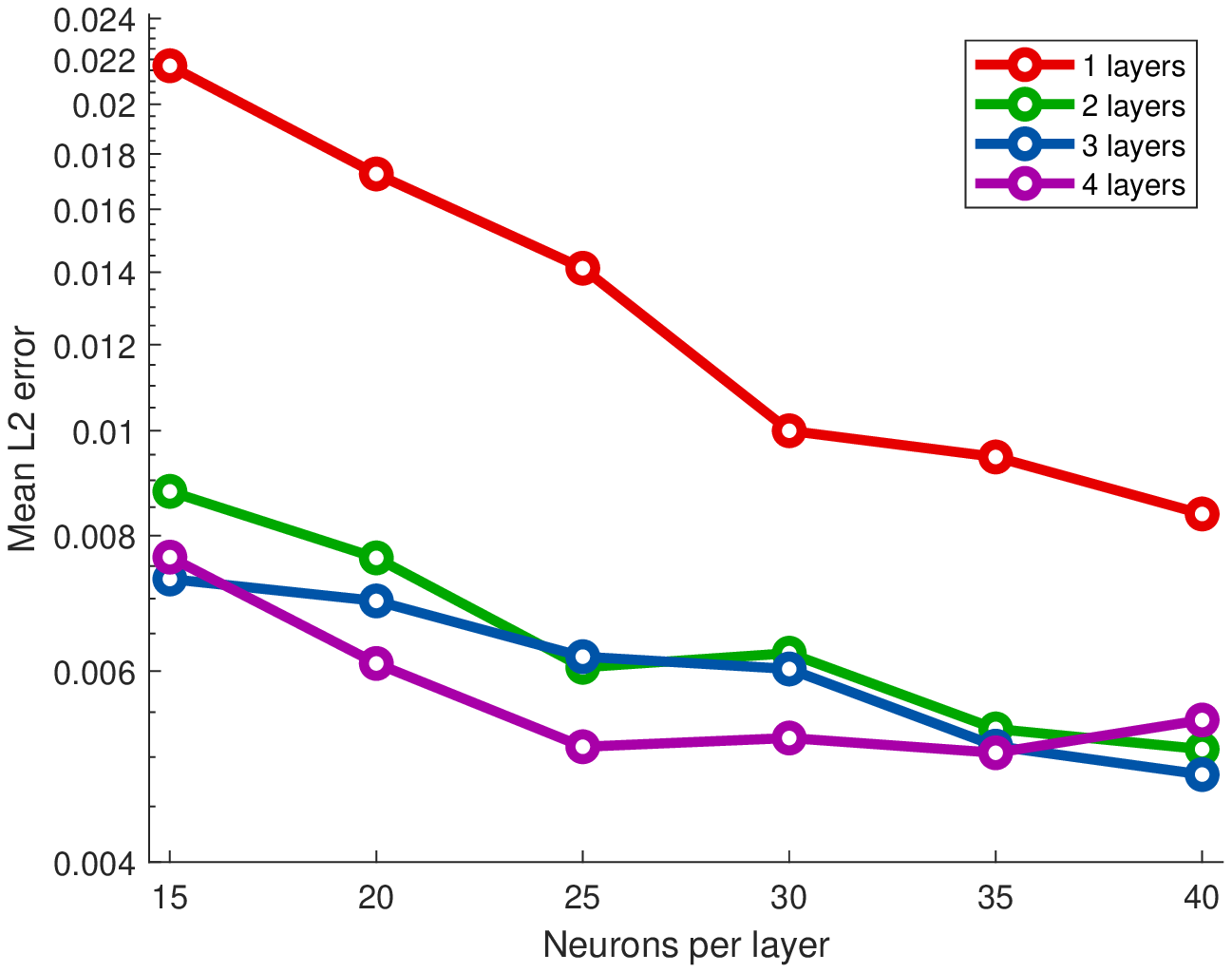}
%\caption{}
}
\quad
\subfigure[\tiny{Target study of damping pendulum}]{
\includegraphics[width=0.33\linewidth]{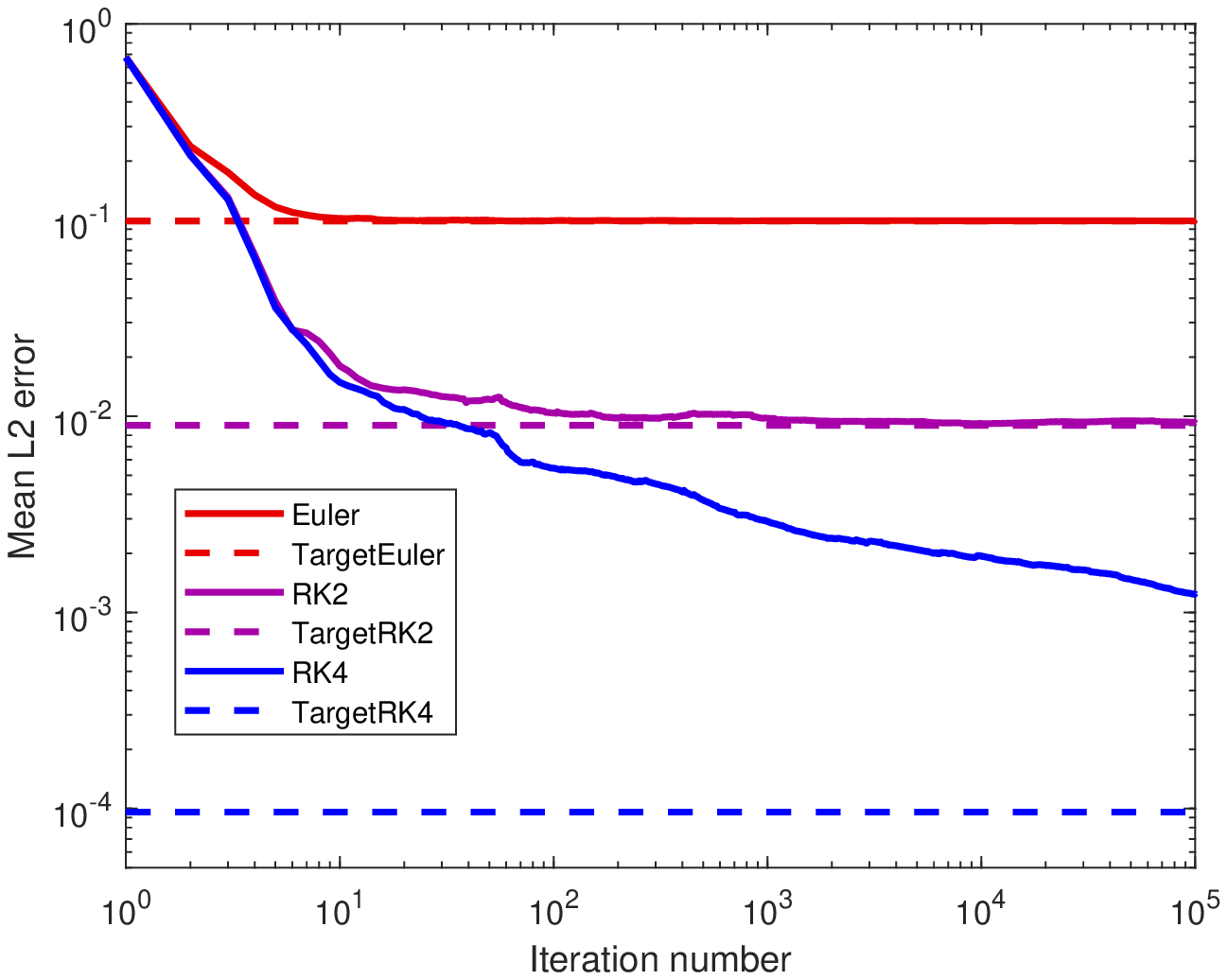}
%\caption{}
}
\caption{Linear and nonlinear ODE systems
} 
\label{intro:fig}
\end{figure}

The second goal is to investigate the accuracy of ResNet solver. For nonlinear ODE system dynamics learning, it turns out taking square error as the loss function and applying stochastic gradient descent method can train the ResNet network very well. To find out the effectiveness of the ResNet solver, we consider two errors. The first group of errors compute the mean $L_2$ errors between ResNet output and the reference solution, denoted as the solid curves (subfigure (b) and (d) for sample linear and nonlinear ODE systems) in Figure \ref{intro:fig}. The second group of errors calculate the mean $L_2$ errors between the training targets and the reference solution, represented as the dashed straight lines in Figure \ref{intro:fig}. We observe the solid curves merge into dashed lines over the number of iterations, which imply the convergence of ResNet learning process. 

In this paper, we consider three resources to generate the targets used in training the ResNet. Specifically We study first order forward Euler, second order Runge-Kutta2 and fourth order Runge-Kutta4 methods with mesh size $\Delta$ to generate the targets. The ResNet can be trained very well to learn the dynamics and the accuracy of ResNet ODE solver is dominated by the accuracy of the target used in training. If the ODE solutions are regular and smooth, we obtain the error orders of $O(\Delta^2)$, $O(\Delta^3)$ and $O(\Delta^5)$ for the ResNet solvers trained from first order forward Euler, second order Runge-Kutta2 and fourth order Runge-Kutta4 methods with mesh size $\Delta$, see the right two figures in Figure \ref{intro:fig} with $\Delta=0.1$. In a word, the ResNet ODE solver can be successfully trained and replicate the three finite difference methods.  

For decades neural networks have been used to model dynamical systems \cite{Chen-Billings-Grant1990,GonzalezGarcia-Martinez-Kevrekidis1998,MILANO2002}. Recent advances in deep learning improve the capability of neural network methods and attract new studies in the field. For example, we have recent works of  \cite{Pathak-Lu-Hunt-Girvan-Ott-2017,Vlachas-Byeon-Wan-Sapsis-Koumoutsakos-2018} exploring reccurrent neural networks from time series data to capture dynamic system chaotic behavior for forecasting; and  \cite{Mardt-Pasquali-Wu-Noe-2018,Yeung-Kundu-Hodas-2019} on nonlinear systems approximation with Koopman operator theory . 

Another group of articles that are closely related to ours focus on applying neural networks to recover the dynamic system (the expression of $\vF(\cdot)$ in (\ref{intro:eq-ODE})) from observed data \cite{raissi2018multistep,Chen-Rubanova-Bettencourt-Duvenaud-2018,Rudy-Kutz-Brunton-2019,sun-zhang-schaeffer2019,Reshniak-Webster2019,xie-zhang-webster-2019,keller-du-2020}, in which points values data on one trajectory $\vx(t)$ are collected and used in training. Classical numerical ODE solvers such as Runge-Kutta methods or Adam-version multistep methods are used in the definition of loss function to train the feed forward neural networks. In our work, learning data are the collection of short segments with time lag $\Delta$ passed and initial locations $\vx(t_0)$ randomly generated over a domain of interest. Our work is similar to \cite{Qin-Wu-Xiu-2019} by Qin. et. al. such that the neural network is applied to approximate the integral $\int_{t_0}^{t_0+\Delta} \vF(\vx(t)) ~dt$ instead of $\vF(\cdot)$ and the network parameters are trained from supervised learning.

Since the integral format of (\ref{intro:ODE-integral}) is the weak formulation of the ODE system (\ref{intro:eq-ODE}), the accuracy of the neural network solver depends how well the feed forward network approximates $\int_{t_0}^{t_0+\Delta} \vF(\vx(t)) ~dt$. It turns out the accuracy of the ResNet ODE solver is more related to the quality of the training target. Even the ResNet ODE solver looks like and is implemented as a first order forward Euler scheme, its accuracy can be as high as any order. Thus we can not simply treat ResNet as a discrete version of forward Euler scheme \cite{Lu-Zhong-Li-Dong-2017}.

The organization of the paper is as follows. In \S\ref{S:2}, we present the problem setup, notations and residual neural network as an ODE solver. In \S \ref{S:3}, we revisit the errors between training targets and exact solution to obtain error orders of $O(\Delta^2)$, $O(\Delta^3)$ and $O(\Delta^5)$ for targets generated from forward Euler, Runge-Kutta2 and Runge-Kutta4 methods. Sequence of numerical experiments are presented in \S \ref{S:4} to verify the observed behavior of ResNet solvers. Finally, we end in \S \ref{S:5} for conclusion.

%example for Table
%\begin{table}[h]
%\begin{tabular}{l l l}
%\hline
%\textbf{Treatments} & \textbf{Response 1} & \textbf{Response 2}\\
%\hline
%Treatment 1 & 0.0003262 & 0.562 \\
%Treatment 2 & 0.0015681 & 0.910 \\
%Treatment 3 & 0.0009271 & 0.296 \\
%\hline
%\end{tabular}
%\caption{Table caption}
%\end{table}

%example for figure
%\begin{figure}[h]
%\centering\includegraphics[width=0.4\linewidth]{placeholder}
%\caption{Figure caption}
%\end{figure}

%%%%%%%%%%%%%%%%%%%%%%%%%%%%%%%%%%%%%%%%%%%%%%%%%%%%%%%%%
\section{Problem setup, notations and residual neural network}
\label{S:2}

\subsection{Setup and notations}
\label{S:2.1}
\vspace{.1in}

We consider solving following ordinary differential equation (ODE) system initial value problem
\begin{equation}
    \frac{d \vx}{d t}=\vF(\vx(t)), ~~\vx(t_0)=\vx_0, \label{eq:ODE}
\end{equation}
where $\vx \in R^n$ are the state variables. 
We assume the right hand side vector function $\vF(\vx(t))$ is Lipschitz continuous, with 
\begin{equation}
\|\vF(\vx^1)-\vF(\vx^2)\|_2\leq L\|\vx^1-\vx^2\|_2.\label{eq:Lipschitz}
\end{equation}
Here $\|\cdot\|_2$ denotes the vector $L^2$ norm. 
Integrating the dynamic system (\ref{eq:ODE}) from $t_0$ to $t_0+\Delta$, where $\Delta$ is an increment or time lag, we obtain the following integral format of the ODE system
\be\label{eq:ODE-integral}
\vx(t_0+\Delta)-\vx(t_0)=\int_{t_0}^{t_0+\Delta} \vF(\vx(t)) ~dt.
\ee
In this paper, we apply residual neural network as an solver to approximate \eqref{eq:ODE-integral}.

The learning data are collected in the form of pairs. Each data pair refers to the solution states at two different time locations along one trajectory. The learning data set is given as
\begin{equation}\label{nn:data-pairs}
S =\left\{\left(\vy_{j}^1,\vy_{j}^2\right)\,|\, \vy_{j}^1\in D\right\}_{j=1}^{J},
\end{equation}
where $J$ denotes the total number of data pairs and $D$ is the domain of interest, from which the initial states of the learning data pairs are collected. For each pair, we have
\begin{equation} \label{nn:data-pairs-targets}
\vy_{j}^1= \vx_j(t_0), ~~~\vy_j^2\approx\vx_j(t_0+\Delta),\quad j=1, \dots, J.
\end{equation}
Again, we have $\Delta$ denote the time lag between the two states, and taken as a constant for all $j$ throughout this paper. Initial state of the learning data $\vy_j^1$ can be randomly generated from the domain of interest $D$, i.e., generated with uniform distribution. We have $\vy_j^2$ denoting the learning target, which is obtained from finite difference scheme approximation with time step size $\Delta$. In this article, we study three explicit methods of forward Euler, Runge-Kutta 2 and Runge-Kutta 4 to generate the target data.

%%%%%%%%%%%%%%%%%%%%%%%%%%%%%%%%%%%%%%%%%%%%%%%%%%%%%%%%%
%%%%%%%%%%%%%%%%%%%%%%%%%%%%%%%%%%%%%%%%%%%%%%%%%%%%%%%%%%

\subsection{One-step Residue Network (ResNet)}
\label{S:2.2}

%{\bf\emph{One-step ResNet as an ODE solver}}

\vspace{.1in}
In this section, we present the residual neural network (ResNet) as a one-step method solving the ordinary differential equation system (\ref{eq:ODE}). The abstract goal of machine learning is to find a function $\mathcal{N}: R^n \to R^n$ such that $\mathcal{N}(\cdot ; \Theta)$ accurately predicts the state $\vx(t_0+\Delta)$, given the initial state $\vx(t_0)$ and the time lag $\Delta$. The function $\mathcal{N}(\cdot ; \Theta)$ is parameterized by weight matrices and biases. The optimal parameter set $\Theta$ will be obtained by training the network intensively over the given data set.

In this paper, we consider a standard fully connected feedforward neural network  (FNN) \cite{Hornik-1991,Leshno-Lin-1993,Pinkus-1999} with $M$ layers. We have $M\geq 3$, since the input and output vectors of the network are considered as the first and last layer. Among the total $M$ layers, the interior $(M-2)$ are the hidden layers. Thus the minimum structure of the neural network involves one hidden layer with $M=3$. We refer to Figure \ref{fig:DNN_structure} for the structure of a standard feedForward neural network. Having $n_i$ ($i =1, \cdots, M$) denote the number of neurons in each layer. The first layer is the input vector with $n_1$ as its dimension and the last layer is the output vector with $n_M$ as its dimension. With the input and output vectors as either the state or the approximation of the state $\vx(\cdot)$ of the ODE system \eqref{eq:ODE}, we have $n_1=n_M=n$. Given the network input $\vp^{in}\in R^n$, the output of the FNN network is denoted as

\begin{equation*}
    \vp^{out} = \mathcal{N}(\vp^{in}; \Theta), \label{training_formula_1}
\end{equation*}
where $\Theta$ is the parameter set including all parameters in the network, i.e., weight matrices and biases connecting all layers.

\begin{figure}[h]
\centering
\subfigure[FeedForward neural network]{
\includegraphics[width=0.45\linewidth]{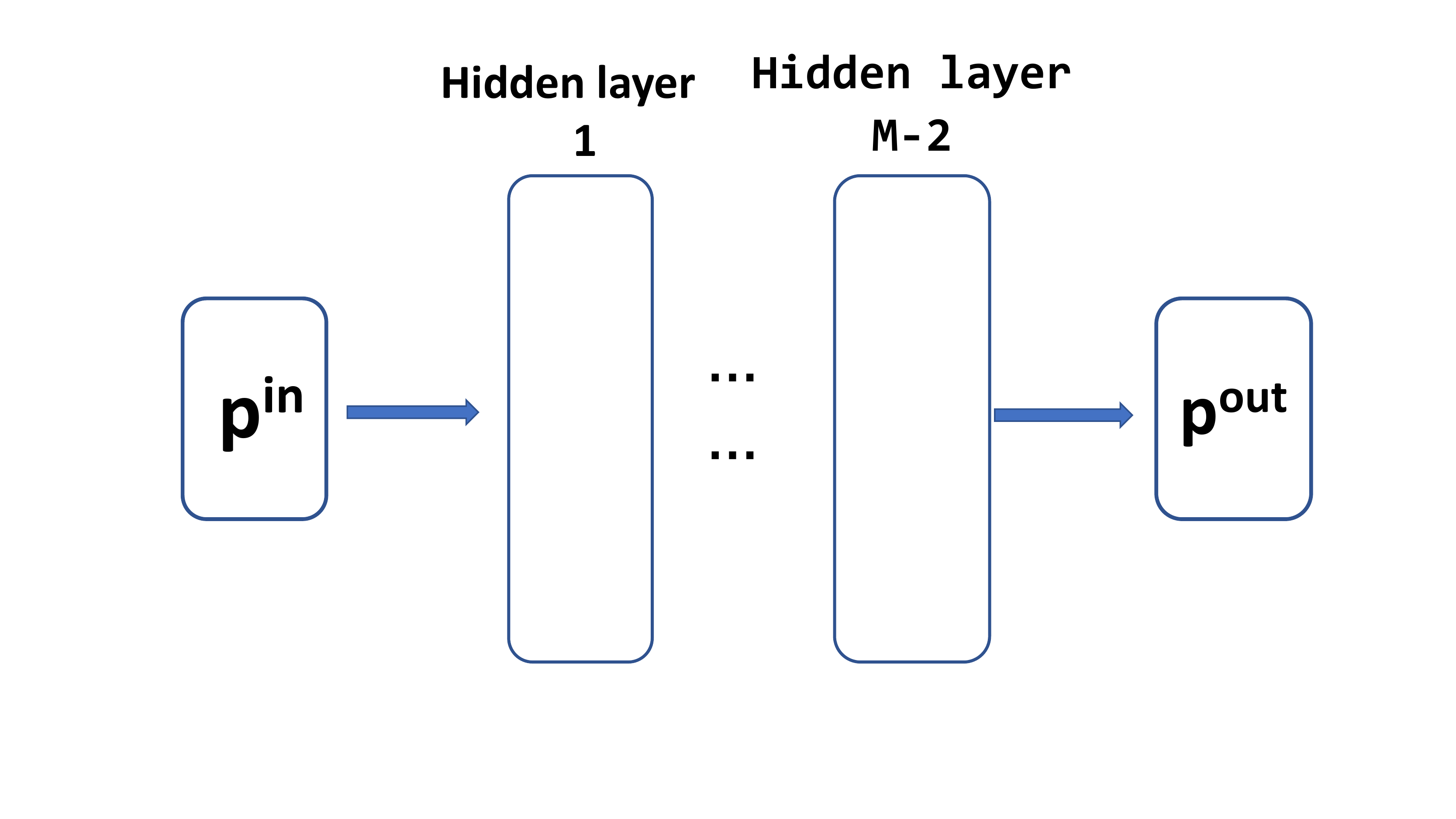}
}
\subfigure[Residual Network]{
\includegraphics[width=0.45\linewidth]{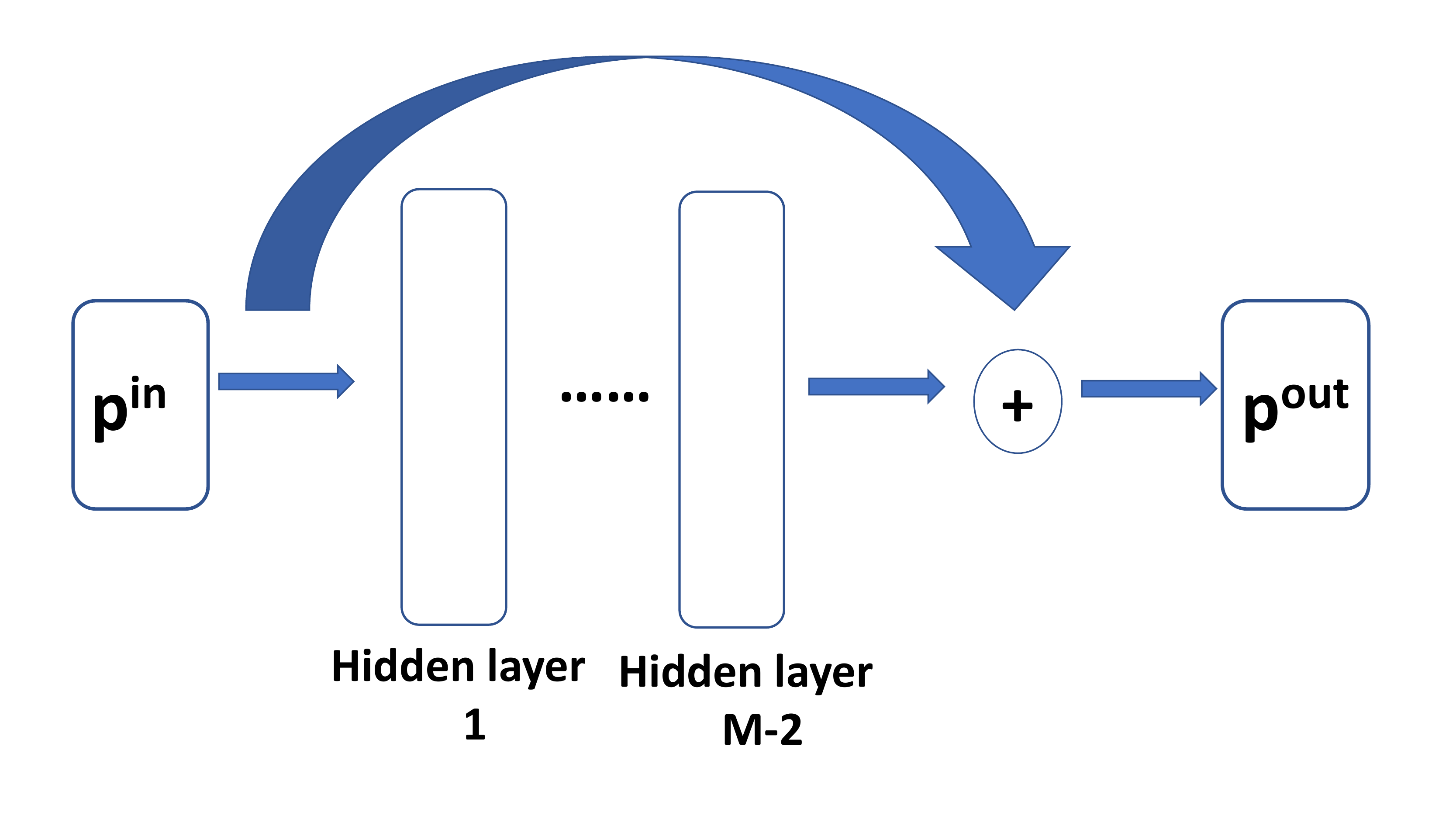}
%\caption{pic1}
}
\caption{Architecture of feedForward and residual neural network}
\label{fig:DNN_structure}
\end{figure}

The idea of residual neural network (ResNet) is to further introduce an identity operator and have the network effectively approximate the ``residue'' of the input-output map.
The structure of ResNet is also illustrated in Figure \ref{fig:DNN_structure} to compare with feedFoward neural network. Now the ResNet consists of $M$ fully connected layers and an identity operator that further "adds" the input $\vp^{in}$ vector into the output $\vp^{out}$. With the identity operator introduced, the resulting ResNet produces the following neural network
\begin{equation}
    \vp^{out} = \vp^{in}+\mathcal{N}(\vp^{in}; \Theta), %~~\vp^{in}=\vy^1_j 
\label{Resnet_formula_1}
\end{equation}
where $\Theta$ again denotes the collection of weight matrices and bias parameters.% Previously ResNet has been viewed as the Euler forward time integrator \cite{Chang_Meng_Haber_Tung_Begert-2018}. C

The core residual neural network $\mathcal{N}(\cdot; \Theta)$ still consists of $M$ layers. Each two consecutive layers is connected with an affine linear transformation and a point-wise nonlinear activation function.
%and aims at filtering the input features in a way that enables learning. 
The function or the mapping $\mathcal{N}(\cdot; \Theta)$ is a composition of following operators,
\begin{equation}
    \mathcal{N}(\cdot; \Theta)=(\sigma_M\circ \vW_{M-1})\circ \cdots \circ (\sigma_2\circ \vW_1). \label{neuralNetwork-structure}
\end{equation}
Here $\circ$ stands for operator composition. We have $\vW_i$ denoting the linear transformation operator or the weight matrix connecting the neurons from $i$-th layer to $(i+1)$-th layer. The parameter set is further augmented with the biases vectors. We have $\sigma_i: R\to R$ denoting the activation function ($i\geq 2$), which is applied to each neuron of the $i$-th layer in a component-wise fashion. There are many activation functions available and widely used, e.g., the sigmoid functions, the ReLU (rectified linear unit), etc. In this paper we apply ReLU function as the activation function. Specifically $\sigma_i=max(0,x)$ is applied between all layers, except to the output layer for which we have $\sigma_M(x) =x$. This is a common choice for deep learning neural network. Notice the activation function starts its application from the second layer $i=2$, since the input vector is regarded as the first layer.

The goal is to train the network to obtain an optimal parameter set $\Theta$, such that the ResNet can accurately approximate the $\Delta$-lag flow map $\vx(t)\to \vx(t+\Delta)$. This is achieved by applying \eqref{Resnet_formula_1}-\eqref{neuralNetwork-structure} with $\vp^{in}_j=\vy_j^1$ to obtain the network output $\vp_j^{out}$,  comparing with the target $\vy_j^2$, and then looping among the data set $S$ to minimize the error or the squared loss function
\begin{equation}
    L(\Theta)=\|\vp_j^{out}-\vy_j^2\|_2^2, \quad \forall j =1,\cdots,J. \label{loss_function1}
\end{equation}
The notation of \eqref{Resnet_formula_1} is abused here. We have $\vp_j^{out}$ referring to the $j$-th output corresponding to its input $\vp^{in}_j=\vy_j^1$, where $j$ is the index of the data pair in the learning data set $S$ of \eqref{nn:data-pairs}. Usually the mean squared error 
\begin{equation*}
    L(\Theta)=\frac{1}{J}\sum_{j=1}^{J} \|\vp_j^{out}-\vy_j^2\|_2^2, \label{loss_function2}
\end{equation*}
is applied as the loss function, referred as the  gradient descent method. To be more efficient, we apply \eqref{loss_function1} instead, which is referred as the stochastic or approximate gradient descent method. Finally, we minimize the error \eqref{loss_function1} iteratively to obtain the optimal parameter set $\Theta$, corresponding to a given tolerance $\epsilon$
\be\label{theta-iteration}
\Theta_{i}\longrightarrow \Theta_{i+1},\quad i=1,\cdots, K.
\ee
Here $K$ is the total number of iterations involved in the training. Again, we go through the whole data set of \eqref{loss_function2} over each iteration to update the parameter set. 

With the optimal parameter set $\Theta^{*}$ determined and the well trained neural network $\mathcal{N}(\cdot; \Theta^{*})$ available, the ResNet \eqref{Resnet_formula_1} can be viewed as an accurate ODE system \eqref{eq:ODE} solver with time step size $\Delta$. In a word, we can apply $\vp^0=\vx(t_0)$, and repeatedly call the ResNet \eqref{Resnet_formula_1} to generate a discrete set of points values approximating the ODE system
\begin{equation}
    \vp^{k+1} = \vp^{k}+\mathcal{N}(\vp^{k}; \Theta^{*}), \quad\vp^{0}=\vx(t_0). \label{Resnet_formula_2}
\end{equation}
The ResNet neural network \eqref{Resnet_formula_2} can be considered as an ODE solver to approximate the solution trajectory, which is similar to an explicit finite difference scheme. Comparing the one-step ResNet \eqref{Resnet_formula_1} and the integral format of the ODE system \eqref{eq:ODE-integral}, it is easy to see that a well trained network operator $\mathcal{N}(\cdot; \Theta^{*})$ is an approximation to the effective increment
\begin{equation}
    \mathcal{N}(\vx(t_0); \Theta^{*})\approx\int_{t_0}^{t_0+\Delta} \vF(\vx(t)) ~dt. \label{N_approximation}
\end{equation}

The initial states in the learning data set $S$ are randomly generated from the domain of interest. There is no error for the ResNet input $\vp_{j}^{in}=\vy_{j}^1$,  since we have $\vy_{j}^1= \vx_j(t_0)$ with $\vx_j(t_0)$ being the initial condition of the dynamic system \eqref{eq:ODE-integral}, where $j$ is the integer index of the learning data set $S$. Thus the error of the ResNet ODE solver mainly comes from the target $\vy_{j}^2$, see \eqref{loss_function1}. 
In the following section, we carry out fundamental analysis and show the accuracy of the ResNet depends on the accuracy or the quality of the target from the data pair \eqref{nn:data-pairs}.

%%%%%%%%%%%%%%%%%%%%%%%%%%%%%%%%%%%%%%%%%%%%%%%%%%%%%%%%%%
%%%%%%%%%%%%%%%%%%%%%%%%%%%%%%%%%%%%%%%%%%%%%%%%%%%%%%%%%
\section{Error estimate and objective of numerical studies}
\label{S:3}
In this section, we first estimate the error between the target $\vy_{j}^2$ and the exact solution $\vx_j(t_0+\Delta)$ or the state after time lag $\Delta$ passed, which turns out to be the dominant error for the ResNet solver. Then we describe the objective of numerical experiments in section \ref{S:4}.

\subsection{Error estimate}
\label{S:3.1}
\vspace{.1in}

The target $\vy_{j}^2$ is obtained from an explicit one-step method such as  Runge-Kutta methods. To better illustrate the error behavior of the target, which is obtained from a multi-stage method, we write out the none autonomous format of the ODE system \eqref{eq:ODE}. In this section, we include the explicit dependence of $\vF(\vx(t),t)$ on $t$ and consider following general format of the ODE. {\emph{For convenience of analysis, none autonomous format is adapted in the section}}.
$$
\frac{d \vx}{d t}=\vF(\vx(t),t), ~~\vx(t_0)=\vx_0.
$$
The integral version of the corresponding ODE system is rewritten as 
\be
\vx_j(t_0+\Delta)-\vx_j(t_0)=\int_{t_0}^{t_0+\Delta} \vF(\vx_j(t),t) ~dt.\label{eq:ODE-integral-sec3}
\ee
The relation between training data pair $(\vy_{j}^1,\vy_{j}^2)$ of \eqref{nn:data-pairs} and the solution of \eqref{eq:ODE-integral-sec3} is
$$
\vy_{j}^1= \vx_j(t_0), ~~~\vy_j^2\approx\vx_j(t_0+\Delta),\quad j=1, \dots, J.
$$

In this paper, we investigate the error performance of the ResNet solver \eqref{Resnet_formula_1} based on the targets $\vy_{j}^2$ obtained from following forward Euler, Runge-Kutta2, and Runge-Kutta4 methods.

\begin{enumerate}
    \item First order forward Euler method 
    \begin{equation}
    \vy_j^2=\vy_j^1+ \Delta \times \vF(\vy_j^1, t_0). \label{scheme:forward_Euler}
\end{equation}

\item Second order Runge-Kutta2 method
\begin{align}
k_1 &=\vF(\vy_j^1,t_0); \quad\quad k_2 =\vF(\vy_j^1+\Delta k_1, t_0+\Delta);\nonumber\\
\vy_j^2 &= \vy_j^1+\Delta\times \left(\frac{1}{2}k_1+\frac{1}{2}k_2\right).
\label{scheme:RK_2}
\end{align}

\item Fourth order Runge-Kutta4 method

\begin{align}
    k_1 &=\vF(\vy_j^1,t_0); \quad \quad
    k_2 =\vF\left(\vy_j^1+\frac{\Delta}{3}k_1,t_0+\frac{\Delta}{3}\right);\nonumber\\
    k_3 &=\vF\left(\vy_j^1-\frac{\Delta}{3} k_1+\Delta k_2, t_0+\frac{2\Delta}{3}\right);\quad 
    k_4=\vF\left(\vy_j^1+\Delta k_1-\Delta k_2+\Delta k_3 ,t_0+\Delta\right)\nonumber\\
    \vy_j^2 &= \vy_j^1+\frac{\Delta}{8}\times \left( k_1+3k_2+3k_3+k_4\right). 
    \label{scheme:RK_4-2}
\end{align}

\end{enumerate}

\begin{lemma}
Given the ResNet training target $\vy_j^2$ of \eqref{nn:data-pairs} obtained from a one-step method with $k$-th order local truncation error, the error between the target and the exact solution $\textbf{x}_j(t_0+\Delta)$ of the ODE system \eqref{eq:ODE-integral} is of $(k+1)$th order
\begin{equation}
  \|\vy_j^{2}-\vx_j(t_0+\Delta)\|_2 \le C\Delta^{k+1},
\end{equation}
where constant $C = C(\vF)$ depends on the Lipschitz constant of $\vF$ and the regularity of the solution.
%$\textbf{x}_j(t_0)$.   
\end{lemma}

\begin{proof}

Suppose $\vy_j^2$ is obtained from a one-step method 
\be\label{lemma-one-step-method}
\vy_j^2=\vy_j^1+\Delta\times \Phi\left(t_0,\vy_j^1,\vF(\vy_j^1),\Delta\right),
\ee
with $\vy_j^1=\textbf{x}_j(t_0)$ and step size of $\Delta$. Denote $\tau_j(\Delta)=O(\Delta^k)$ as the local truncation error of the one-step method (\ref{lemma-one-step-method}). We have
\be\label{lemma-exact}
\textbf{x}_j(t_0+\Delta)=\textbf{x}_j(t_0)+\Delta\times \Phi\left(t_0,\textbf{x}_j(t_0),\vF(\textbf{x}_j(t_0)),\Delta\right)+\Delta\times\tau_j(\Delta).
\ee
Subtract \eqref{lemma-one-step-method} from \eqref{lemma-exact} and apply the fact of $\vy_j^1=\textbf{x}_j(t_0)$, we have
$$
\|\textbf{x}_j(t_0+\Delta)-\vy_j^{2}\|_2
\leq \Delta\|\Phi\left(t_0,\vy_j^1,\vF(\vy_j^1),\Delta\right)-\Phi\left(t_0,\textbf{x}_j(t_0),\vF(\textbf{x}_j(t_0)),\Delta\right)\|_2 +\Delta\times\tau_j(\Delta)\leq C \Delta^{k+1},
$$
where $C=C(\vF)$ depends on the regularity of the system.
\end{proof}

\vspace{.05in}
\begin{rem}
The target $\vy_j^{2}$ is obtained with one step and with step size of $\Delta$ ONLY, thus its error to the exact solution is one order higher than the $k$-th order. 
\end{rem}
 
Next, we apply the idea of interpolation polynomial and its approximation to integration to find out the errors between the target and the exact solution. Specifically we obtain the orders of errors of $O(\Delta^2)$, $O(\Delta^3)$ and $O(\Delta^5)$ for the first order forward Euler method \eqref{scheme:forward_Euler}, second order Runge-Kutta2 method \eqref{scheme:RK_2} and fourth order Runge-Kutta4 method \eqref{scheme:RK_4-2} respectively. 

\vspace{.1in}
\noindent
{\bf Case I:} \textbf{$\vy_j^2$ obtained from Forward Euler method \eqref{scheme:forward_Euler}}
\vspace{.1in}

Given $\vy_j^1=\vx_j(t_0)$, subtract the exact solution $\vx_j(t_0+\Delta)$ of \eqref{eq:ODE-integral-sec3} from $\vy_j^2$ of the forward Euler method \eqref{scheme:forward_Euler}, we have 

\begin{align}
    \|\vy_j^2-\vx_j(t_0+\Delta)\|_2
    &= \left\|\vy_j^1+ \Delta \times \vF(\vy_j^1, t_0)-\left(\vx_j(t_0)+\int_{t_0}^{t_0+\Delta} \vF(\vx_j(t),t) ~dt\right)\right\|_2 \nonumber\\
    &=\left\|\int_{t_0}^{t_0+\Delta}  \left(\vF(\vx_j(t_0), t_0)- \vF(\vx_j(t),t)\right) ~dt\right\|_2 \nonumber\\
    &= \left\|\int_{t_0}^{t_0+\Delta}  \frac{d}{dt}\vF(\vx_j(\xi(t)),\xi(t)) (t-t_0) ~dt\right\|_2 \nonumber\\
    &=\frac{\Delta^2}{2}\left\|\frac{d}{dt}\vF(\vx_j(\eta),\eta) \right\|_2\leq C\Delta^2. \label{error_for_Forward_Euler}
\end{align}
Here $\frac{d}{dt}\vF(\vx(t),t)=\frac{\partial \vF}{\partial \vx}\vF + \frac{\partial \vF}{\partial t}$ refers to the complete derivative to the $t$ variable, with $\frac{\partial \vF}{\partial \vx}$ denoting the Jacobian matrix of the vector function $\vF$ on variable $\vx(t)$ and $\frac{d \vx}{dt}=\vF$. Forward Euler method can be considered as a constant quadrature rule approximation to the integral of the ODE system \eqref{eq:ODE-integral-sec3}. Weighted mean value theorem is applied to estimate the error term.

\vspace{.1in}     
\noindent
{\bf Case II:}      
\textbf{$\vy_j^2$ obtained from 2nd order Runge-Kutta method \eqref{scheme:RK_2}}
\vspace{.1in}   

Again we have $\vy_j^1=\vx_j(t_0)$. Subtract $\vx_j(t_0+\Delta)$ of \eqref{eq:ODE-integral-sec3} from $\vy_j^2$ of the second order Runge-Kutta method \eqref{scheme:RK_2}, we have 

\begin{align*}
   \|\vy_j^2-\vx_j(t_0+\Delta)\|_2
    &= \left\|\vy_j^1+\Delta\times\left(\frac{k_1+k_2}{2}\right)-\left(\vx_j(t_0)+\int_{t_0}^{t_0+\Delta} \vF(\vx_j(t),t) ~dt\right)\right\|_2 \nonumber\\
    &=\left\|\Delta\times\left(\frac{k_1+k_2}{2}\right)-\int_{t_0}^{t_0+\Delta} \vF(\vx_j(t),t) ~dt\right\|_2 \nonumber\\
    &\leq \left\|\Delta\times\left(\frac{k_1+\widetilde{k_2}}{2}\right)-\int_{t_0}^{t_0+\Delta} \vF(\vx_j(t),t) ~dt\right\|_2 +\frac{\Delta}{2}\|k_2-\widetilde{k_2}\|_2,
\end{align*}
where $k_2=\vF(\vy_j^1+\Delta k_1, t_0+\Delta)$, $k_1=\vF(\vy_j^1,t_0)$ and $\widetilde{k_2}=\vF(\vx_j(t_0+\Delta), t_0+\Delta)$. With the $O(\Delta)$ local truncation error of the forward Euler method approximating $\vx_j(t_0+\Delta)$ and applying the Lipschitz continuity of $\vF$ of the dynamic system, we have
\begin{align*}
  \|k_2-\widetilde{k_2}\|_2
    &= \left\|\vF(\vy_j^1+\Delta k_1, t_0+\Delta)-\vF(\vx_j(t_0+\Delta), t_0+\Delta)\right\|_2 \nonumber\\
    &\leq L\|\vy_j^1+\Delta k_1-\vx_j(t_0+\Delta)\|_2 
    \leq C\Delta^2. 
\end{align*}
Here $C$ represents a generic constant. The error from the two-points quadrature rule can be estimated as
\begin{align*}
    &  \left\|\Delta\times\left(\frac{k_1+\widetilde{k_2}}{2}\right)-\int_{t_0}^{t_0+\Delta} \vF(\vx_j(t),t) ~dt\right\|_2 \nonumber\\
    =&\left\|\int_{t_0}^{t_0+\Delta}  (\textbf{G}_1(t)- \vF(\vx_j(t),t)) ~dt\right\|_2 \nonumber\\
   = &\frac{1}{2}\left\|\frac{d^2}{dt^2}\vF(\vx_j(\eta),\eta)\right\|_2\left|\int_{t_0}^{t_0+\Delta} (t-t_0)\left(t-(t_0+\Delta)\right) ~dt\right|\nonumber\\ =&\frac{\Delta^3}{12}\left\|\frac{d^2}{dt^2}\vF(\vx_j(\eta),\eta)\right\|_2 \leq C\Delta^3. \label{error_for_RK2}
\end{align*}
Combine the above arguments, we have
\be
\|\vy_j^2-\vx_j(t_0+\Delta)\|_2\leq C\Delta^3.
\label{error_for_RK2}
\ee
Here $\textbf{G}_1(t)$ denotes the linear interpolation polynomial that interpolates $\vF(\vx(t),t)$ at $t_0$ and $t_0+\Delta$. And $\frac{d^2}{dt^2}\vF(\vx(\cdot),\cdot)$ denotes the complete second derivative of $\vF(\vx(t),t)$ to $t$. This 2-stage Runge-Kutta method can be considered as a trapezoidal quadrature rule approximating the integration.

\vspace{.1in}     
\noindent
{\bf Case III:}  
\textbf{$\vy_j^2$ obtained from 4th order Runge-Kutta method \eqref{scheme:RK_4-2}}
\vspace{.1in}  

With $\vy_j^1=\vx_j(t_0)$ and subtract $\vx_j(t_0+\Delta)$ of \eqref{eq:ODE-integral-sec3} from $\vy_j^2$ of the fourth order Runge-Kutta method \eqref{scheme:RK_4-2}, we have 
\begin{align*}
& \, \|\vy_j^2-\vx_j(t_0+\Delta)\|_2\nonumber\\
%=&\left\|\vy_j^1+\frac{\Delta(k_1+3k_2+3k_3+k_4)}{8}-\left(\vx_j(t_0)+\int_{t_0}^{t_0+\Delta} \vF(\vx(t),t) ~dt\right)\right\|_2 \nonumber\\
=&\left\|\frac{\Delta(k_1+3k_2+3k_3+k_4)}{8}-\int_{t_0}^{t_0+\Delta} \vF(\vx_j(t),t) ~dt\right\|_2 \nonumber\\
\leq& \left\|\frac{\Delta(k_1+3\widetilde{k_2}+3\widetilde{k_3}+\widetilde{k_4)}}{8}-\int_{t_0}^{t_0+\Delta} \vF(\vx_j(t),t) ~dt\right\|_2 +\left\|\frac{\Delta(k_1+3k_2+3k_3+k_4)}{8}-\frac{\Delta(k_1+3\widetilde{k_2}+3\widetilde{k_3}+\widetilde{k_4)}}{8}\right\|_2.
\end{align*}
Terms of $k_2, k_3$ and $k_4$ are from the Runge-Kutta4 method \eqref{scheme:RK_4-2}, with $k_1=\vF(\vy_j^1,t_0)=\vF(\vx_j(t_0),t_0)$. We have  $\widetilde{k_2}=\vF(\vx_j(t_0+\frac{\Delta}{3}),t_0+\frac{\Delta}{3})$, $\widetilde{k_3}=\vF(\vx_j(t_0+\frac{2\Delta}{3}),t_0+\frac{2\Delta}{3})$ and 
$\widetilde{k_4}=\vF(\vx_j(t_0+\Delta),t_0+\Delta)$ introduced that $k_2, k_3$ and $k_4$ approximate. Rewrite the Runge-Kutta4 method of \eqref{scheme:RK_4-2} as a one-step method, $\vy_j^2=\vy_j^1+\Delta \Phi\left(t_0,\vy_j^1,\vF(\vy_j^1),\Delta\right)$, we have
\begin{align*}
&\left\|\frac{\Delta(k_1+3k_2+3k_3+k_4)}{8}-\frac{\Delta(k_1+3\widetilde{k_2}+3\widetilde{k_3}+\widetilde{k_4)}}{8}\right\|_2 \nonumber\\
= & \Delta\left\|\Phi\left(t_0,\vy_j^1,\vF(\vy_j^1),\Delta\right)-\Phi\left(t_0,\textbf{x}_j(t_0),\vF(\textbf{x}_j(t_0)),\Delta\right)\right\|_2
\leq  C\Delta^5. 
\end{align*}
Here $C$ represents a generic constant. The error from the four-points quadrature rule can be estimated as
\begin{align*}
& \left\|\Delta\times\left(\frac{k_1+3\widetilde{k_2}+3\widetilde{k_3}+\widetilde{k_4}}{8}\right)-\int_{t_0}^{t_0+\Delta} \vF(\vx_j(t),t) ~dt\right\|_2 \nonumber\\
= &\left\|\int_{t_0}^{t_0+\Delta}  (\textbf{G}_3(t)- \vF(\vx_j(t),t)) ~dt\right\|_2 \nonumber\\
\leq &C\left\|\frac{d^4 \vF}{dt^4}\right\|_2 \left|\int_{t_0}^{t_0+\Delta} (t-t_0)(t-(t_0+\frac{\Delta}{3}))(t-(t_0+\frac{2\Delta}{3}))\left(t-(t_0+\Delta)\right) ~dt\right| \leq  C \Delta^5. 
\end{align*}
Again $C$ represents a generic constant. Summarize the above arguments, we have
\be
\|\vy_j^2-\vx_j(t_0+\Delta)\|_2\leq C\Delta^5.
\label{error_for_RK4}
\ee
Here $\textbf{G}_3(t)$ denotes the cubic interpolation polynomial that interpolates $\vF(\vx(t),t)$ at $t_0$, $t_0+\Delta/3$, $t_0+2\Delta/3$ and $t_0+\Delta$. And $\frac{d^4\vF}{dt^4}$ denotes the complete fourth derivative of $\vF(\vx(t),t)$ to $t$ variable at somewhere. This version of 4-stage Runge-Kutta method can be considered as the three-eighth Simpson quadrature rule approximating the integration.

\begin{theorem}\label{thm_1}
Have $\textbf{x}_j(t_0+\Delta)$ denote the solution of the ODE system \eqref{eq:ODE-integral} with initial $\textbf{x}_j(t_0)$. Suppose the parameter set $\Theta$ of the ResNet \eqref{Resnet_formula_1} is well trained such that the error between the output $\vp_j^{out}$ and its target $\vy_j^2$ is on machine round-off error level. Given the % ResNet input taken as $\vp_j^{in}=\textbf{x}_j(t_0)$ and
 target $\vy_j^2$ obtained from a $k$-th order one-step method, we have 
\begin{equation}\label{ResNet-error}
  \|\vp_j^{out}-\vx_j(t_0+\Delta)\|_2 \le C\Delta^{k+1},
\end{equation}
where constant $C = C(\vF)$ depends on the regularity of the ODE system.
\end{theorem}

\begin{proof}
Apply triangle inequality and we have
\begin{align*}
    \|\vp_j^{out}-\vx_j(t_0+\Delta)\|_2
    &=\|\vp_j^{out}-\vy_j^2+\vy_j^2-\vx_j(t_0+\Delta)\|_2\nonumber\\
    &\le \|\vp_j^{out}-\vy_j^2\|_2+\|\vy_j^2-\vx_j(t_0+\Delta)\|_2 \nonumber\\
    &=\mbox{Training error}+\mbox{Target error}. \nonumber\\
    &\leq C \Delta^{k+1}.
 %   \label{error_formula_1}
\end{align*}
\end{proof}

The ResNet \eqref{Resnet_formula_1} can be trained very well to obtain an {\emph{optimal}} parameter set of weight matrices and biases. Numerical experiments show the training errors of $\|\vp_j^{out}-\vy_j^2\|_2$ are small or trained to the tolerance error level and can be ignored when comparing to the target error of $\|\vy_j^2-\vx_j(t_0+\Delta)\|_2 $. Tests show ResNet neural network does {\emph{a very good job}} approximating the target $\vy_j^2$ with stochastic gradient descent method and can be used as an accurate solver for the ODE system.

\subsection{Objective and implementation setup}
\label{S:3.2}
\vspace{.1in}

Now we describe the objective of studies for the rest of the article mainly through numerical tests. For each ODE system or in each numerical example, we conduct following three studies:
\begin{enumerate}
    \item Neural network architecture study in terms of numbers of layers and neurons;
    \item Training targets study in terms of error behavior;
    \item Solution trajectory approximation.
\end{enumerate}

In the first architecture study, we vary the number of layers and neurons per layer and observe how the rectangular arrangements of neurons affects the performance of the network. The specific arrangement of neurons and layers is called the \textit{architecture} of the network. For each neural network with an architecture setting, we train the network with a highly accurate training target $\vz_i^2$ generated from a refined mesh Runge-Kutta4 method (mesh size $h=\Delta/1000$), and will be trained long enough for the error curve to be almost flat. Once the training is finished for this network, we then use a test set of $J$ pairs of data $(\vz_i^1, \vz_i^2)^J_{i=1}$ with $\vp_i^{in}=\vz^1_i$ as the ResNet input and compare the ResNet output $\vp_i^{out}$ to the reference solution $\vz_i^2$ to calculate the error. The following two error norms are computed to measure the worst case and the average case of errors among the total $J$ pairs of data.
\begin{equation}\label{error-Max-Linfy}
Max (L_{\infty})=\max_{1\leq i\leq J} \|\vp_i^{out}-\vz_i^2\|_{\infty}, 
\end{equation}
\begin{equation}\label{error-Mean-L2}
Mean(L_2)= \frac{1}{J} \sum_{i=1}^{J} \|\vp_i^{out} -\vz_i^2\|_2.
\end{equation}

To avoid the influence from the initialization of weights and biases, which are generated from a normal distribution around zero, 
we apply a total ten runs of random initialization of the parameter set and average the output errors of \eqref{error-Max-Linfy} and \eqref{error-Mean-L2} as the performance measurement of that specific ResNet. Notice that same learning data set and same number of iterations (i.e. $K=500$ of \eqref{theta-iteration}) as the stopping condition are applied for each neural network architecture study. Furthermore, we do not differentiate the learning data set and the test set, since no difference is observed. We simply apply the learning data set to calculate the errors of \eqref{error-Max-Linfy} and \eqref{error-Mean-L2}.  

%The learning method used is a stochastic gradient descent method, with a mini-batch size of 1. One iterations represents one complete run through the training set. While more complicated learning methods exist, we believe our results will generalize to different learning algorithms.

The goal of the second study is to verify the result of Theorem \ref{thm_1} that the ResNet solver error  $\|\vp_j^{out}-\vx_j(t_0+\Delta)\|_2$ is dominated by the target error  $\|\vy_j^2-\vx_j(t_0+\Delta)\|_2$. Following the architecture study, we pick {\emph{one}} efficient and accurate ResNet network architecture on which we run the subsequent tests. A deeper neural network with many hidden layers has many more parameters than a simpler network, so is computationally more expensive to train. The computational cost and the accuracy measured by \eqref{error-Max-Linfy} and \eqref{error-Mean-L2} together are considered as metrics by which we select the {\emph{one}} ResNet architecture that will be used for the target study and solution trajectory simulation. 

In the second \textit{target} study, we adopt three different target data $\vy_j^2$ as learning data sets to obtain three independent ResNet networks. We then evaluate the performance of the three networks separately. This study involves the target $\vy_j^2$ computed from the forward Euler method \eqref{scheme:forward_Euler}, Runge-Kutta2 method \eqref{scheme:RK_2} and Runge-Kutta4 method \eqref{scheme:RK_4-2} with mesh size $\Delta$. We measure the following mean $L_2$ error between the target and the reference solution $\vz_i^2$
\begin{equation}\label{error-target}
\mbox{Target Mean} (L_2)\, \mbox{error}= \frac{1}{J} \sum_{i=1}^{J} \|\vy_i^2 -\vz_i^2\|_2.
\end{equation}
For each target generated from one of the three finite difference methods, we also calculate the mean $L_2$ error of \eqref{error-Mean-L2} between the ResNet output and the reference solution after every iteration. Numerical tests show that the ResNet output error curve converges to the target error as the number of iterations increases. In a word, the ResNet solver error $\|\vp_j^{out}-\vx_j(t_0+\Delta)\|_2$ is dominated by its target error $\|\vy_j^2-\vx_j(t_0+\Delta)\|_2$. Neural network can successfully learn and replicate the three finite difference methods.  

Finally in the solution trajectory simulation, we call the ResNet solvers repeatedly to generate a discrete set of points approximating the continuous reference solution curve as a standard one step finite difference method.
All three ResNet solvers trained from forward Euler method, Runge-Kutta2 method and Runge-Kutta4 method are applied and compared to the reference solution to further check the accuracy and capability of the neural network solvers. 

Throughout all numerical examples, constant time lag of $\Delta=0.1$ or smaller, due to stability restriction of the finite difference methods, is applied. For the learning data set, a total number of $J=500$ data pairs are applied to the linear ODE systems (Example 1) and a total number of $J=2000$ data pairs are applied to nonlinear ODE systems (all other examples).

%%%%%%%%%%%%%%%%%%%%%%%%%%%%%%%%%%%%%%%%%%%%%%%%%%%%%%
%%%%%%%%%%%%%%%%%%%%%%%%%%%%%%%%%%%%%%%%%%%%%%%%%%%%%%%%%
\section{Numerical experiments}
\label{S:4}

In this section, we consider a sequence of ordinary differential equation systems and carry out the objective of architecture study, target study and solution curve simulation listed in section \ref{S:3.2}. 
%On each system, we run a series of empirical tests to verify the predicted behavior of ResNet (\ref{Resnet_formula_1}) as an ODE solver. 
%We also include a dashed line, which represents the mean $L_2$ error between the fine-mesh Runge-Kutta-4 outputs and the outputs from each of the less-accurate methods. This correlates to $I_2$ from (\ref{error_for_Forward_Euler}),(\ref{error_for_RK2}), and (\ref{error_for_RK4}), and represents the theoretical limit for how well the network could possibly perform.

%could specifically reference the equations from section 3, which explain the specific methods we are using.

%%%%%%%%%%%%%%%%%%%%%%%%%%%%%%%%%%%%%%%%%%%%%%%%%%
\begin{example}\label{ex1} {\bf \emph {Linear ODE systems}}
\end{example}

In this example we consider six linear ODE systems of the form
\be\label{eq:ex1}
\dot{\vx}=\textbf{A}\vx+\textbf{b}, 
\ee
where $\textbf{A}\in R^{2\times 2}$ and $\textbf{b}\in R^2$, see \cite{Wu-Xiu-2018}. The system coefficients matrices $\textbf{A}$, the non-homogeneous vectors $\textbf{b}$ and domains of interest from which learning data are drawn are outlined in Table \ref{tab:linear ODE}. 
\begin{table}[htbp]
    \centering
    \begin{tabular}{l l l}
    \toprule 
         & Linear ODE system & domain of interest  \\
         \midrule
         Saddle point& $A=\left[\begin{matrix} 
                                1 & 1\\
                                1 & -1
                                \end{matrix}\right], \,\,\,\,b=\left[\begin{matrix} 
                                -2 \\
                                0
                                \end{matrix}\right]$ & $D=[0, 2]\times[0, 2]$\\
         \midrule  
         Nodal sink & $A=\left[\begin{matrix} 
                                -2 & 1\\
                                1 & -2
                                \end{matrix}\right], b=\left[\begin{matrix} 
                                -2 \\
                                1
                                \end{matrix}\right]$ & $D=[-2, 0]\times[-1, 1]$ \\
         \midrule
         Improper node & $A=\left[\begin{matrix} 
                                1 & -4\\
                                4 & -7
                                \end{matrix}\right],\quad b=\left[\begin{matrix} 
                                0 \\
                                0
                                \end{matrix}\right]$ & $D=[-1, 1]\times[-1, 1]$ \\
         \midrule 
         Star point & $A=\left[\begin{matrix} 
                                -1 & 0\\
                                0 & -1
                                \end{matrix}\right],\, b=\left[\begin{matrix} 
                                0 \\
                                0
                                \end{matrix}\right]$ & $D=[-1, 1]\times[-1, 1]$ \\
         \midrule
         Center point &$A=\left[\begin{matrix} 
                                 1 & 2\\
                                -5 & -1
                                \end{matrix}\right],\, b=\left[\begin{matrix} 
                                0 \\
                                0
                                \end{matrix}\right]$ & $D=[-1, 1]\times[-1, 1]$ \\
         \midrule
         Spiral point &$A=\left[\begin{matrix} 
                                -1 & -1\\
                                2 & -1
                                \end{matrix}\right], b=\left[\begin{matrix} 
                                -1 \\
                                5
                                \end{matrix}\right]$ & $D=\{\vx|(x_1+2)^2+(x_2-1)^2\le 1\}$ \\
         \midrule
         \bottomrule
    \end{tabular}
    \caption{Example \ref{ex1} linear ODE systems coefficients matrices.}
    \label{tab:linear ODE}
\end{table}

We apply a training set of $J=500$ data pairs with a time lag of $\Delta = 0.1$ in this example, and each network is trained for a total of $K=500$ iterations. To save space, we only present the results of the Saddle point and Nodal sink systems, with the other four linear ODE systems behaving similarly. In Figure \ref{fig:ex1-architecture} we list the architecture studies of the two systems, with architectures running between 1 to 4 hidden layers and 2 to 10 neurons per layer. We find that 1 hidden layer has a high degree of accuracy and efficiency for all the linear ODE systems. For the saddle point system, one hidden layer of two neurons is the optimal architecture, while for the nodal sink system, one hidden layer of six neurons is the optimal. We now use these two ResNet networks for the target study and the solution curve approximations.

\begin{figure}[htbp]
\centering
\subfigure[Saddle point $L_{\infty}$ error]{
\includegraphics[width=0.31\linewidth]{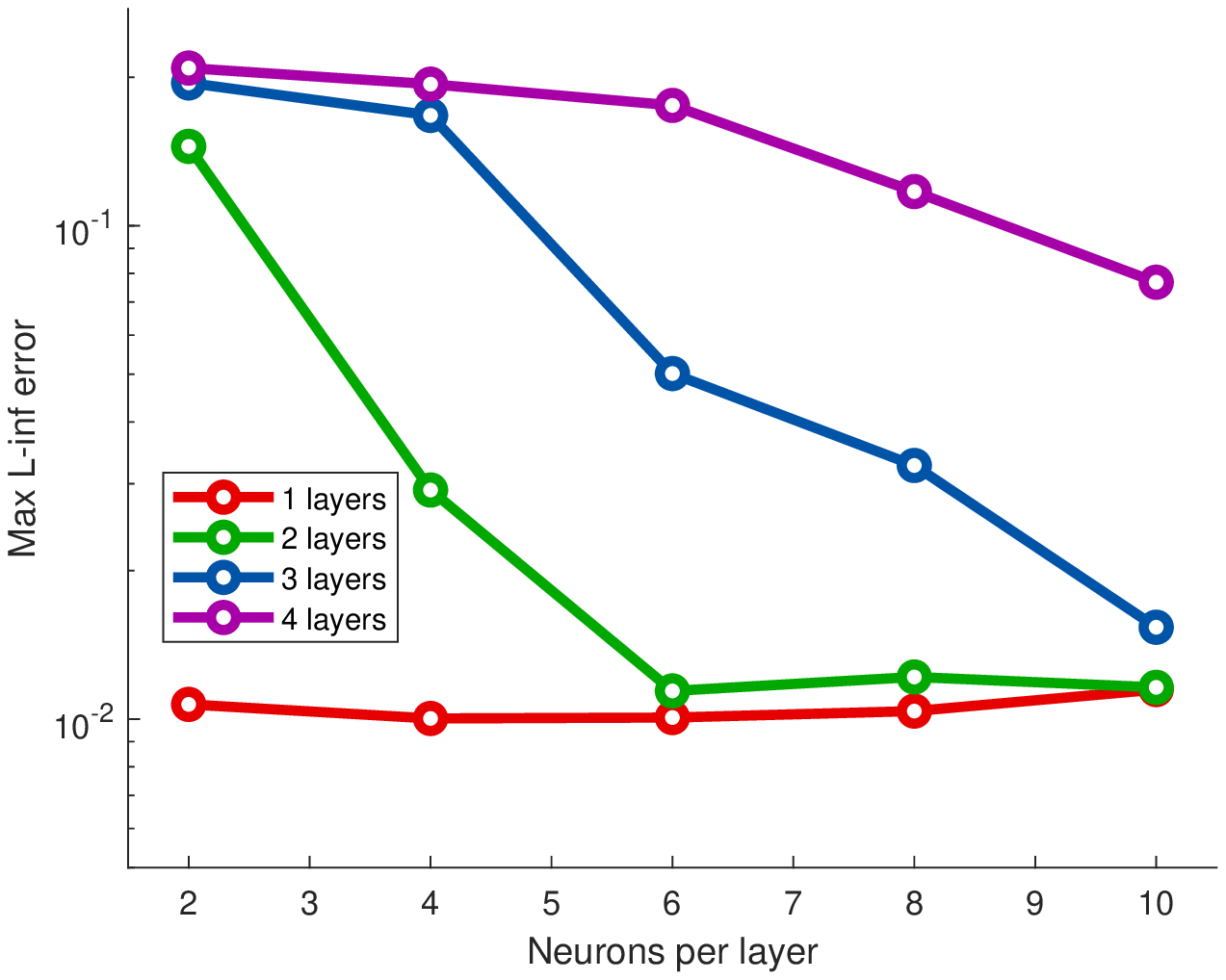}
}
\subfigure[Saddle point $L_2$ error]{
\includegraphics[width=0.31\linewidth]{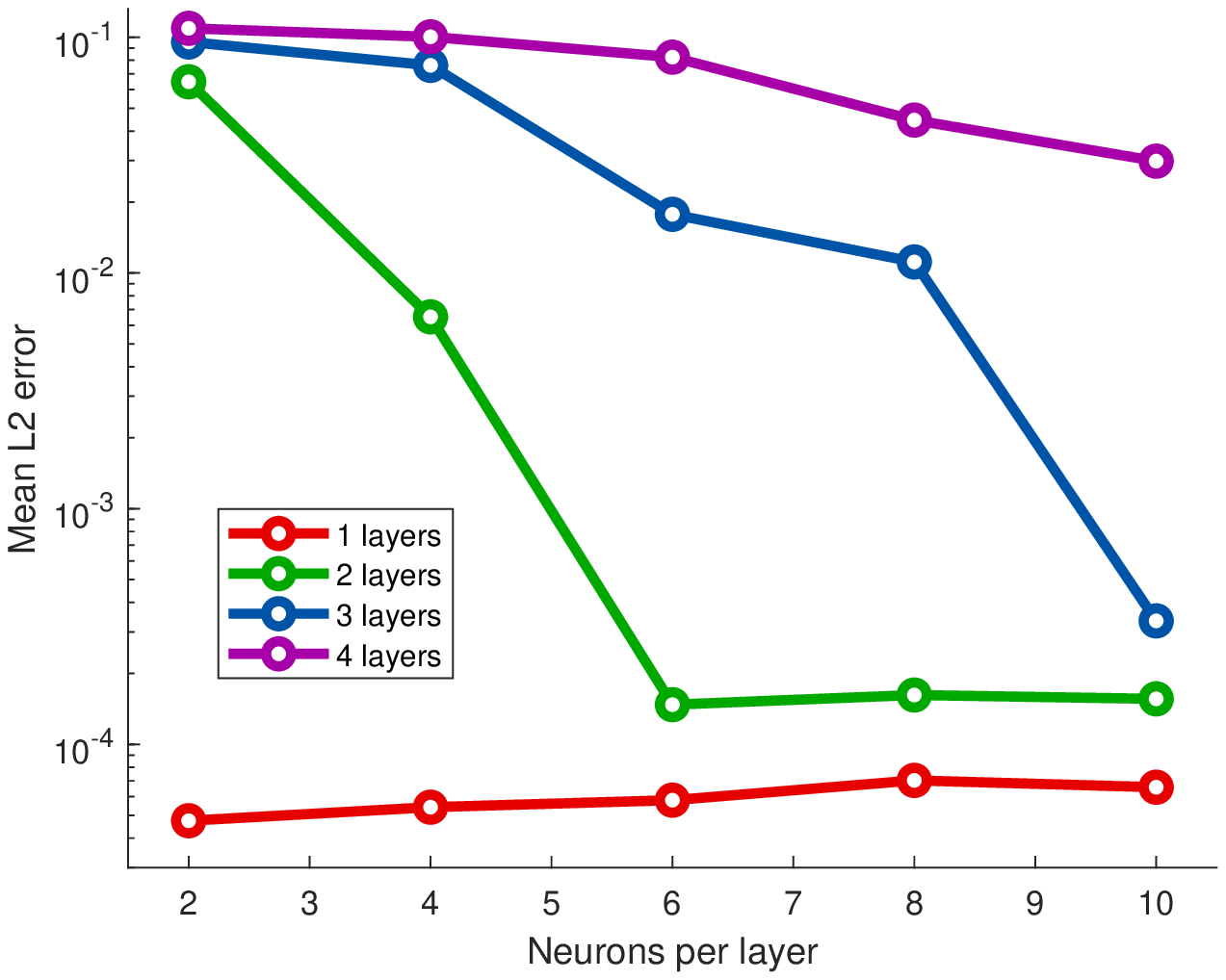}
}\quad\quad
\subfigure[Nodal sink $L_{\infty}$ error]{
\includegraphics[width=0.31\linewidth]{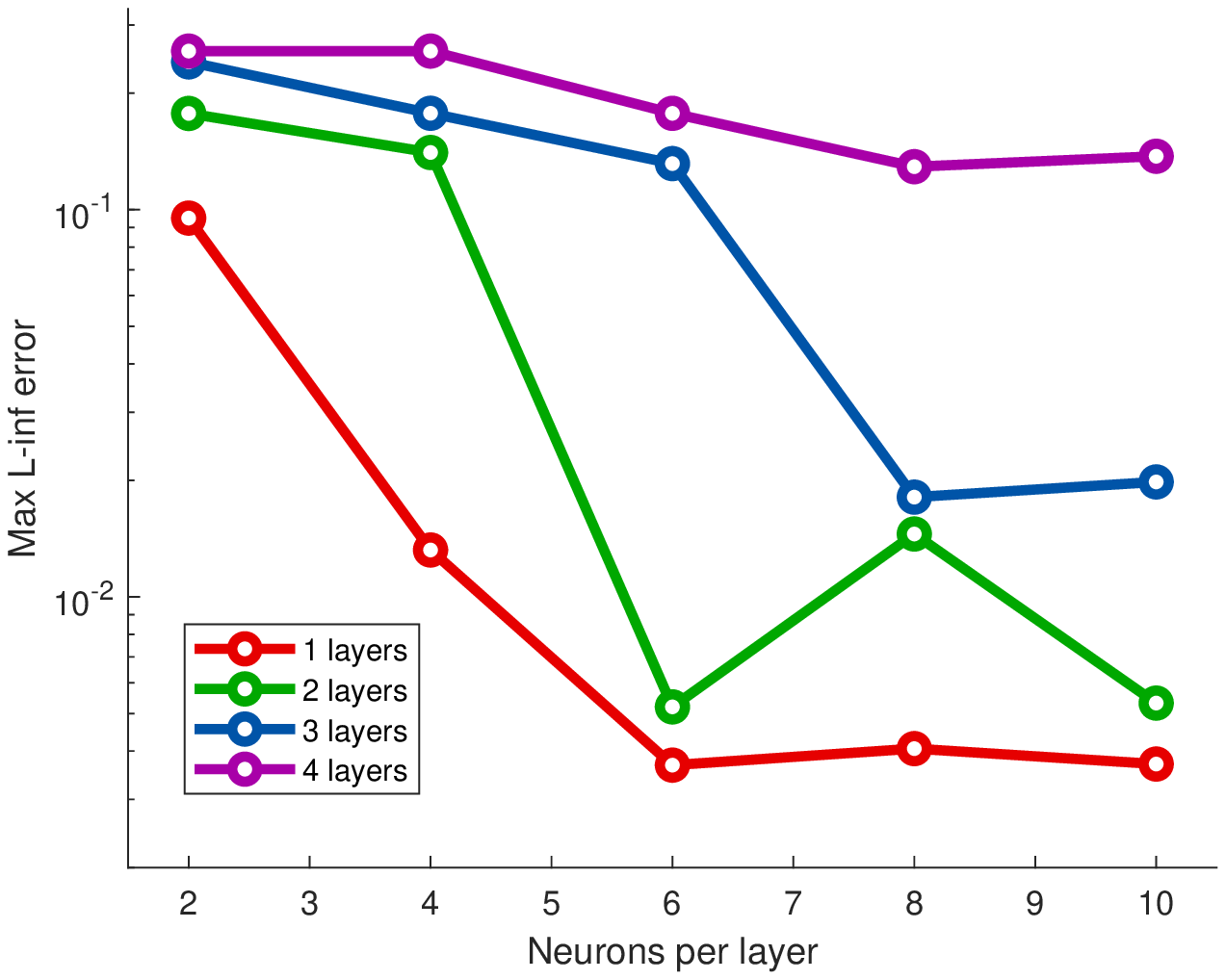}
}
\subfigure[Nodal sink $L_2$ error]{
\includegraphics[width=0.31\linewidth]{fig/ex1/linearnodalsinkL2.eps}
}
\caption{Example \ref{eq:ex1} of linear ODE systems architecture studies  }
\label{fig:ex1-architecture}
\end{figure}
For the target study we have $\vy_j^2$ generated from forward Euler, Runge-Kutta2 and Runge-Kutta4 methods with mesh size $\Delta=0.1$, and compare the ResNet output error of \eqref{error-Mean-L2} to the target error of \eqref{error-target}. In Figures \ref{fig:ex1-accuracy} part (a) and part (b) we present the target study for the Saddle point and Nodal sink systems. Dashed lines are for target errors and solid curves are for ResNet output errors corresponding to iterations. We observe the solid curves converging into the dashed lines with enough iterations, which implies that the ResNet output errors become dominated by the target errors. Furthermore we observe the errors reach  $10^{-2}$ for forward Euler target, to $10^{-3}$ for Runge-Kutta2 target and to $10^{-5}$ for Runge-Kutta4 target, which perfectly corresponds to the results of  $O(\Delta^2)$, $O(\Delta^3)$ and $O(\Delta^5)$ with $\Delta=0.1$, as derived in section \ref{S:3.1}.

\begin{figure}[htbp]
\centering
\subfigure[Target study for saddle point]{
\includegraphics[width=0.31\linewidth]{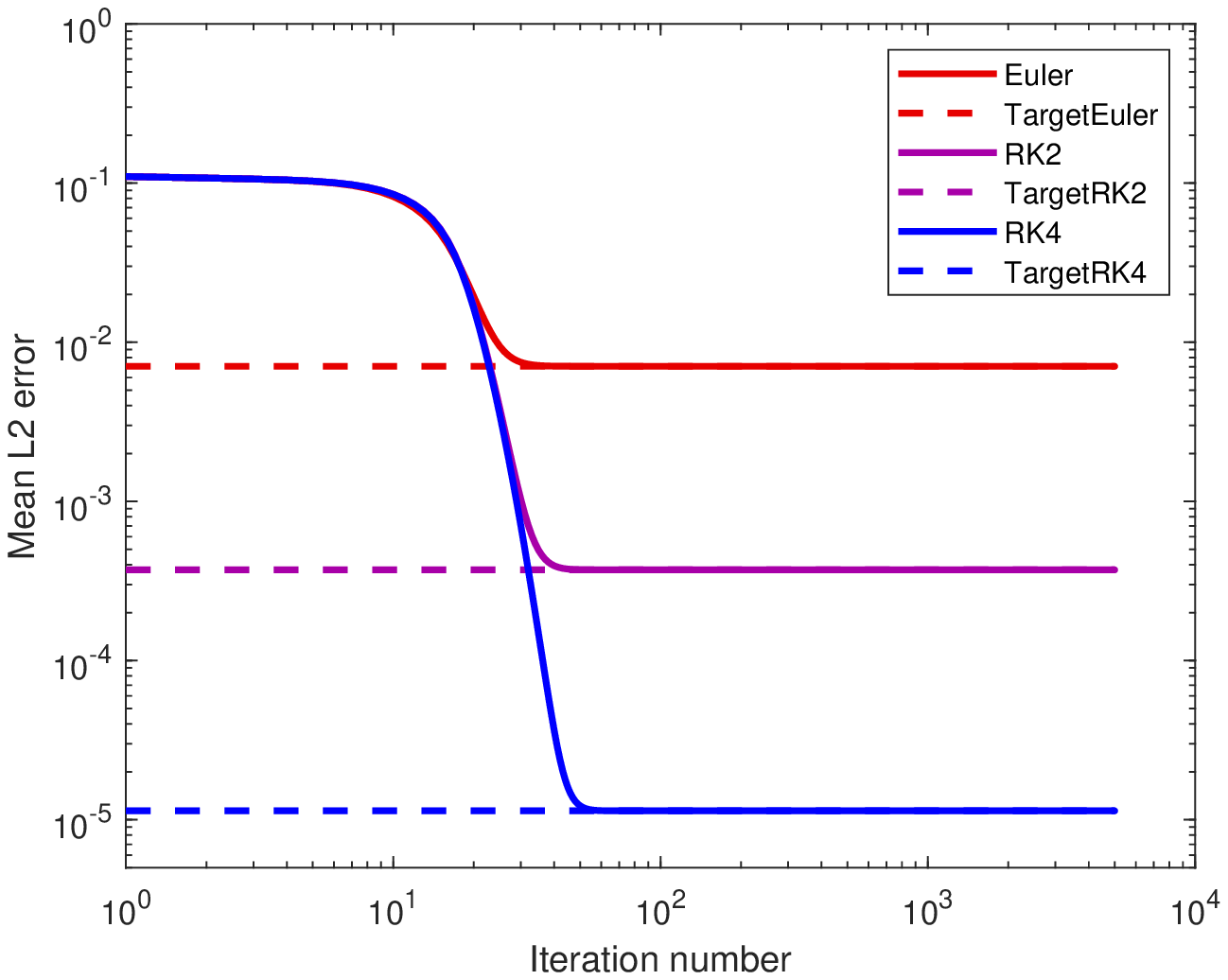}
%\caption{pic1}
}
\subfigure[Target study for nodal sink]{
\includegraphics[width=0.31\linewidth]{fig/ex1/linearnodalsinkunreliabledata.eps}
%\caption{}
}
\quad\quad
\subfigure[Trajectory for saddle point]{
\includegraphics[width=0.31\linewidth]{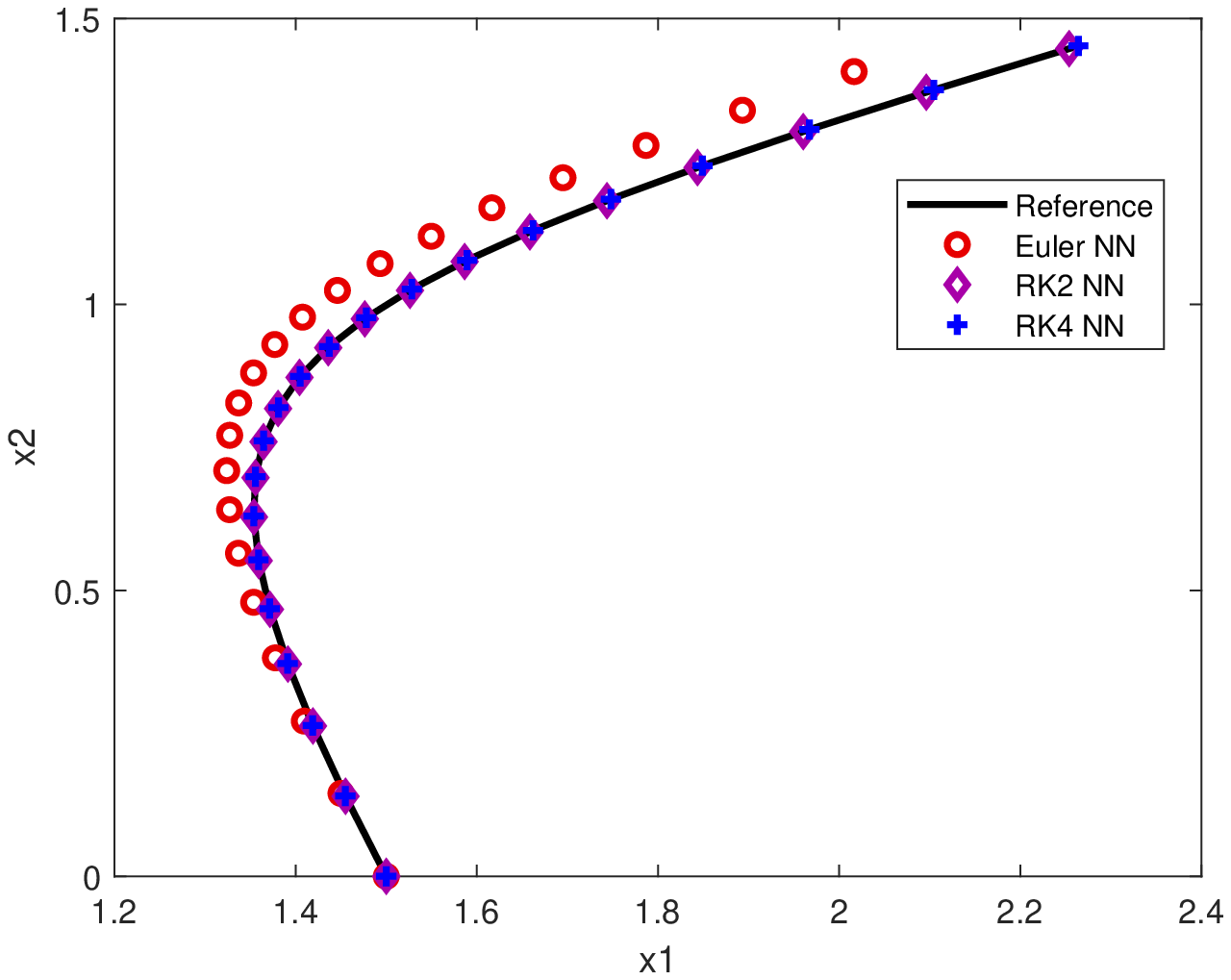}
%\caption{}
}
\subfigure[Trajectory for nodal sink]{
\includegraphics[width=0.31\linewidth]{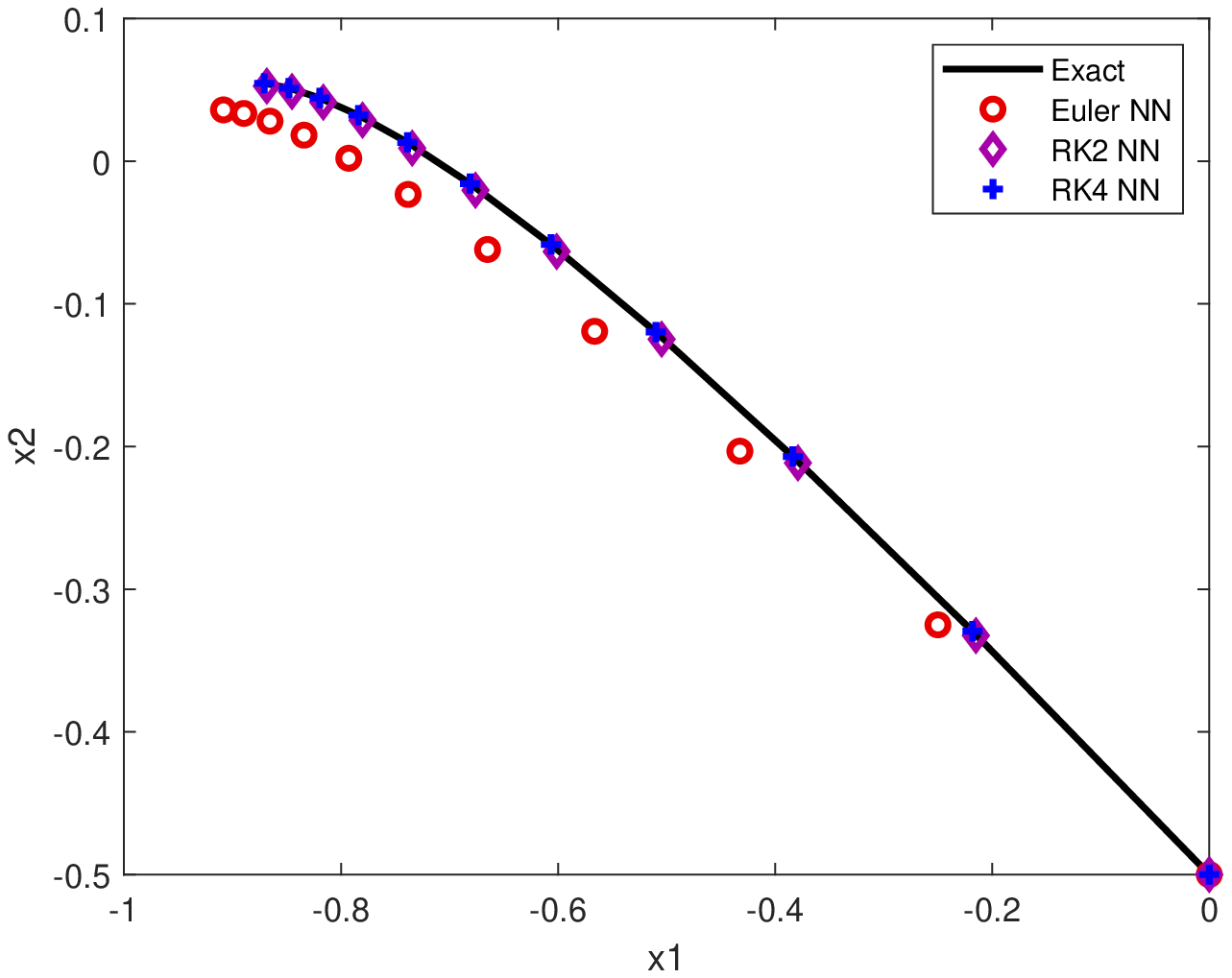}
%\caption{}
}
\caption{Example \ref{eq:ex1} of linear ODE systems target studies and trajectory approximations.
} 
\label{fig:ex1-accuracy}
\end{figure}

Finally, we apply the well trained networks to simulate solution trajectories in the phase plane, and compare against the highly accurate reference solution. In Figure \ref{fig:ex1-accuracy} part (c) and (d), we present the saddle point system trajectory simulation with initial position $\vx=(1.5,0)$ and up to final time $T=2.0$, while the nodal sink system trajectory has initial position $\vx=(0,-0,5)$ and final time $T=2.0$. The ResNet solver trained on forward Euler data quickly generates large errors, while the other two ResNet solvers agree very well with the reference solution trajectory. 

%In our experience, the values of scale $10^{-1}$ are not due to networks consistently ending at that error, but are instead due to a portion of the networks not training at all. This failure to train is characteristic of narrow ReLU networks, and further analysis can be found in \cite{Lu-2018}-?

%%%%%%%%%%%%%%%%%%%%%%%%%%%%%%%%%%%%%%%%%%%%%%%%%%

\begin{example}\label{ex2} {\bf \emph {Damped oscillating pendulum}}
\end{example}

In this example, we consider the motion of a damped oscillating pendulum modeled by the following ODE system.
\begin{fleqn}
\be\label{ex2-system}
 \left\{\begin{aligned}
     \dot{x_1} &= x_2.\\
\dot{x_2} &= -\gamma x_2-\omega^2 \sin(x_1).
 \end{aligned}
\right.
\ee
\end{fleqn}
Coefficient $\omega^2$ is related to the local gravitational acceleration and the length of the pendulum, and determines the frequency $\omega$ of the oscillation. Coefficient $\gamma$ is a linear drag coefficient, which serves to gradually decrease the magnitude of oscillations. Here we have $\omega^2=8.91$ and $\gamma=0.2$. The domain of interest is taken as $D= [-\pi,\pi]\times [-2\pi,2\pi]$.

\begin{figure}[htbp]
\centering
\subfigure[Training errors]{
\includegraphics[width=0.30\linewidth]{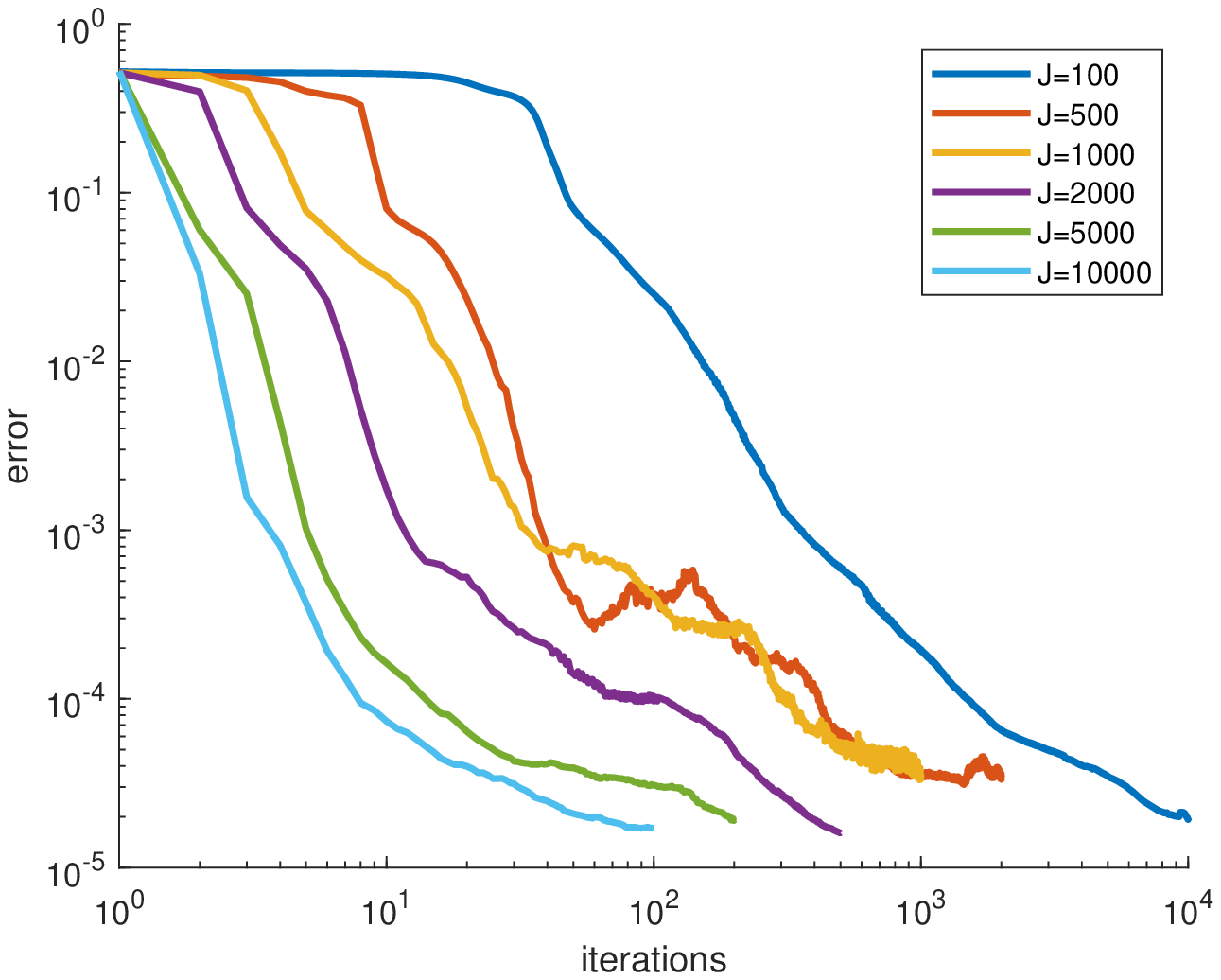}
%\caption{pic1}
}
\quad
\subfigure[Test set errors]{
\includegraphics[width=0.30\linewidth]{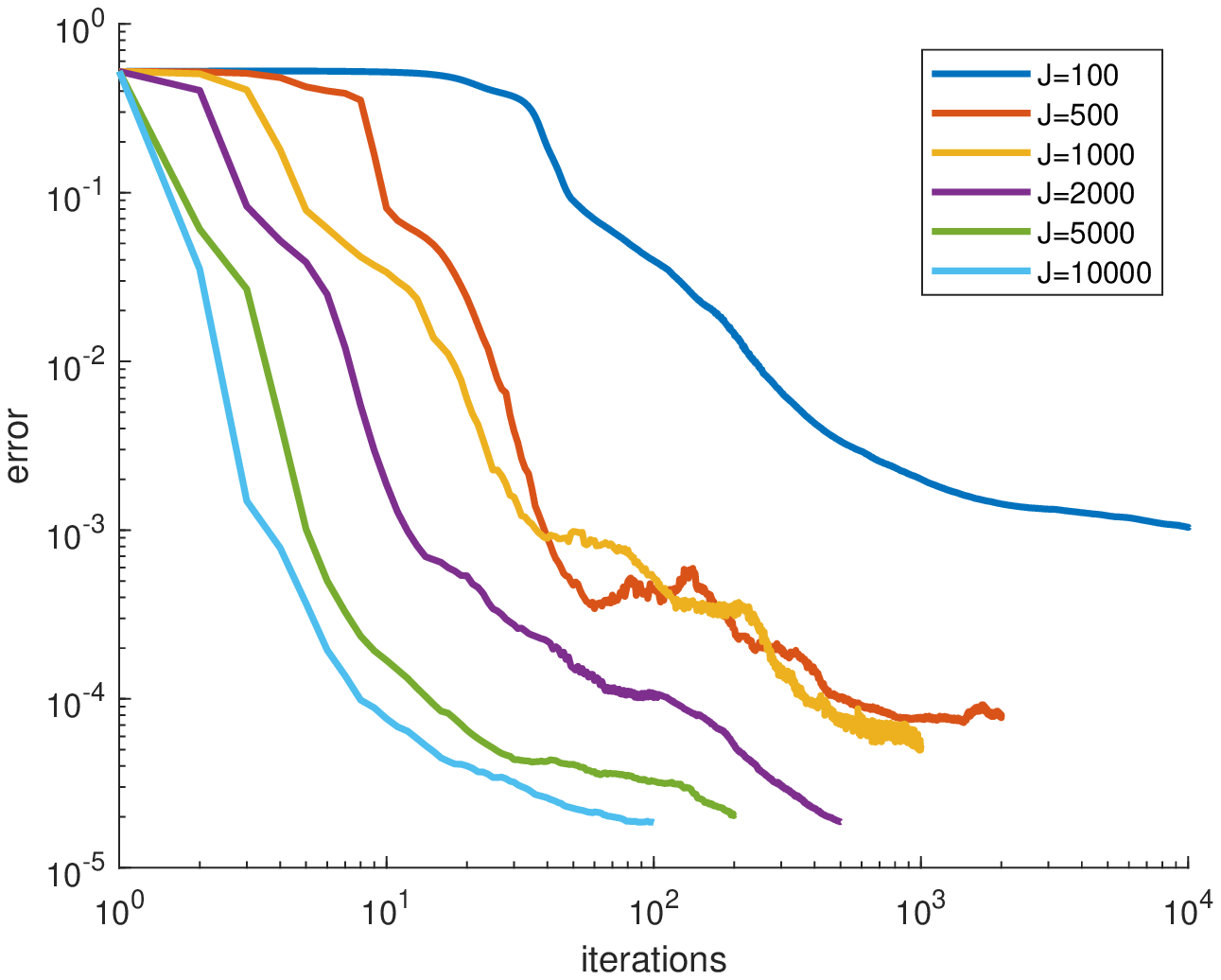}
%\caption{}
}
\caption{Learning data set density test with Example \ref{ex2}}
\label{fig:ex2-density}
\end{figure}

Before our three studies, we use this nonlinear ODE system to demonstrate the learning data set density problem. For deeper neural networks with more unknowns, the $J = 500$ training data pairs used in the previous example may not be enough to prevent sizable generalization errors. The density test is based on a neural network structure of six hidden layers and forty neurons per layer. We vary the number of training pairs from $J=100$ to $J=10000$ and observe how the mean square errors evolve over iterations for each $J$ value. A total fixed number of $10^6$ updates is applied for each $J$, leading to a total number of iterations that depends on the $J$ value. Error curves with different $J$ values are displayed in Figure \ref{fig:ex2-density}. This test gives us a guideline for what density obtains the smallest error with approximately the same computational time. {\emph{We pick the candidate of $J=2000$ data pairs and apply this learning data pair set value in the rest examples of the section}}.

\begin{figure}[htbp]
\centering
\subfigure[Architecture study with $L_{\infty}$ error]{
\includegraphics[width=0.33\linewidth]{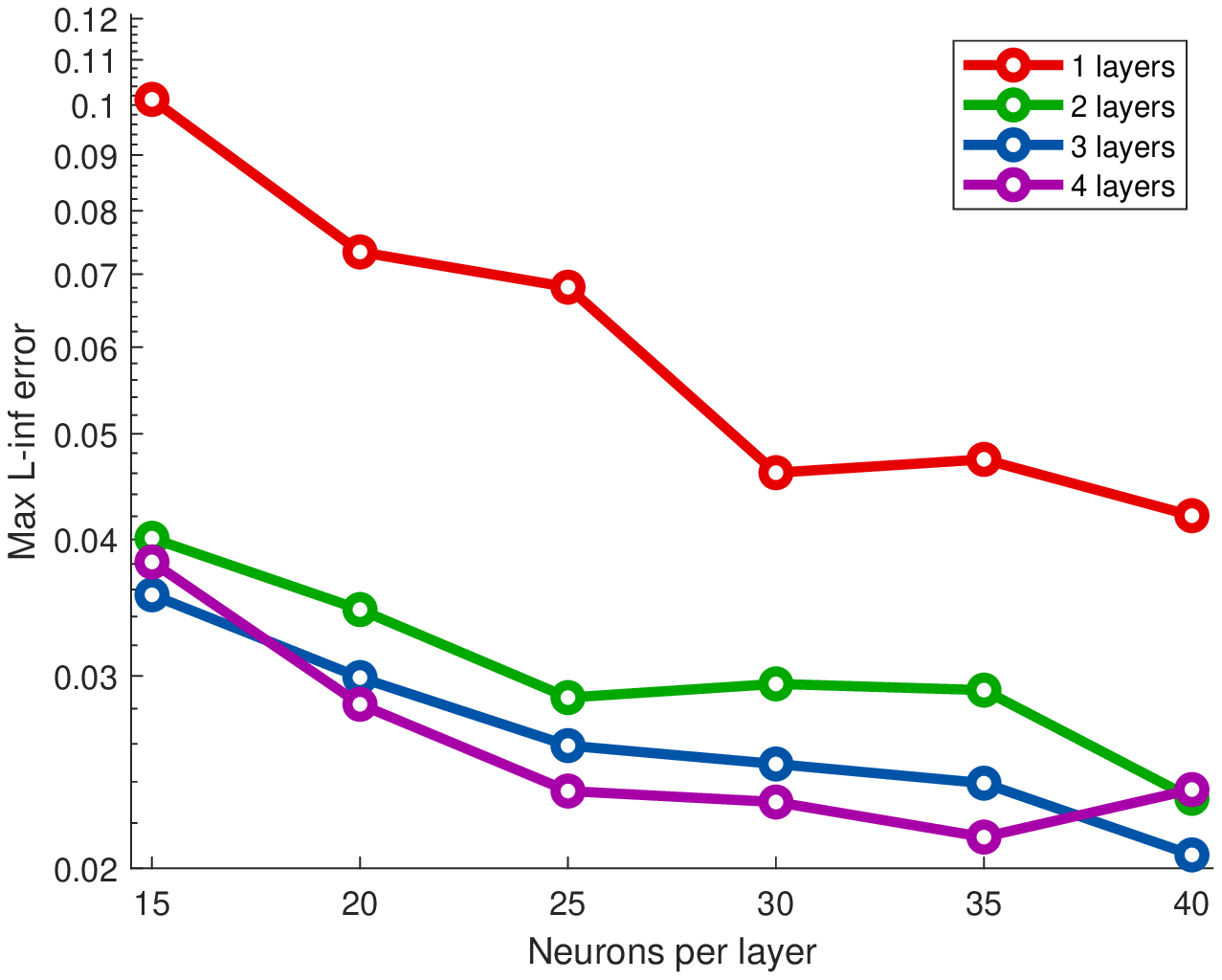}
%\caption{pic1}
}
\subfigure[Architecture study with $L_2$ error]{
\includegraphics[width=0.33\linewidth]{fig/ex3/damppendL2.eps}
%\caption{}
}\quad\quad
\subfigure[Target study on the system]{
\includegraphics[width=0.33\linewidth]{fig/ex3/damppendunreliabledata.eps}
%\caption{}
}
\quad
\subfigure[Phase plane trajectory simulation]{
\includegraphics[width=0.33\linewidth]{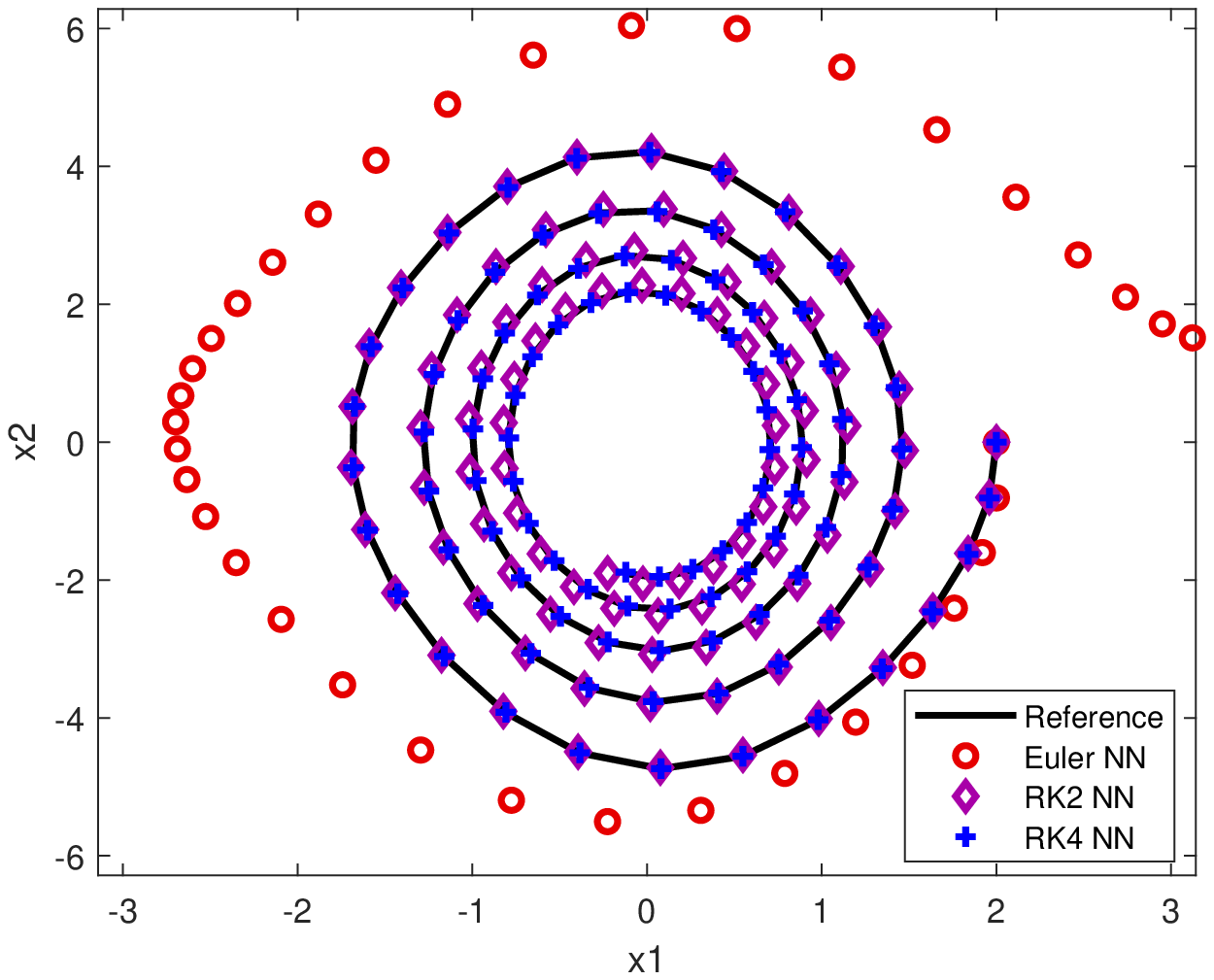}
%\caption{}
}
\caption{Example \ref{ex2} of the damped pendulum system.
} 
\label{fig:ex2-result}
\end{figure}

For the architecture study, each network is trained with $K=500$ iterations and the measured maximum $L_{\infty}$ error of \eqref{error-Max-Linfy} and mean $L_2$ error of \eqref{error-Mean-L2} for each architecture setting are outputted and illustrated in Figure \ref{fig:ex2-result} part (a) and part (b). The optimal neural network we choose has 2 hidden layers and 40 neurons per layer. 

For the target study, ResNet output errors of \eqref{error-Mean-L2} and target errors of \eqref{error-target} with targets generated from forward Euler, Runge-Kutta2 and Runge-Kutta4 with mesh size $\Delta=0.1$ are displayed in Figure \ref{fig:ex2-result} part (c). For this nonlinear problem, the target errors are consistent to the analyzed error orders of $O(\Delta^2)$, $O(\Delta^3)$ and $O(\Delta^5)$ for the three finite difference schemes. Solid curves of ResNet errors merge into the dashed lines of target errors over iterations. ResNet networks are very well trained and the network errors are dominated by the errors from targets.

In part (d) of Figure \ref{fig:ex2-result} we present the phase plane trajectory simulation by the well trained three neural networks. The solution curve starts at $\vx=(2,0)$ and runs up to final time $T=10.0$. The ResNet solver trained from forward Euler with relatively large mesh size $\Delta=0.1$ quickly diverges after one cycle of oscillation. Yet the ResNet solvers trained from same mesh size Runge-Kutta2 and Runge-Kutta4 methods behave well and almost exactly match the reference solution after several rounds of oscillations.

%%%%%%%%%%%%%%%%%%%%%%%%%%%%%%%%%%%%%%%%%%%%%%%%%%
%%%%%%%%%%%%%%%%%%%%%%%%%%%%%%%%%%%%%%%%%%%%%%%%%

\begin{example}\label{ex3} {\bf \emph {Nonlinear ODE system with four critical points}}
\end{example}

In this example, we consider the nonlinear ODE system
\begin{fleqn}
\be
 \left\{\begin{aligned}
     \dot{x_1} &= -(x_1-x_2)(1-x_1-x_2),\\
     \dot{x_2} &= x_1(2+x_2),
 \end{aligned}
\right.
\ee
\end{fleqn}
with its solutions' qualitative behavior explained in \cite{Boyce_DiPrima-2009}. The domain of interest is taken as $D = [-4, 4]\times[-3,3]$, which contains four critical points. The origin $(0, 0)$ is a unstable saddle point and $(0, 1)$ is an asymptotically stable spiral point. Node $(-2, -2)$ is an asymptotically stable point and point $(3, -2)$ is an unstable node.

\begin{figure}[htbp]
\centering
\subfigure[Architecture study with $L_{\infty}$ error]{
\includegraphics[width=0.33\linewidth]{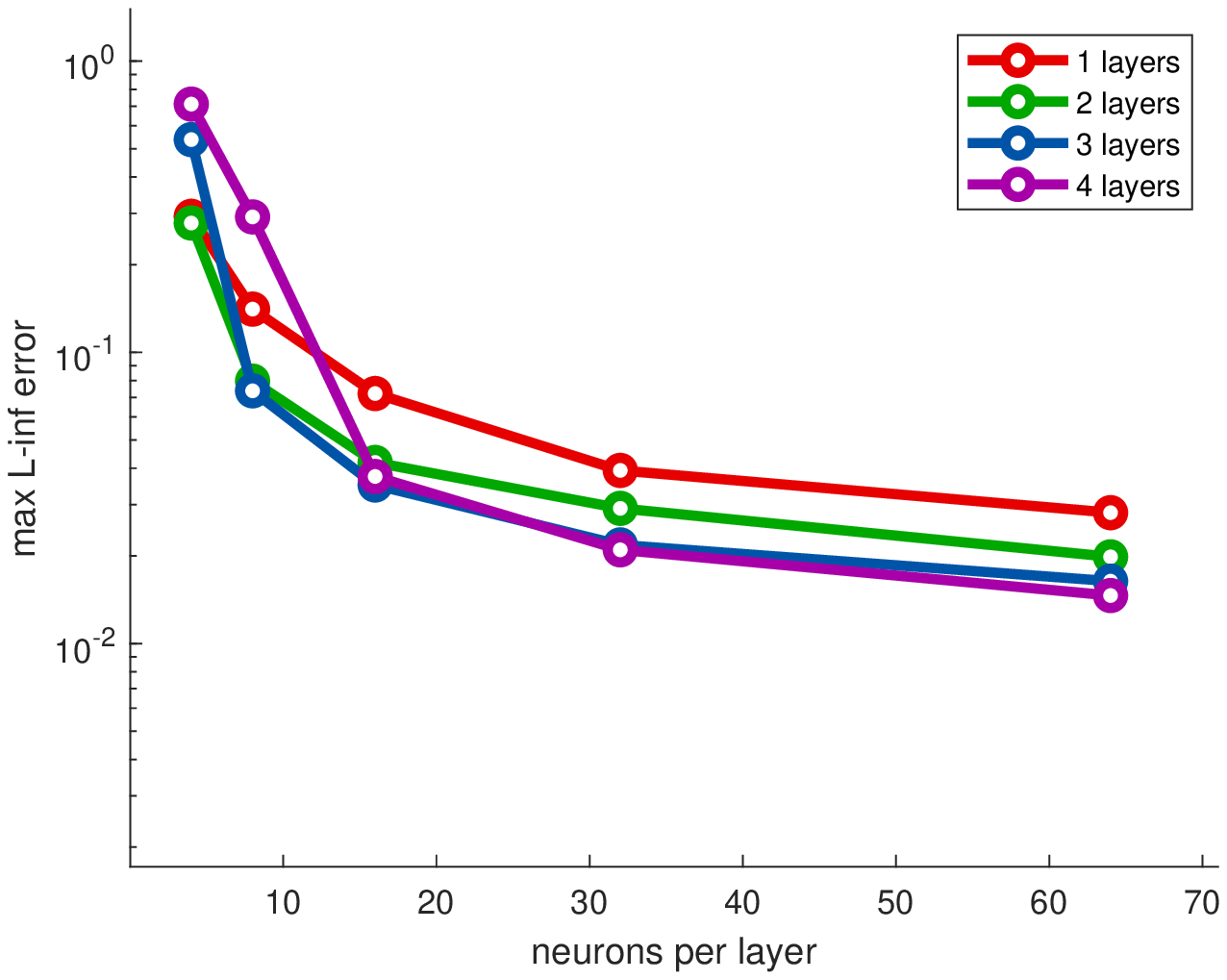}
%\caption{}
}
\subfigure[Architecture study with $L_{2}$ error]{
\includegraphics[width=0.33\linewidth]{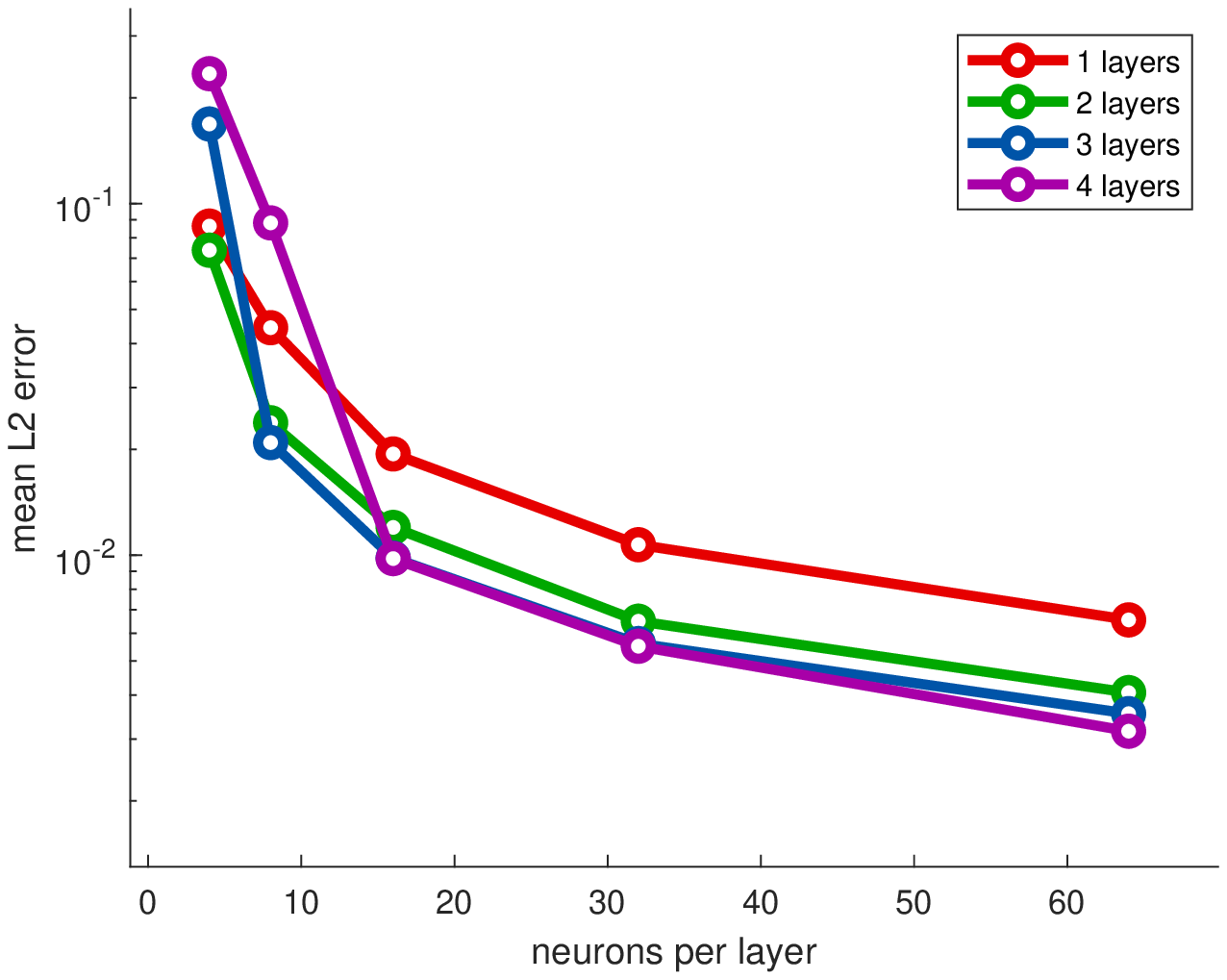}
%\caption{}
}
\quad\quad
\subfigure[Target study on the system]{
\includegraphics[width=0.33\linewidth]{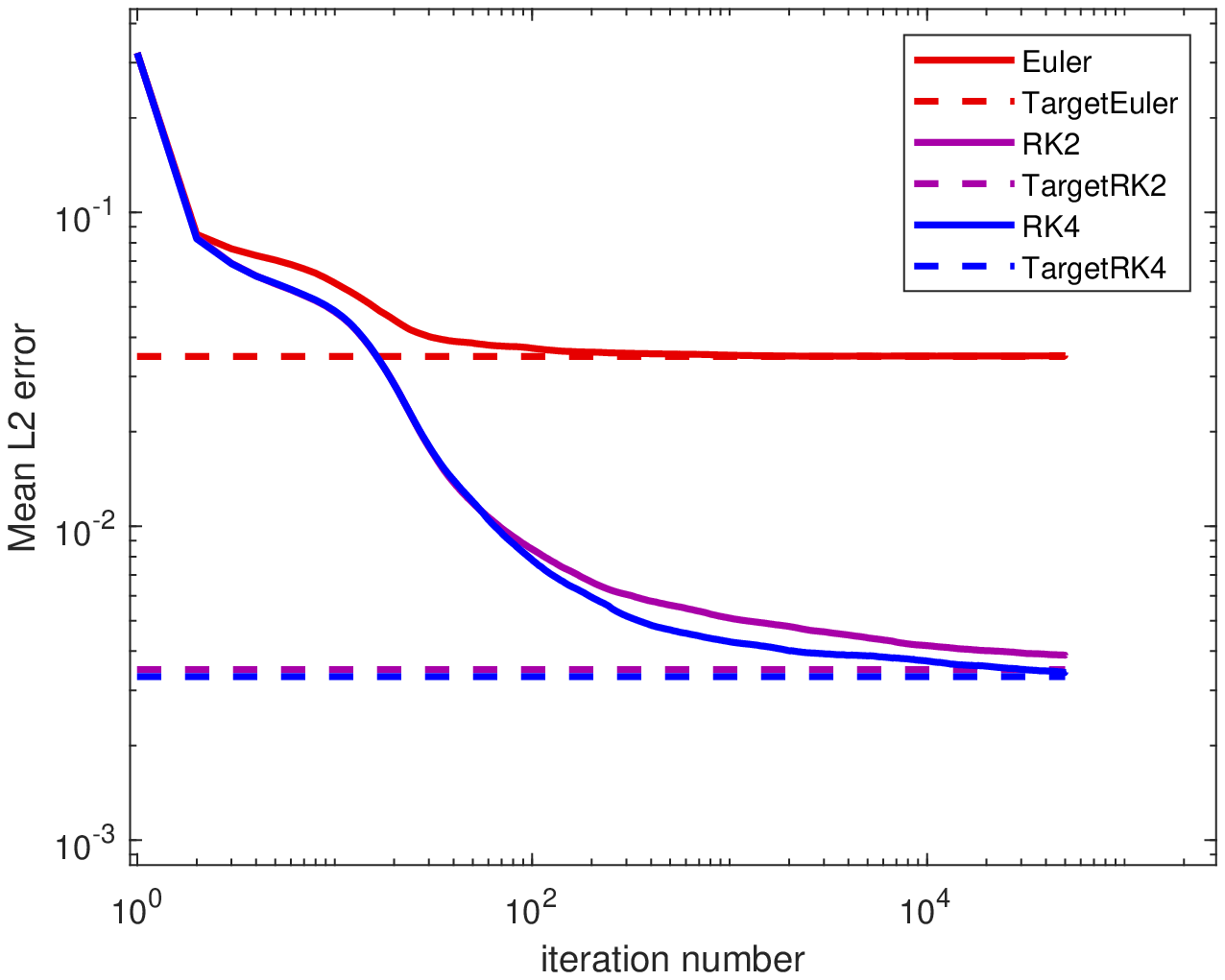}
%\caption{}
}
\subfigure[Phase plane trajectory simulation]{
\includegraphics[width=0.33\linewidth]{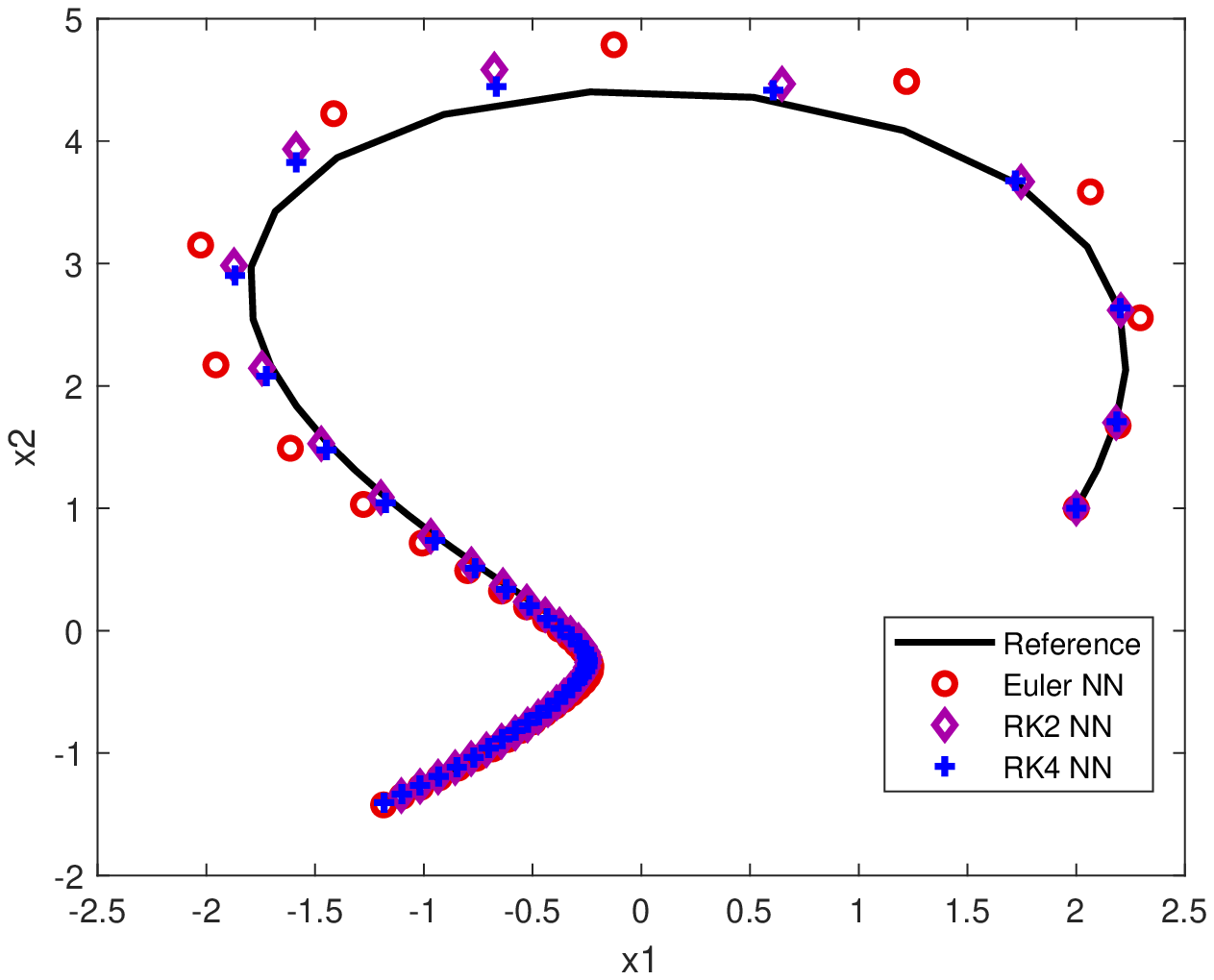}
%\caption{}
}
\caption{Example \ref{ex3} of the nonlinear ODE system with four critical points
} 
\label{fig:ex4-result}
\end{figure}

Time lag of $\Delta= 0.05$ is used in this example. For the architecture study, maximum $L_{\infty}$ error of \eqref{error-Max-Linfy} and mean $L_2$ error of \eqref{error-Mean-L2} for each architecture setting are illustrated in Figure \ref{fig:ex4-result} of part (a) and part (b). Mean $L_2$ error is one order of magnitude smaller than the $L_{\infty}$ error. We find the optimal choice for the network is hidden 2 layers and 64 neurons per layer.

For the target study, both network output and target errors of \eqref{error-Mean-L2} and \eqref{error-target} with targets generated from forward Euler, Runge-Kutta2 and Runge-Kutta4 with mesh size $\Delta=0.05$ are displayed in Figure \ref{fig:ex4-result} part (c). ResNet networks are very well trained and the network errors are dominated by target errors.

Solutions of the given problem are not regular enough, thus no difference is observed for the target error between Runge-Kutta2 and Runge-Kutta4 methods. The simulated solution trajectory starts at $(2,1)$ and runs up to final time $T=4.0$. In Figure \ref{fig:ex4-result} part (d) we present the three well trained ResNet approximations of the curve. ResNet solver trained from forward Euler gives larger error over the trace, yet the error becomes smaller when the solution curve gets close to the asymptotically stable node $(-2, -2)$.

% \begin{figure}[htbp]
% \centering
% \includegraphics[width=0.75\linewidth]{fig/phase-portrait-demo-ex4.pdf}
% \caption{phase portrait demo (from the book)
% } 
% \label{fig:ex4-portrait}
% \end{figure}

%%%%%%%%%%%%%%%%%%%%%%%%%%%%%%%%%%%%%%%%%%%%%%%%%%%%%%%%%%%%%%%%%

\begin{example}\label{ex4} {\bf \emph {Cubic power ODE system with unit disk barrier}}
\end{example}
 
In this example, we consider the nonlinear ODE system
\begin{fleqn}
\be
 \left\{\begin{aligned}
     \dot{x_1} &= x_2-x_1(x_1^2+x_2^2-1),\\
     \dot{x_2} &= -x_1-x_2(x_1^2+x_2^2-1),
 \end{aligned}
\right.
\ee
\end{fleqn}
with the unit disk $D_b=\{\vx: \|\vx\|_2=1\}$ as a barrier. We consider the square domain of $D=[-2, 2]\times [-2, 2]$ as the domain of interest, which includes the barrier $D_b$.

\begin{figure}[htbp]
\centering
\subfigure[Architecture study with $L_{\infty}$ error]{
\includegraphics[width=0.33\linewidth]{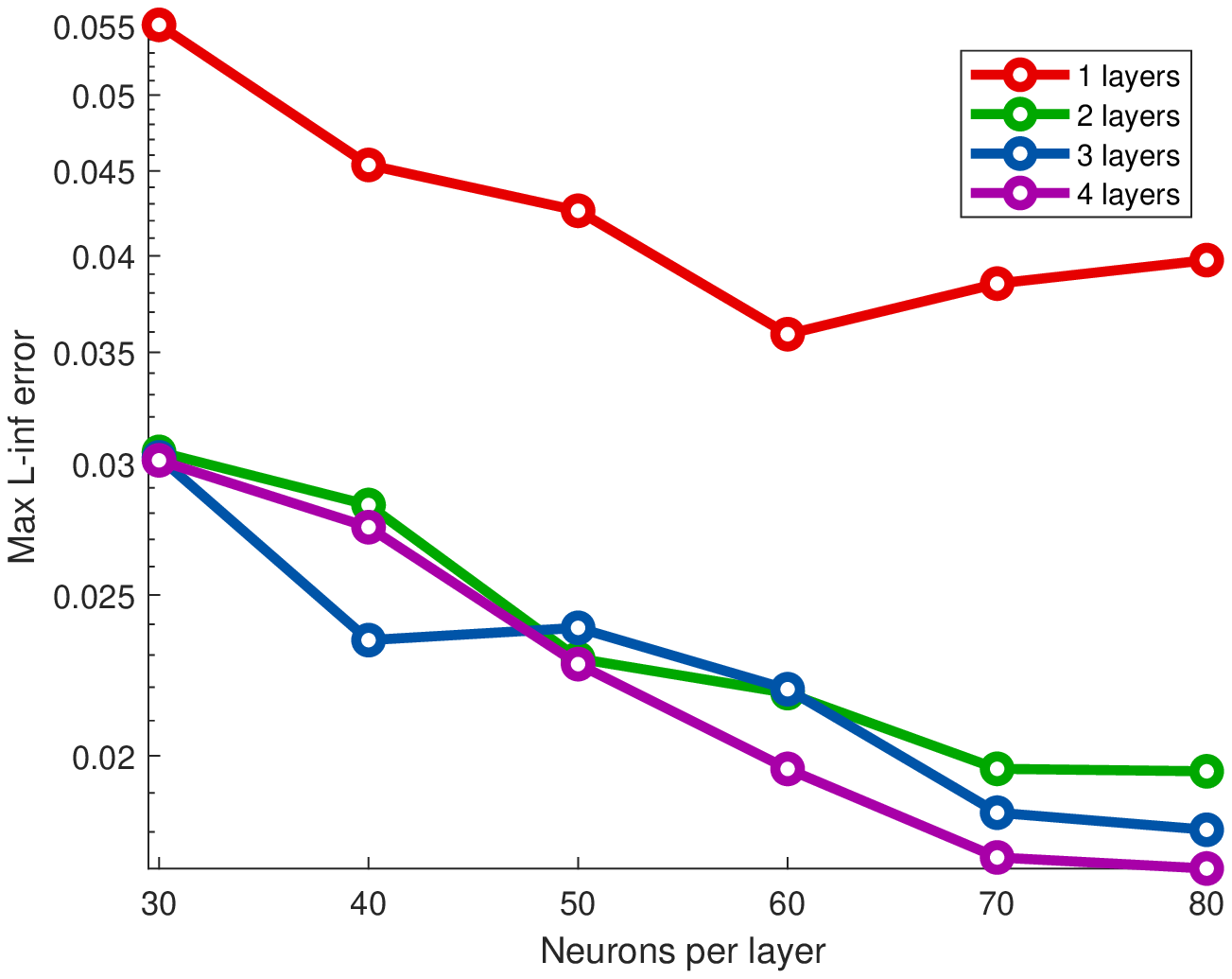}
%\caption{pic1}
}
\subfigure[Architecture study with $L_2$ error]{
\includegraphics[width=0.33\linewidth]{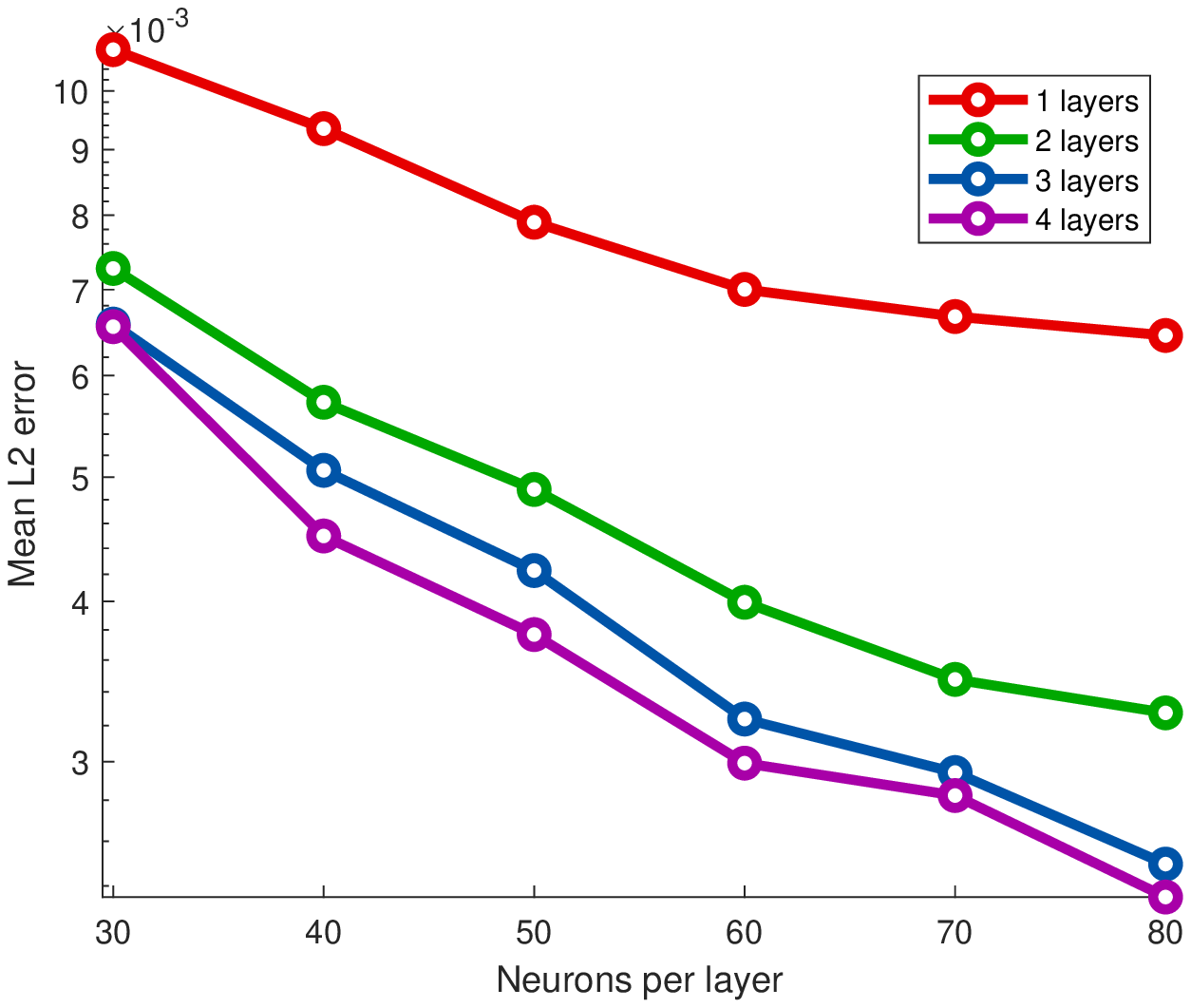}
%\caption{}
}\quad\quad
\subfigure[Target study on the system]{
\includegraphics[width=0.33\linewidth]{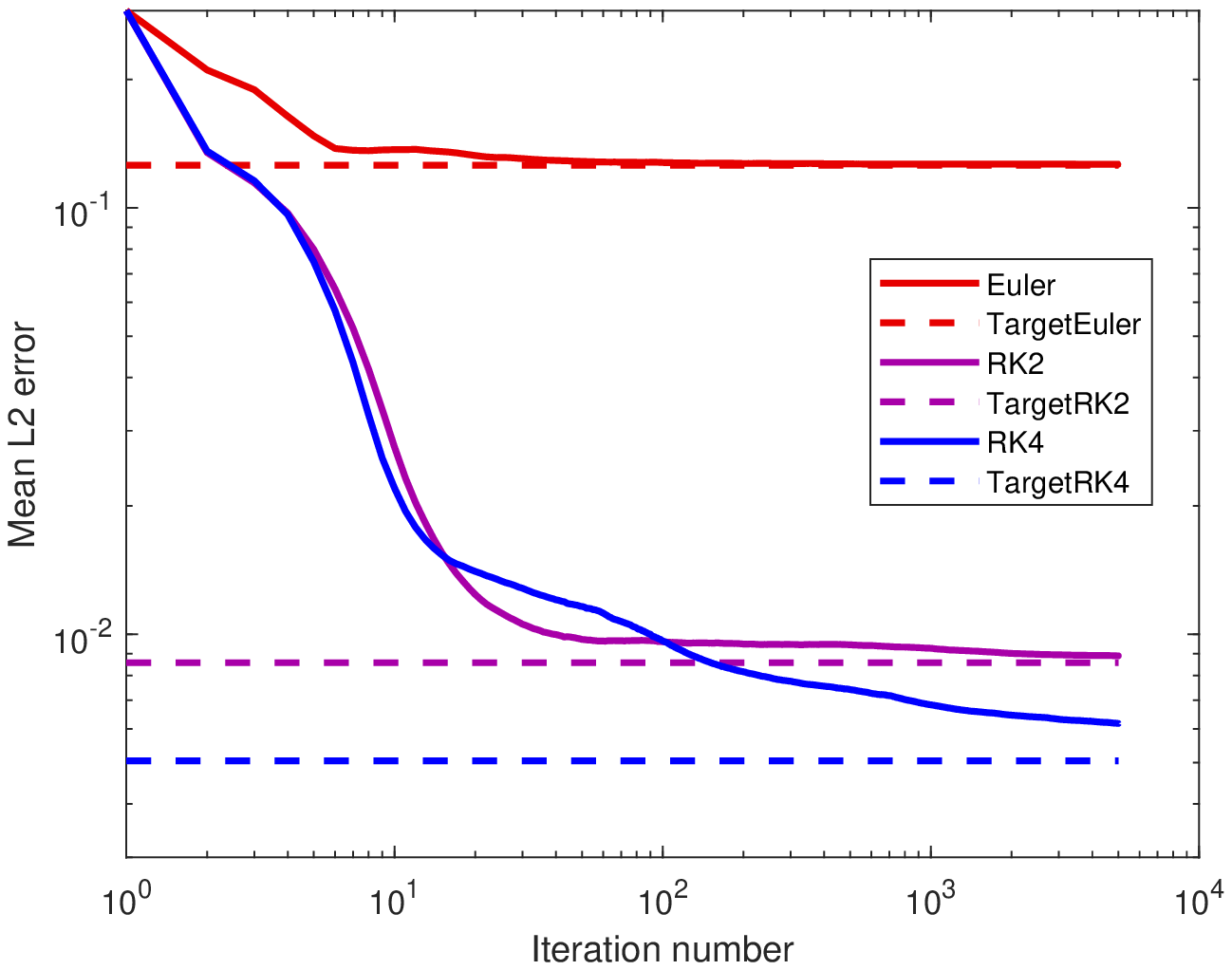}
%\caption{}
}
\quad
\subfigure[Phase plane trajectory simulation]{
\includegraphics[width=0.33\linewidth]{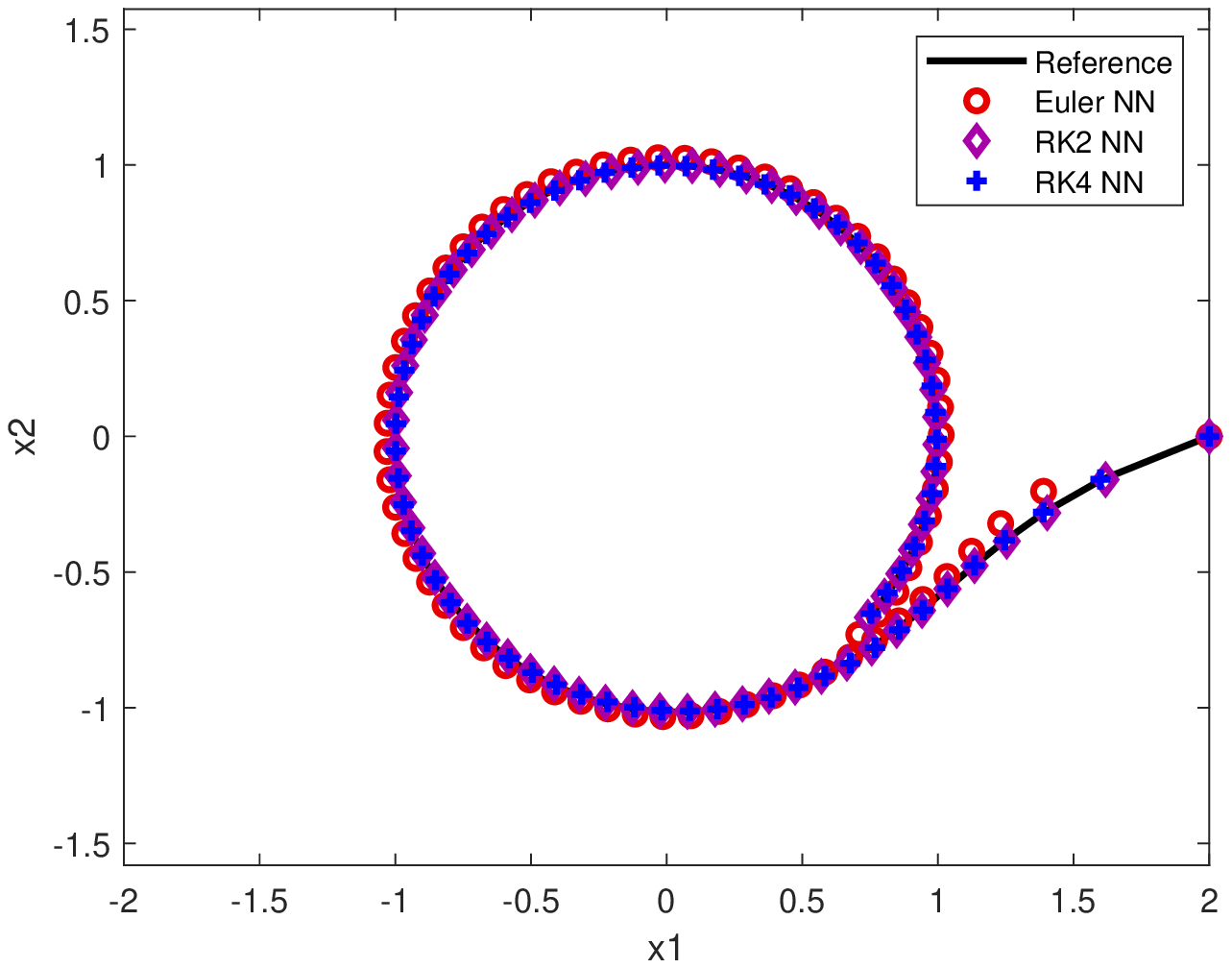}
%\caption{}
}
\caption{Example \ref{ex4} of the cubic power ODE system with unit disk barrier
} 
\label{fig:ex5-results}
\end{figure}

Time lag of $\Delta= 0.1$ is used in this example. For the architecture study, maximum $L_{\infty}$ error of \eqref{error-Max-Linfy} and mean $L_2$ error of \eqref{error-Mean-L2} for each architecture setting are computed and illustrated in Figure \ref{fig:ex4-result} of part (a) and part (b). The chosen optimal network architecture has 3 hidden layers and 80 neurons per layer in the network. 

For the target study, ResNet output and target errors of \eqref{error-Mean-L2} and \eqref{error-target} are displayed in Figure \ref{fig:ex4-result} part (c), where the training targets are generated from forward Euler, Runge-Kutta2 and Runge-Kutta4 methods with mesh size $\Delta=0.1$. Solid curves of ResNet errors merge into the dashed lines of target errors over iterations. Networks are well trained and ResNet errors are dominated by target errors.

The initial condition of the simulated solution curve on the phase plane is set as $\vx = [2, 0]$ and we run the simulation to final time $T = 7.0$. Figure \ref{fig:ex5-results} part (c) show the three well trained ResNet approximations of the trajectory. Even though the forward Euler ResNet network error is on the level of $10^{-1}$, its simulation matches well with the reference solution, which goes into the barrier and circulates around $D_b$ as $t \rightarrow \infty$.

%%%%%%%%%%%%%%%%%%%%%%%%%%%%%%%%%%%%%%%%%%%%%%%%%%

%%%%%%%%%%%%%%%%%%%%%%%%%%%%%%%%%%%%%%%%%%%%%%%%%%

\begin{example}\label{ex5} {\bf \emph {Modified Lotka-Volterra Predator-Prey model}}
\end{example}

In this example we consider the modified Lotka-Volterra model 
\begin{fleqn}
\be
 \left\{\begin{aligned}
     \dot{x_1} &= x_1\left(1-0.2x_1-\frac{2x_2}{x_1+6}\right),\\
     \dot{x_2} &= x_2\left(-0.25+\frac{x_1}{x_1+6}\right),
 \end{aligned}
\right.
\ee
\end{fleqn}
from \cite{Boyce_DiPrima-2009}.
Time lag of $\Delta= 0.1$ is applied and all learning data pairs are taken from the domain of interest $D = [0, 5]\times[0, 5]$. For the architecture study, maximum $L_{\infty}$ error of \eqref{error-Max-Linfy} and mean $L_2$ error of \eqref{error-Mean-L2} for each architecture setting are calculated and illustrated in Figure \ref{fig:ex6-result} part (a) and part (b). The optimal architecture of the network is 2 hidden layers with 128 neurons per layer.

\begin{figure}[htbp]
\centering
\subfigure[Architecture study with $L_{\infty}$ error]{
\includegraphics[width=0.33\linewidth]{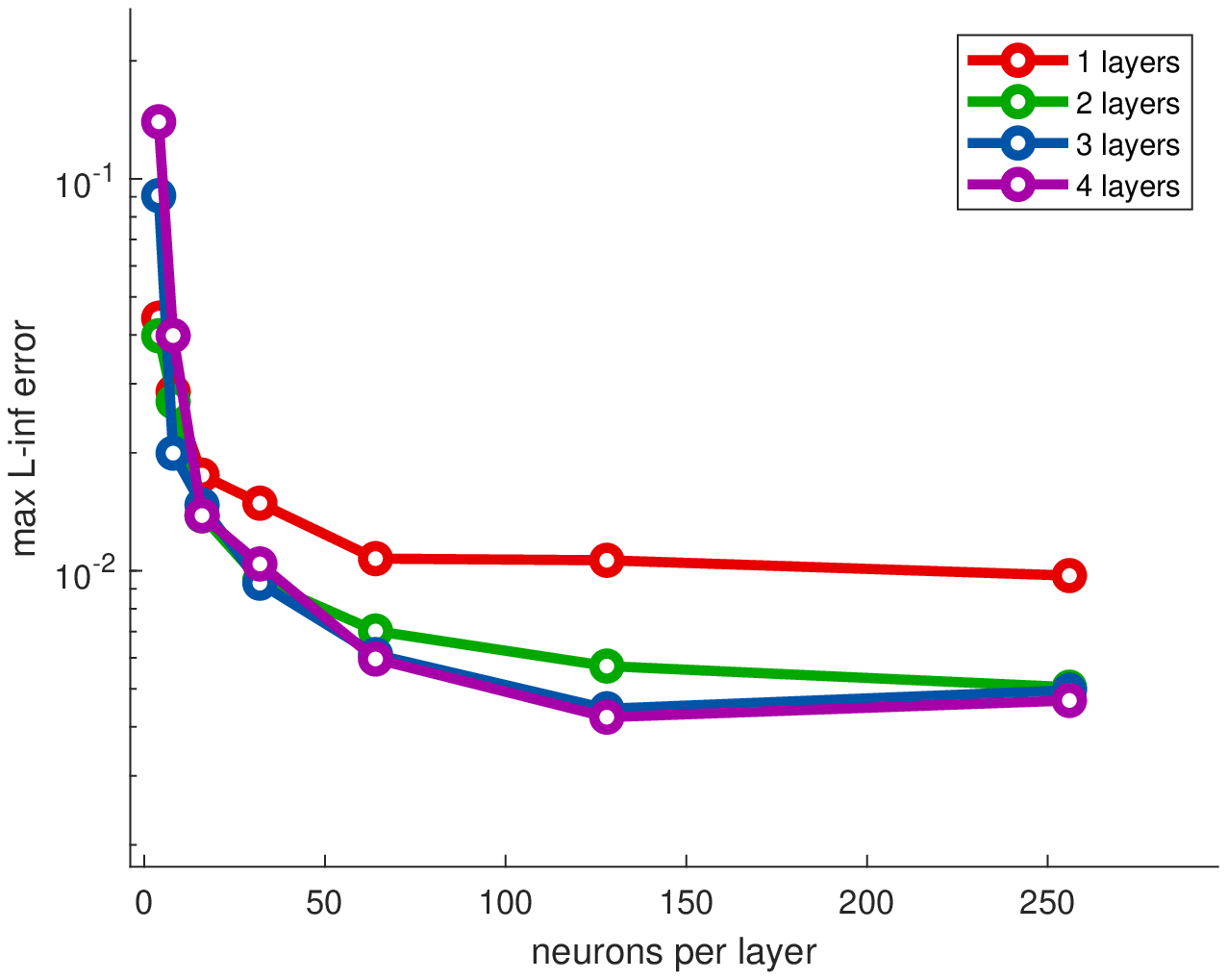}
%\caption{}
}
\subfigure[Architecture study with $L_{2}$ error]{
\includegraphics[width=0.33\linewidth]{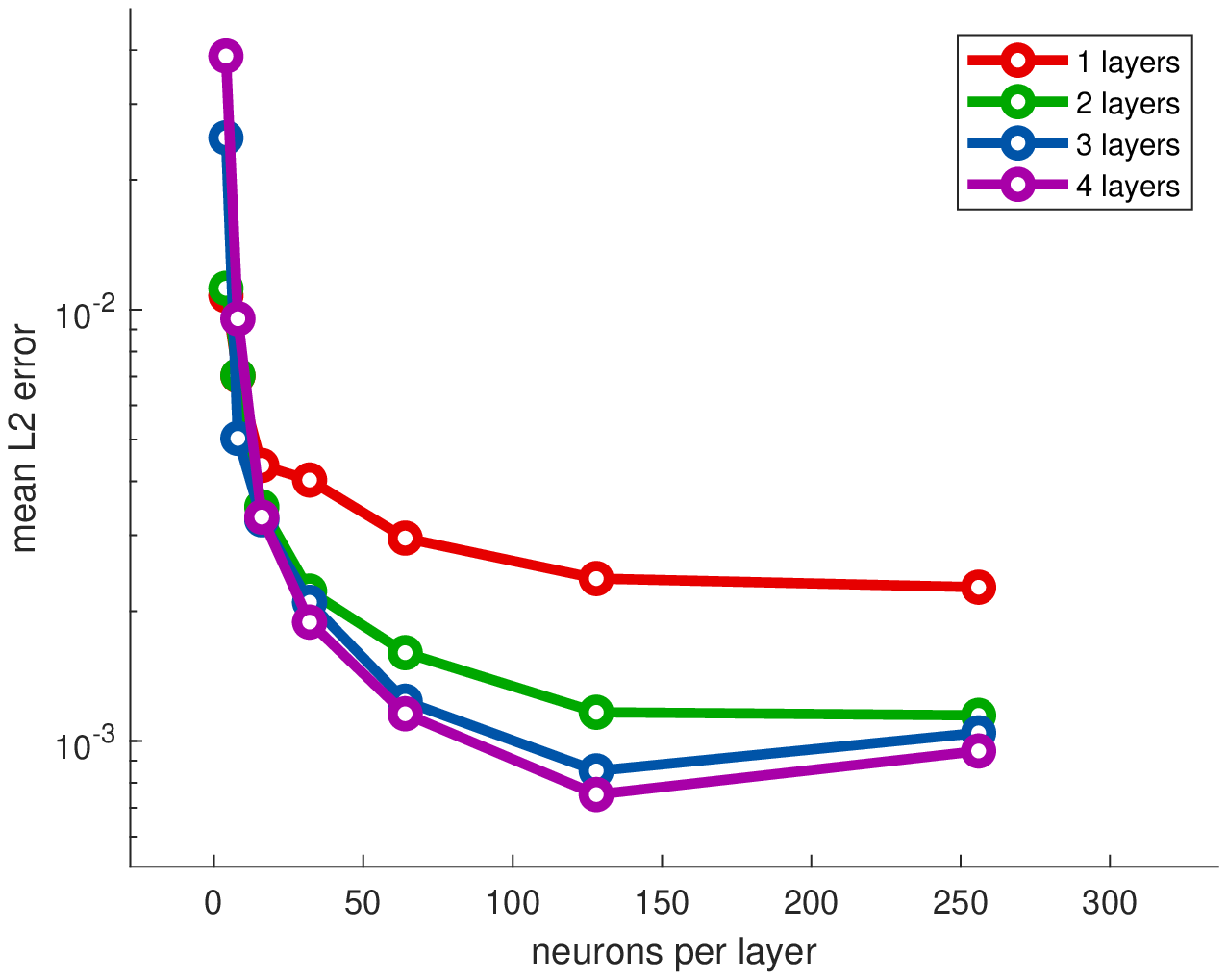}
%\caption{}
}
\quad\quad
% \subfigure[Optimally trained network vs true trajectory]{
% \includegraphics[width=0.45\linewidth]{fig/ex7/ex7-optimal.eps}
% %\caption{}
% }
\subfigure[Target study of the system]{
\includegraphics[width=0.33\linewidth]{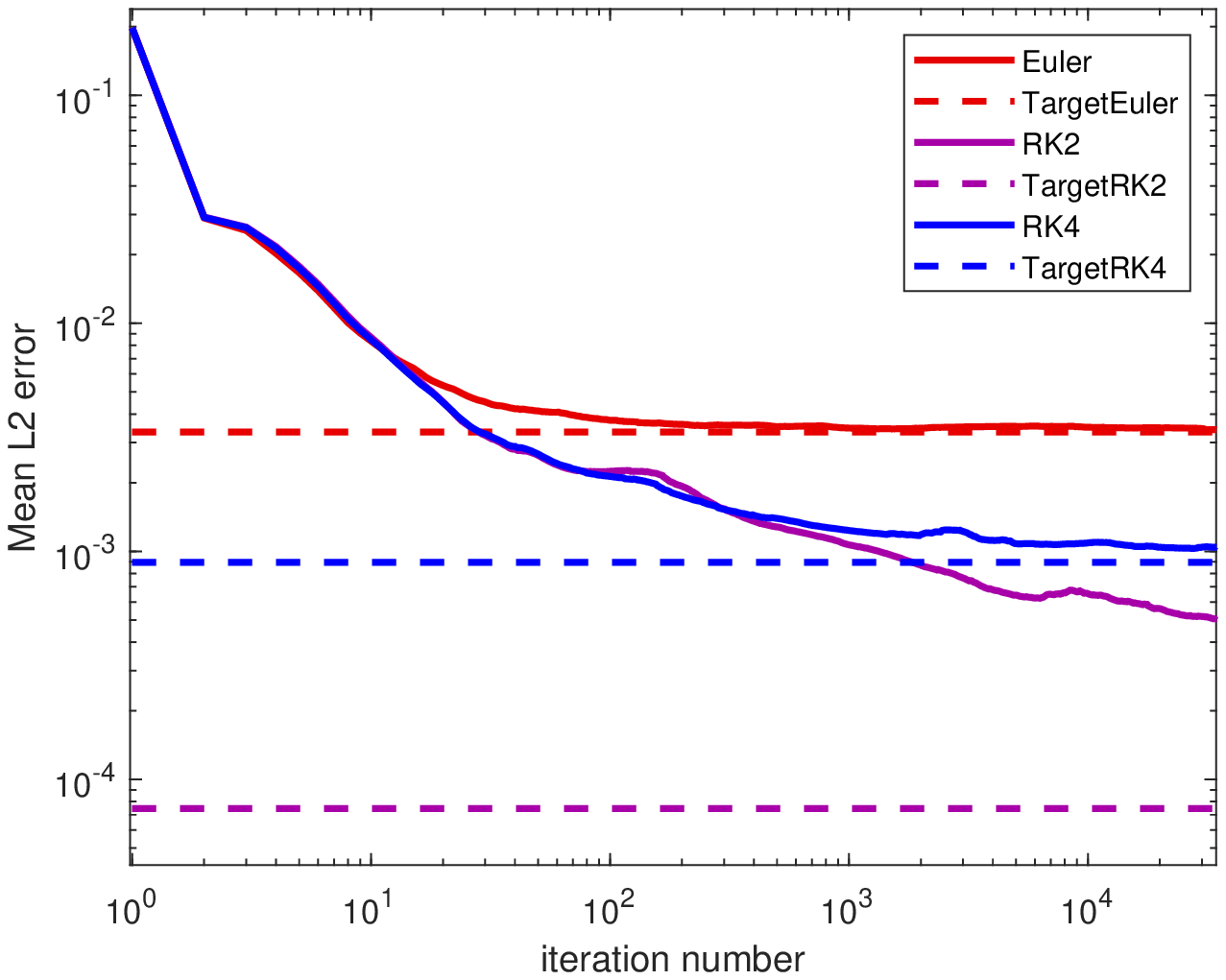}
%\caption{}
}
\subfigure[Phase plane trajectory simulation]{
\includegraphics[width=0.33\linewidth]{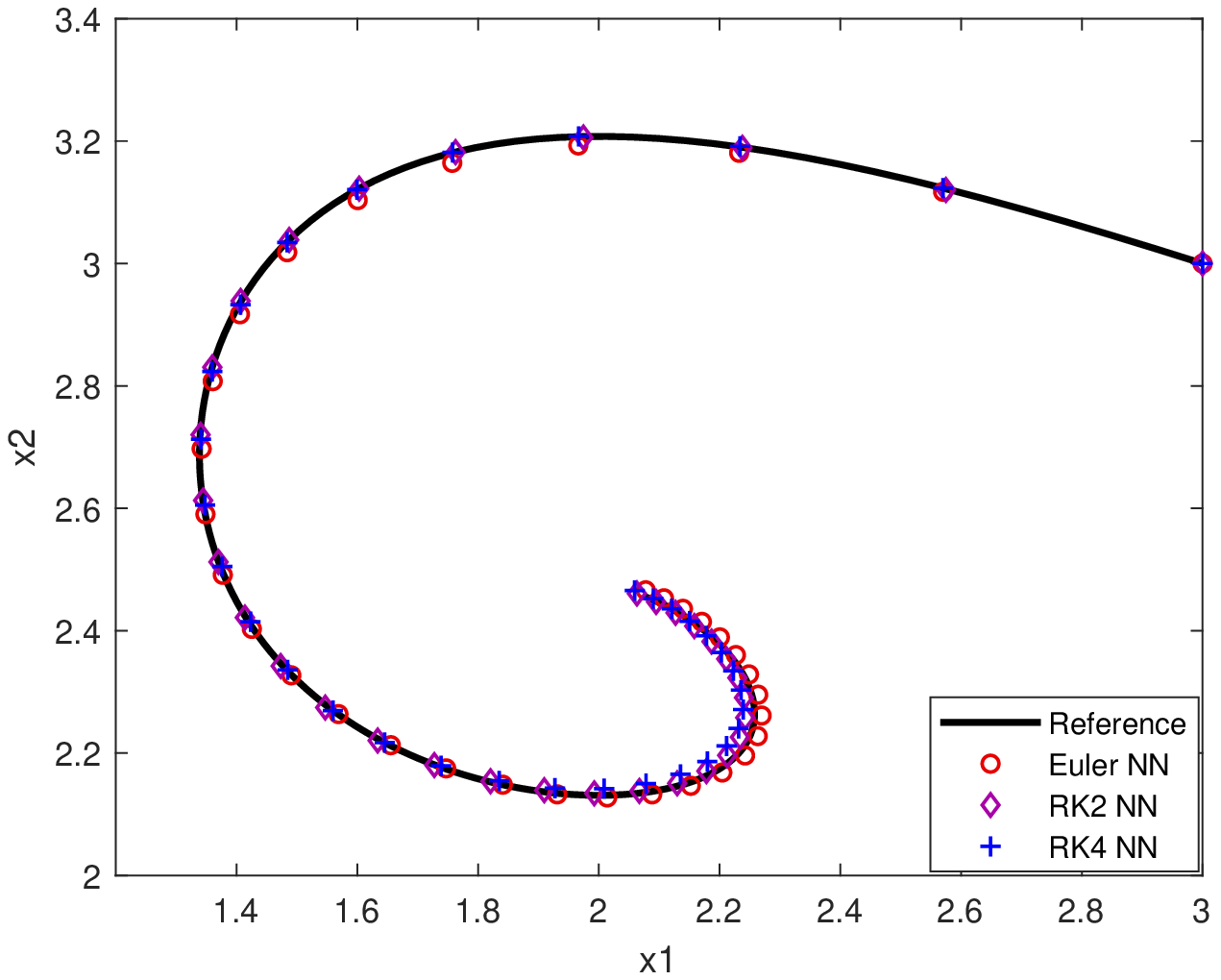}
%\caption{}
}
\caption{Example \ref{ex5} of the modified Lotka-Volterra predator-prey model 
} 
\label{fig:ex6-result}
\end{figure}

For the target study, ResNet output and target errors of \eqref{error-Mean-L2} and \eqref{error-target} are calculated and displayed in Figure \ref{fig:ex4-result} part (c), with the training targets generated from forward Euler, Runge-Kutta2 and Runge-Kutta4 methods correspondingly with mesh size $\Delta=0.1$. Solid curves of ResNet errors tend to merge into the dashed lines of target errors over iterations. For this example, errors of Runge-Kutta2 is smaller than that of Runge-Kutta4.

The simulated solution trajectory starts at $\vx=(3, 3)$ and runs up to final time  $T=20.0$. In Figure \ref{fig:ex6-result} part (d) we present the three well trained ResNet solvers approximations of the trajectory. For visualization purpose, every six output points are skipped when drawing the results of part (d) of Figure \ref{fig:ex6-result}. The target study shows the ResNet trained from Runge-Kutta2 gives smaller error in terms of one step implementation with time lag $\Delta=0.1$. However, the three ResNet solvers all agree well with the reference solution after long time run of $T=20.0$.

%%%%%%%%%%%%%%%%%%%%%%%%%%%%%%%%%%%%%%%%%%%%%%%%%%

\begin{example}\label{ex6} {\bf \emph {Non-autonomous ODE system}}
\end{example}

In this example we consider the following non-autonomous ODE system involving three unknowns and right hand side that depends on $t$ explicitly
\begin{fleqn}
\be\label{ex:ex6}
 \left\{\begin{aligned}
     \dot{x_1} &= x_1+x_3-t+e^{-t},\\
\dot{x_2} &=x_1+x_2+5,  \\
\dot{x_3} &= -2x_1-x_3-2t-e^{-t}.
 \end{aligned}
\right.
\ee
\end{fleqn}
The ResNet solver can not handle the non-autonomous ODE system directly, since the time variable $t$ changes at each time step and should be treated as an input in the neural network. Thus we introduce auxiliary variable $x_4=t$ and rewrite the system of \eqref{ex:ex6} into the following autonomous  system with four variables
\begin{fleqn}
\begin{equation*}
 \left\{\begin{aligned}
     \dot{x_1} &= x_1+x_3-x_4+e^{-x_4},\\
\dot{x_2} &=x_1+x_2+5,  \\
\dot{x_3} &= -2x_1-x_3-2x_4-e^{-x_4},\\
\dot{x_4} &=1.
\end{aligned}
\right.
\end{equation*}
\end{fleqn}

\begin{figure}[htbp]
\centering
\subfigure[Architecture study with $L_{\infty}$ error]{
\includegraphics[width=0.33\linewidth]{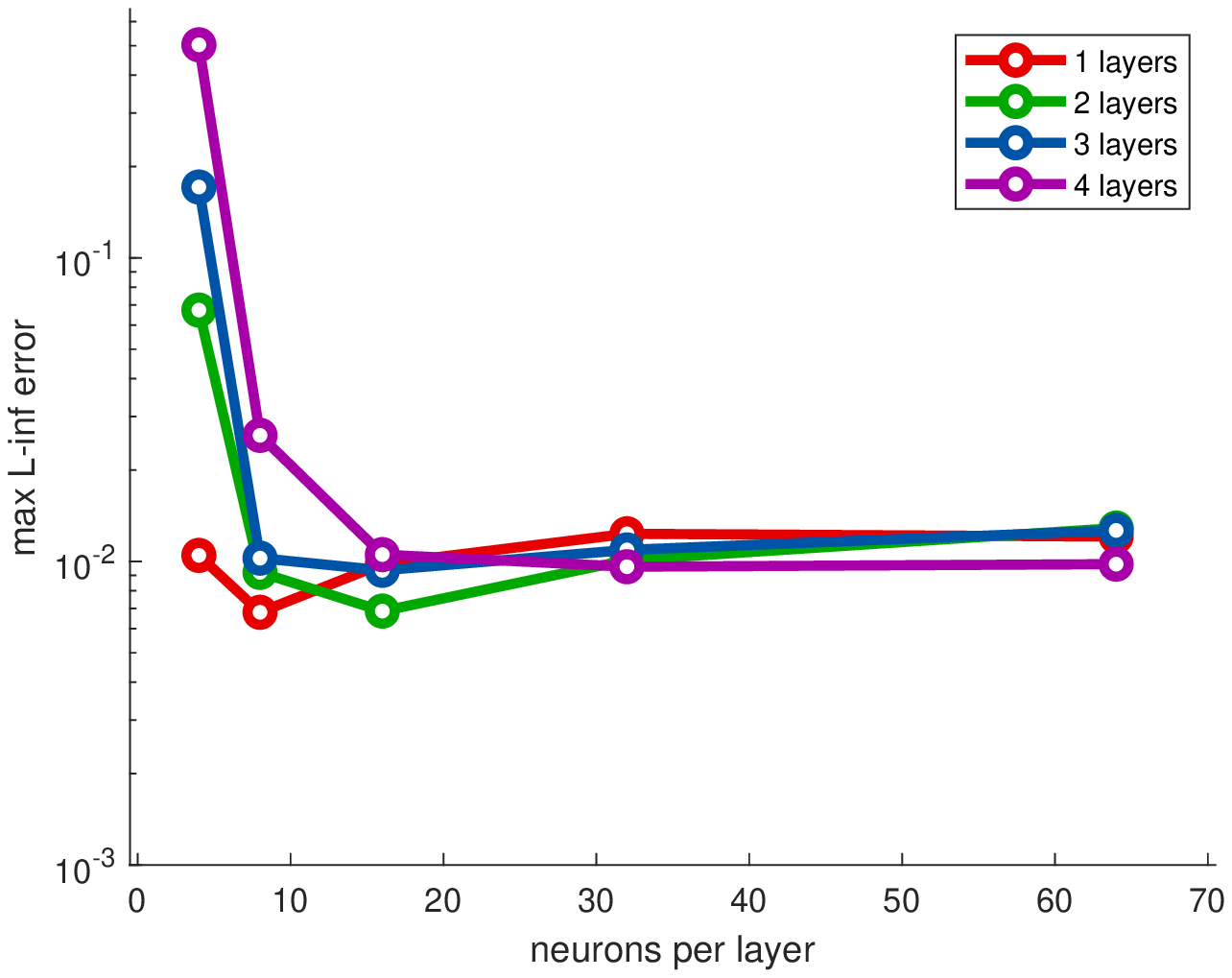}
}
\subfigure[Architecture study with $L_2$ error]{
\includegraphics[width=0.33\linewidth]{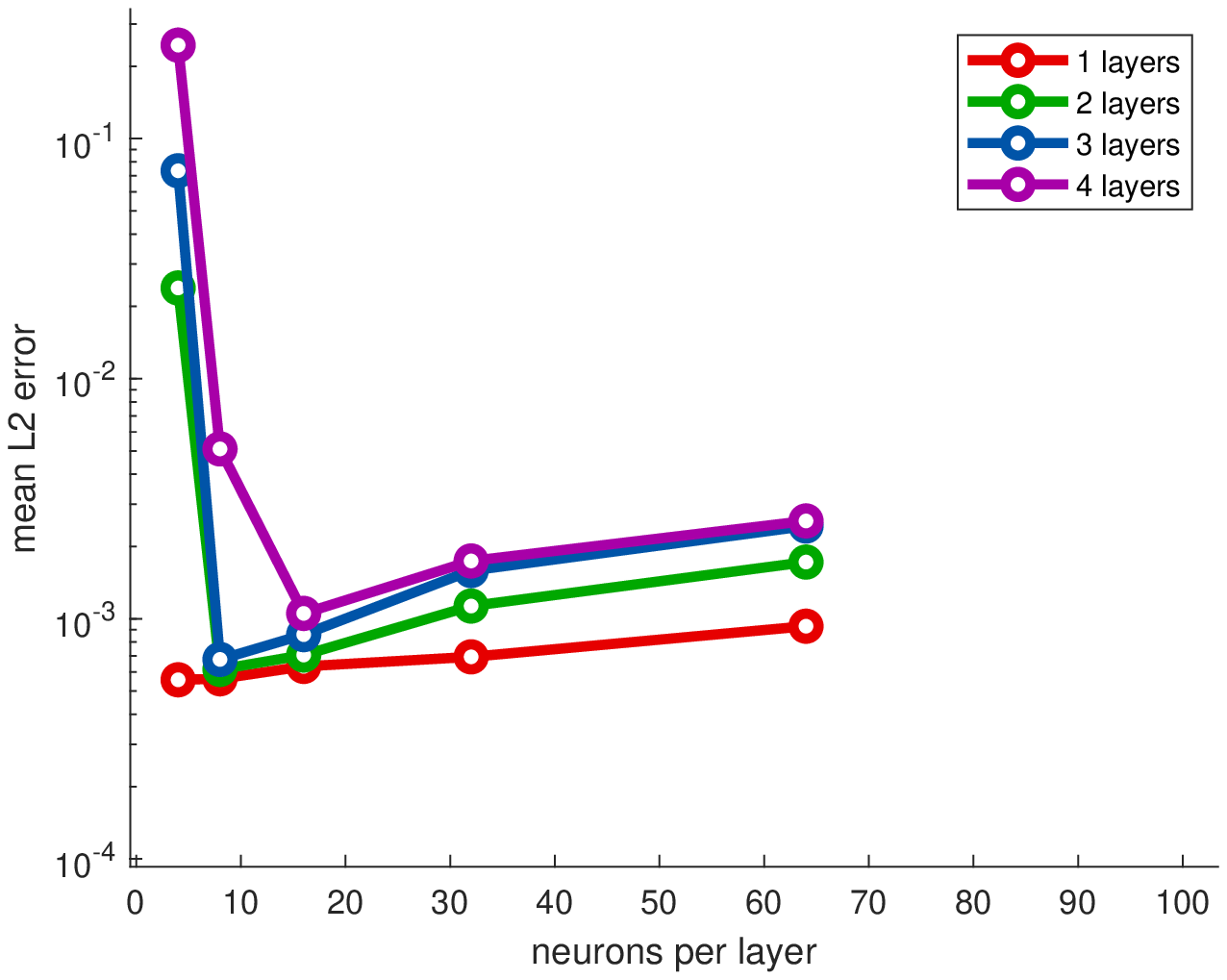}
%\caption{}
}\quad\quad
\subfigure[Target study on the system]{
\includegraphics[width=0.33\linewidth]{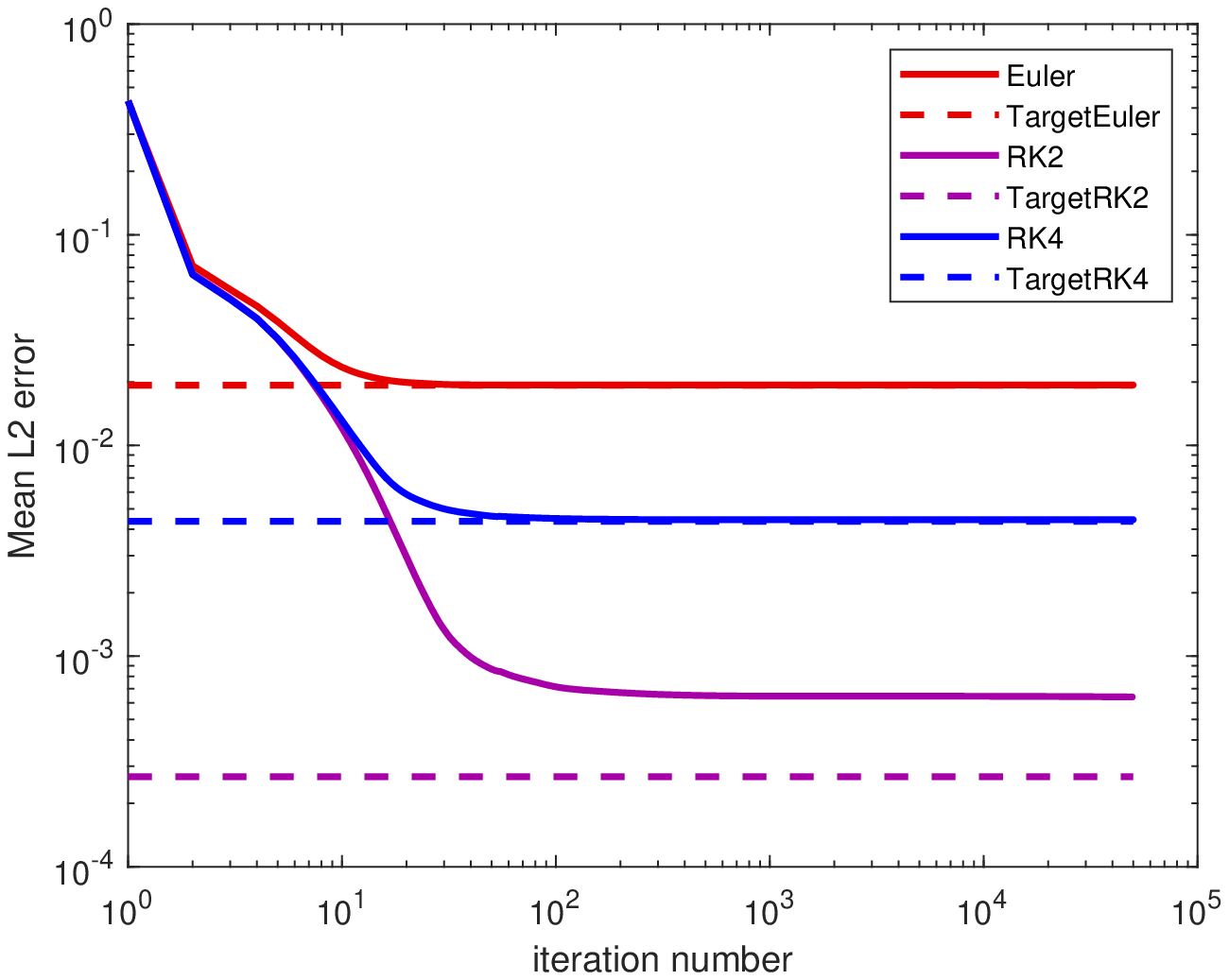}
%\caption{}
}
\subfigure[3D trajectory simulation]{
\includegraphics[width=0.33\linewidth]{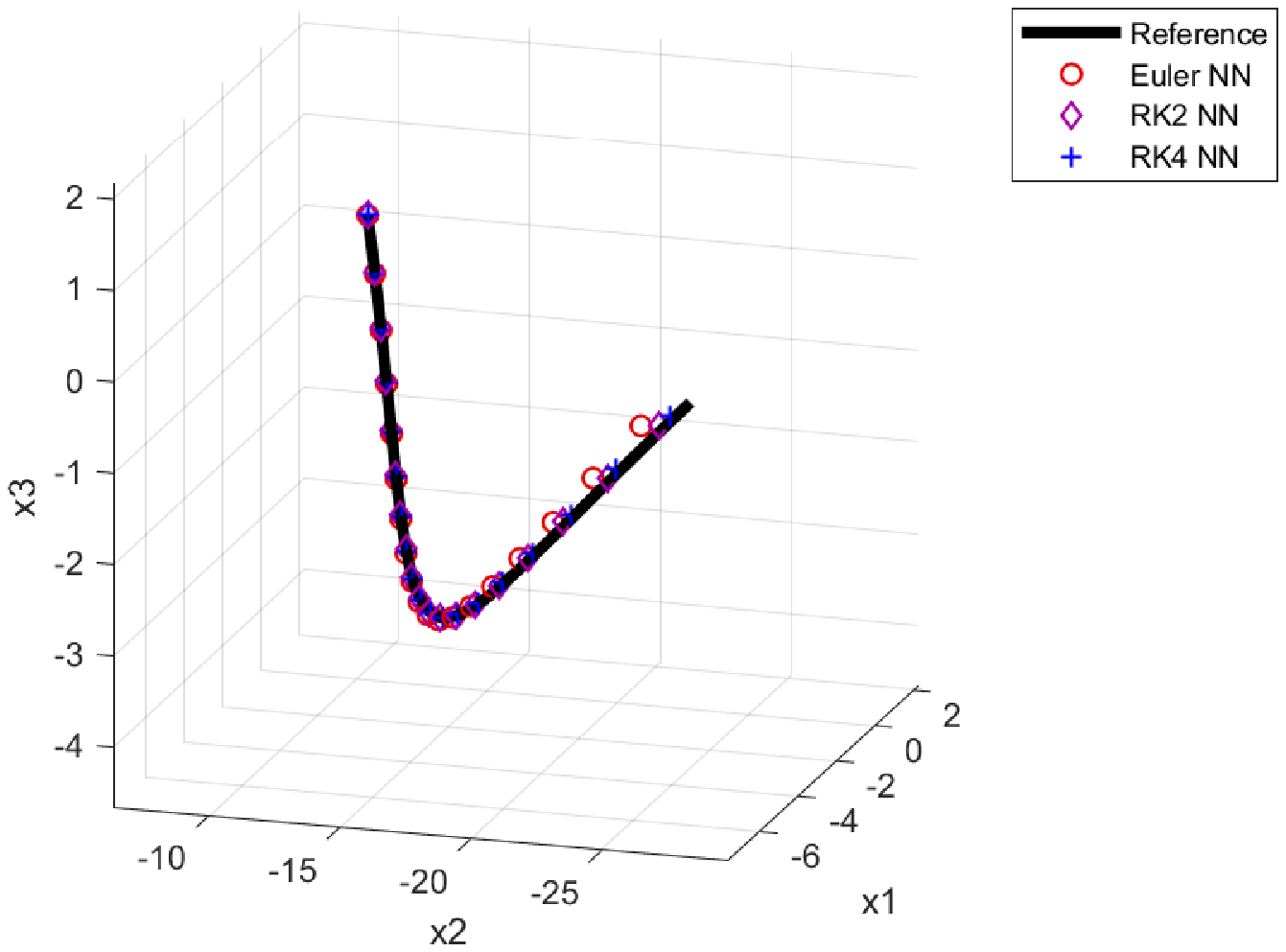}
%\caption{}
}
\caption{Example of \ref{ex6} of none autonomous ODE system
} 
\label{fig:ex3-result}
\end{figure}
Notice the system of (\ref{ex:ex6}) is equivalent to the above  autonomous ODE system. For this example we adapt the domain of interest of $D=[-5,5]\times[-10,0]\times[-6,4]\times[1, 2]$ and time lag of $\Delta= 0.05$. We still apply $J=2000$ training data pairs which are sufficient for the autonomous system since the major body of the system is linear. The architecture study shows the optimal choice is 1 layer with 8 neurons per layer, according to the errors of \eqref{error-Max-Linfy} and \eqref{error-Mean-L2}. Again, errors with different arrangement of architecture settings are illustrated in Figure \ref{fig:ex7-result} of part (a) and part (b). 

For the target study, network output and target errors of \eqref{error-Mean-L2} and \eqref{error-target} with training data generated from forward Euler, Runge-Kutta2 and Runge-Kutta4 with mesh size $\Delta=0.05$ are displayed in Figure \ref{fig:ex3-result} part (c). Solid curves of ResNet errors merge into the dashed lines of target errors over iterations. For this example, target and ResNet errors from Runge-Kutta2 are smaller than that of Runge-Kutta4. The networks are still well trained and ResNet errors are dominated by target errors.

Regarding solution curve approximation, we have the trajectory initial location set at $(2,-9, 0)$ with initial time $t_0=1.1$ and run up to final time $T=2.0$. The three well trained ResNet simulations accompanied with the reference solution are presented in Figure \ref{fig:ex3-result} part (d). The three ResNet solvers all agree well with the reference solution.

% \begin{figure}[htbp]
% \centering
% \includegraphics[width=0.45\linewidth]{fig/ex3-exact.jpg}
% \end{figure}

%From \eqref{exact_solution_ex3}, we attach the exact solution (general solution) expression of the original system (\ref{ex:example3}). The system linear part coefficient matrix have 3 eigenvalues, with the first 2 as complex conjugate pairs of $\lambda_{1,2}=\pm i$ and the third as $\lambda_3=1$. The first 2 eigenvectors are the real and imaginary part of the eigenvector for $\lambda_1=+i?$, and the third is the eigenvector for $\lambda_3=1$. With initial value $(x^I_1,x^I_2,x^I_3)$, for example at $t=0$, we are supposed to plug into the general solution and figure out the coefficients of $c_1, c_2, c_3$. Yet, we see, if we choose $t_0=0$ and we have $\sin(0)=0$, we will never recover $c_2$ coefficient? That seems to imply initial time should not be taken as $t_0=0$? Any thought?

%\begin{figure}[htbp]
%\centering
%\quad
%\subfigure[Accuracy training on the system]{
%\includegraphics[width=0.7\linewidth]{fig/ex2/ex2-different_target2.eps}
%\caption{}
%}
%\label{fig:ex3-result2}
%\end{figure}

%%%%%%%%%%%%%%%%%%%%%%%%%%%%%%%%%%%%%%%%%%%%%%%%%%
\begin{example}\label{ex7} {\bf \emph {Van der Pol oscillator}}
\end{example}

In this example, we consider the following second order ODE
\be\label{eq:ex7}
u''-\mu (1-u^2)u'+u=0,
\ee
that describes the current $u$ in an electric circuit involving a triode. Parameter $\mu>0$ is a constant that determines the sharpness of the oscillatory limit cycle. Here we have $\mu=0.2$. We introduce variable $x_2(t)$ to approximate $u'(t)$, with $x_1(t)=u(t)$ we rewrite the second order ODE of \eqref{eq:ex7} into the following first order system
\begin{equation*} \label{eq:ex7-system}
 \left\{\begin{aligned}
     \dot{x_1} &= x_2,\\
\dot{x_2} &= -x_1+\mu(1-x_1^2)x_2.
 \end{aligned}
\right.
\end{equation*}
Since the second variable $x_2$ represents the solution slope that changes dramatically over the limit cycle, we adapt the domain of interest as $D=[-3, 3]\times[-20, 20]$ to accommodate the $x_2$ variable. Time lag of $\Delta= 0.05$ is applied.
\begin{figure}[ht]
\centering
\subfigure[Architecture study with $L_{\infty}$ error]{
\includegraphics[width=0.33\linewidth]{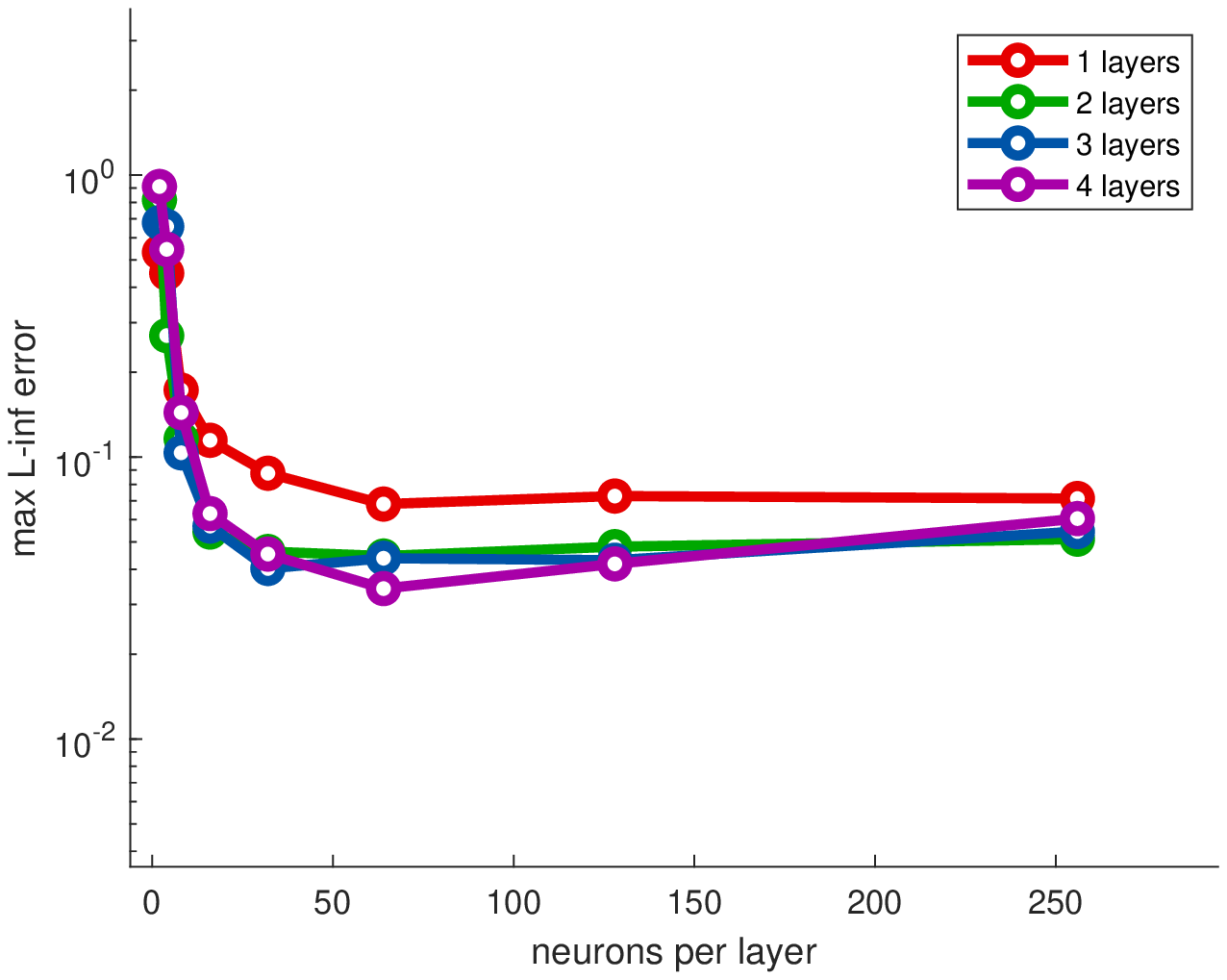}
%\caption{}
}
\subfigure[Architecture study with $L_{2}$ error]{
\includegraphics[width=0.33\linewidth]{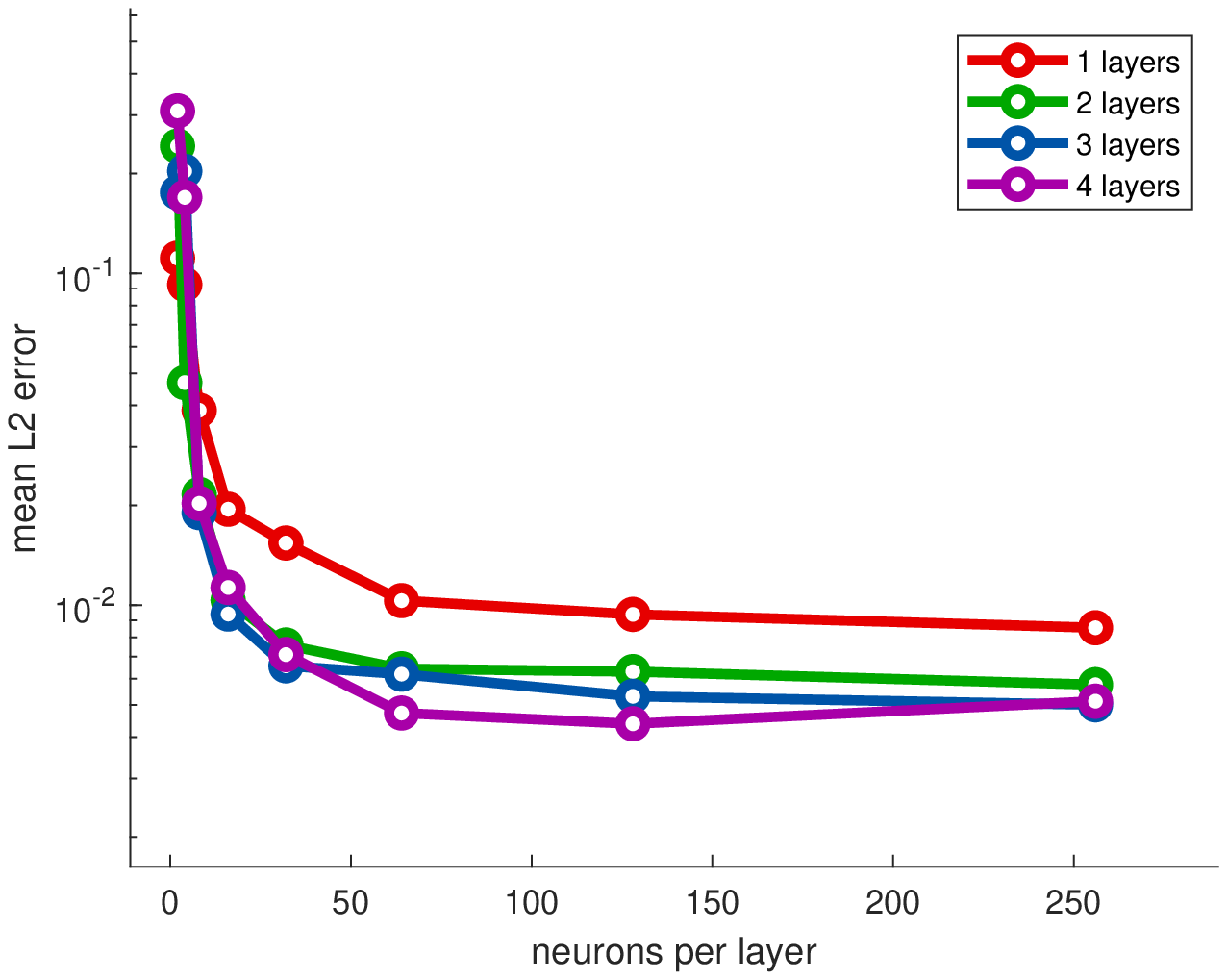}
%\caption{}
}
\quad
% \subfigure[Optimally trained network vs true trajectory]{
% \includegraphics[width=0.45\linewidth]{fig/ex8/ex8-optimal.eps}
% %\caption{}
% }
\subfigure[Target study on the system]{
\includegraphics[width=0.33\linewidth]{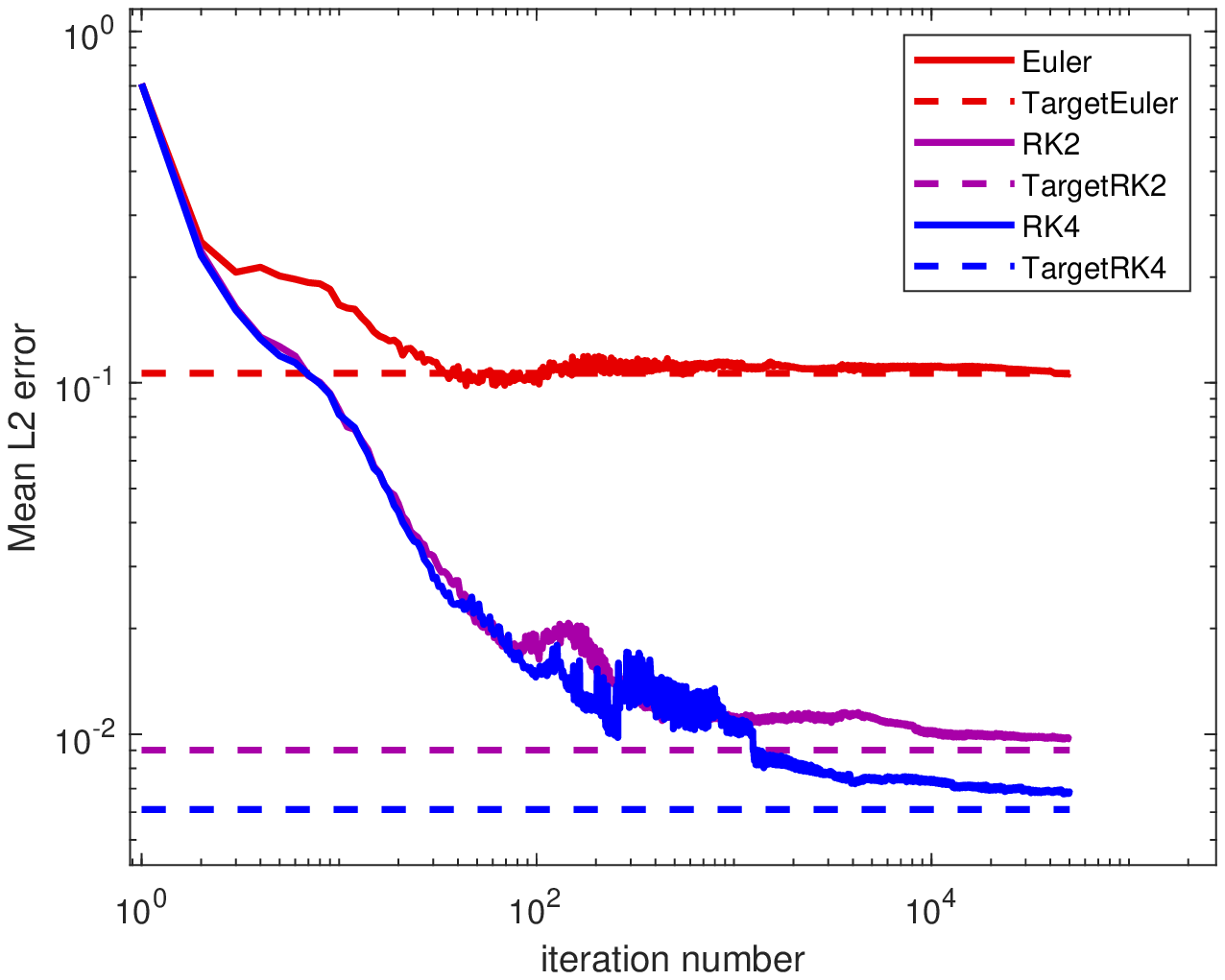}
%\caption{}
}
\subfigure[Phase plane trajectory simulation]{
\includegraphics[width=0.33\linewidth]{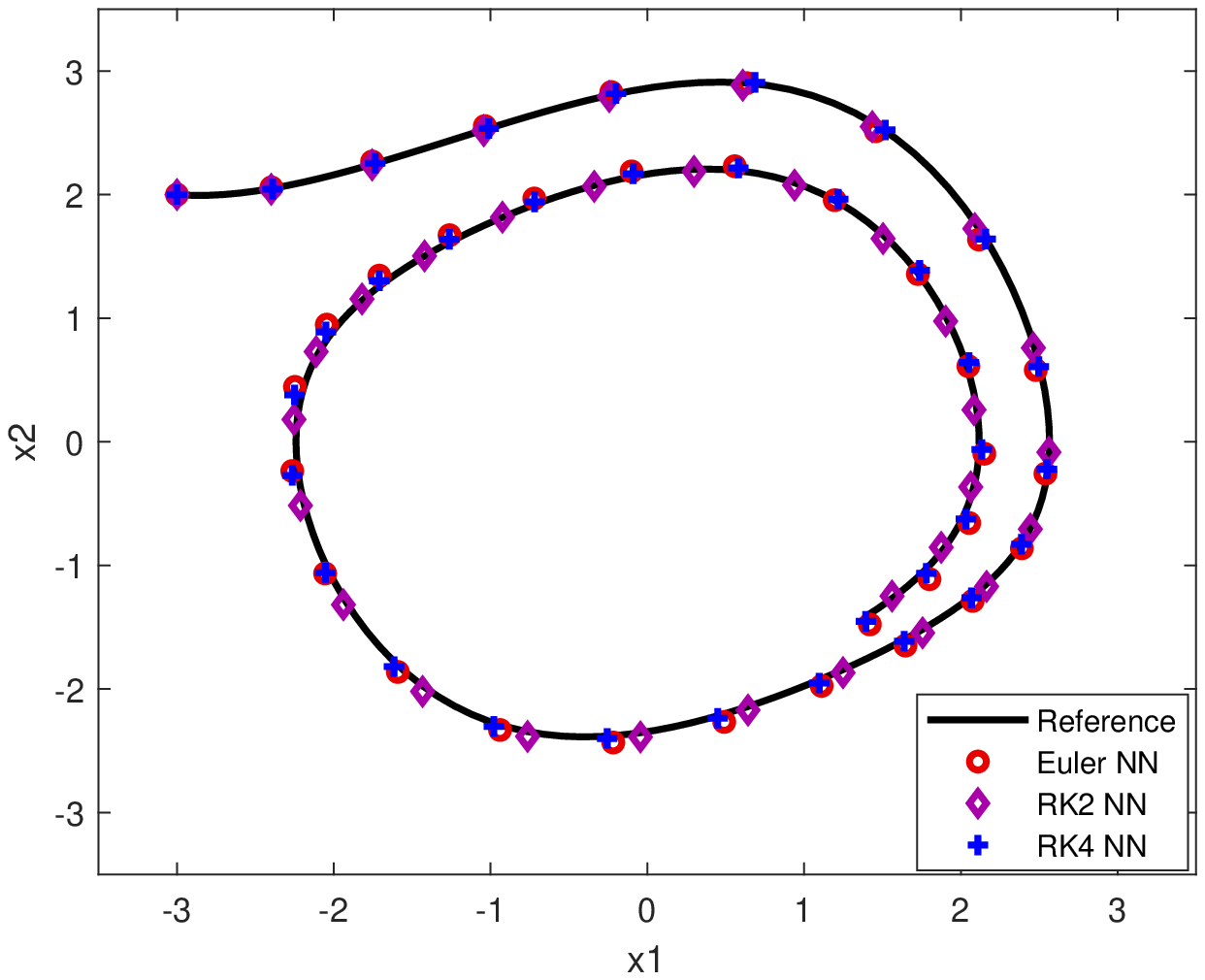}
%\caption{}
}
\caption{Example \ref{ex7} of the Van der Pol equation
} 
\label{fig:ex7-result}
\end{figure}

For the architecture study, maximum $L_{\infty}$ error of \eqref{error-Max-Linfy} and mean $L_2$ error of \eqref{error-Mean-L2} for each architecture setting are calculated and illustrated in Figure \ref{fig:ex7-result} of part (a) and part (b). The optimal architecture we find for the network is 2 hidden layers with 64 neurons per layer.

For the target study, network output and target errors of \eqref{error-Mean-L2} and \eqref{error-target} with training targets generated from forward Euler, Runge-Kutta2 and Runge-Kutta4 with mesh size $\Delta=0.05$ are displayed in Figure \ref{fig:ex7-result} part (c). Solid curves of ResNet errors merge into the dashed lines of target errors over iterations. Networks are well trained and ResNet errors are dominated by target errors.

The simulated solution curve initial conditions are picked as $u(0)=-3$ and $u'(0)=2$ and runs up to final time $T=10$. In Figure \ref{fig:ex7-result} part (d) we present the three well trained ResNet approximations of the curve. The output symbols are drawn with a few points skipped to avoid dense representation of the curve. The three ResNet solvers behave similarly and all agree well with the reference solution. 

%%%%%%%%%%%%%%%%%%%%%%%%%%%%%%%%%%%%%%%%%%%%%%%%%%
\begin{example}\label{ex8} {\bf \emph {Fitzhugh–Nagumo equation}}
\end{example}

In this example we study the following ODE system
\begin{fleqn}
\be
 \left\{\begin{aligned}
     \dot{x_1} &= 3(x_1+x_2-\frac{1}{3}x_1^3-k),\\
     \dot{x_2} &= -\frac{1}{3}(x_1+0.8x_2-0.7),
 \end{aligned}
\right.
\ee
\end{fleqn}
which models the transmission of neural impulses along an axon. Coefficient $k$ is the external stimulus. We have $k = 0.5$ taken in this example which leads to the critical point as an unstable spiral point. 
\begin{figure}[ht]
\centering
\subfigure[Architecture study with $L_{\infty}$ error]{
\includegraphics[width=0.33\linewidth]{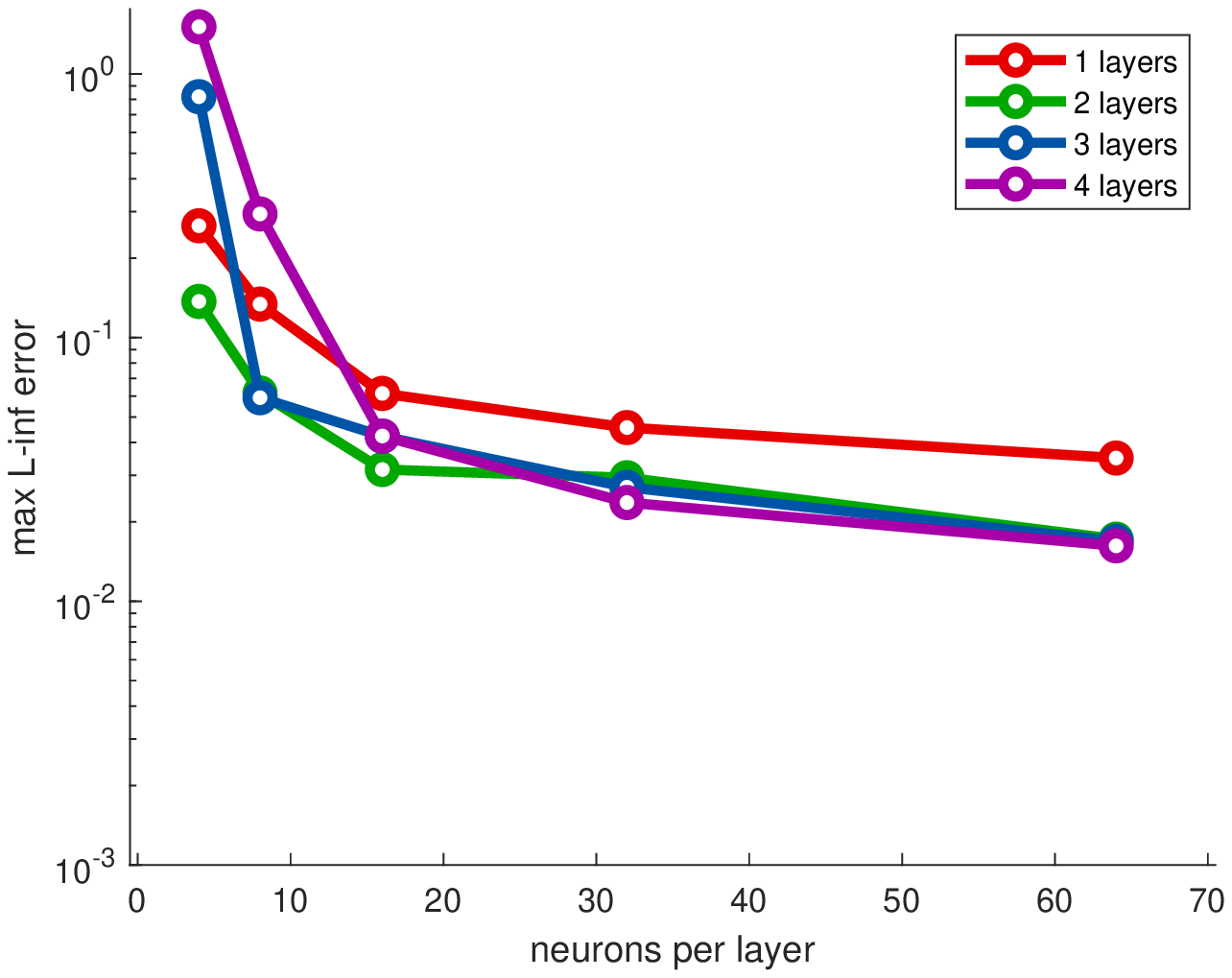}
%\caption{}
}
\subfigure[Architecture study with $L_{2}$ error]{
\includegraphics[width=0.33\linewidth]{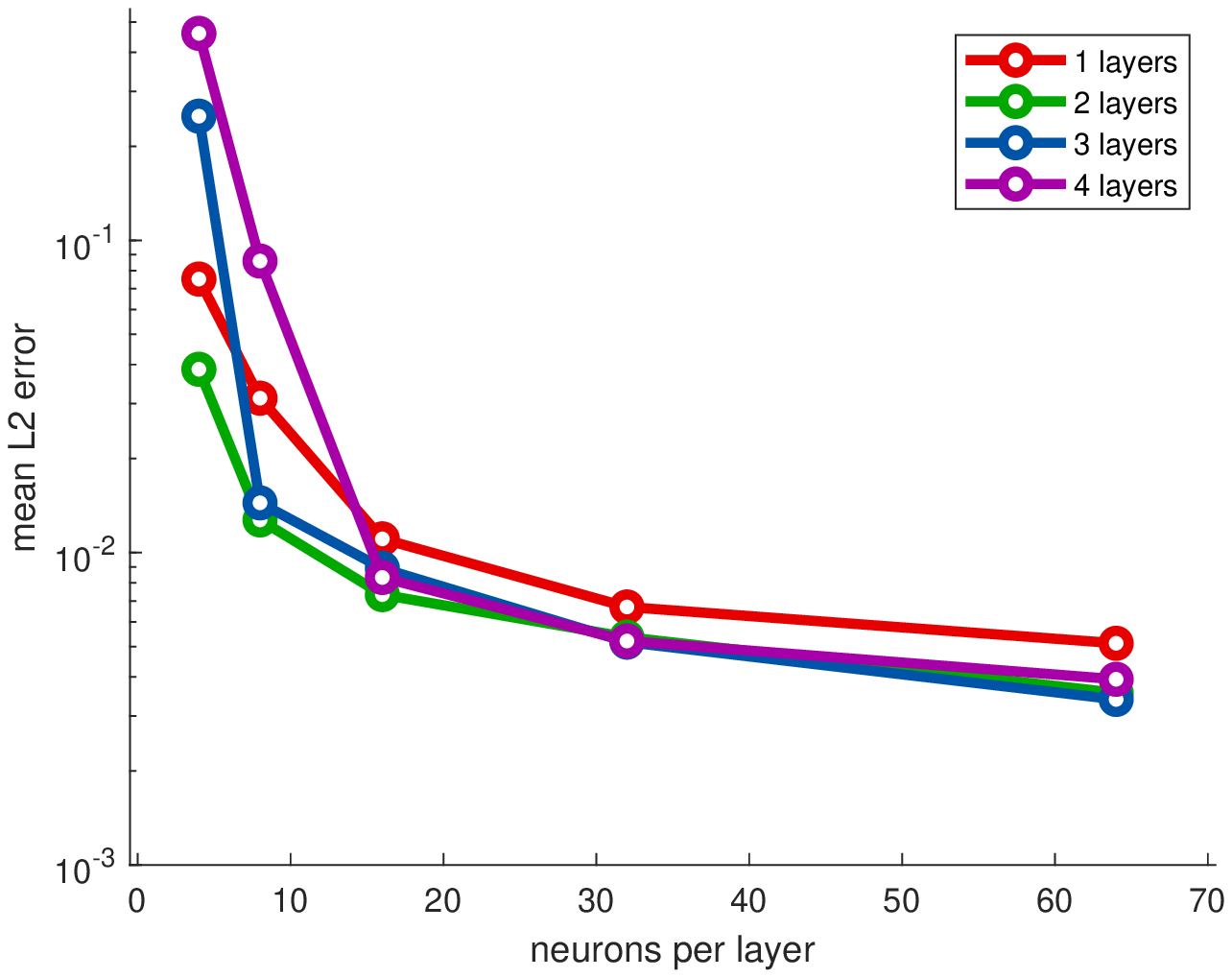}
%\caption{}
}
\quad
\subfigure[Target study on the system]{
\includegraphics[width=0.33\linewidth]{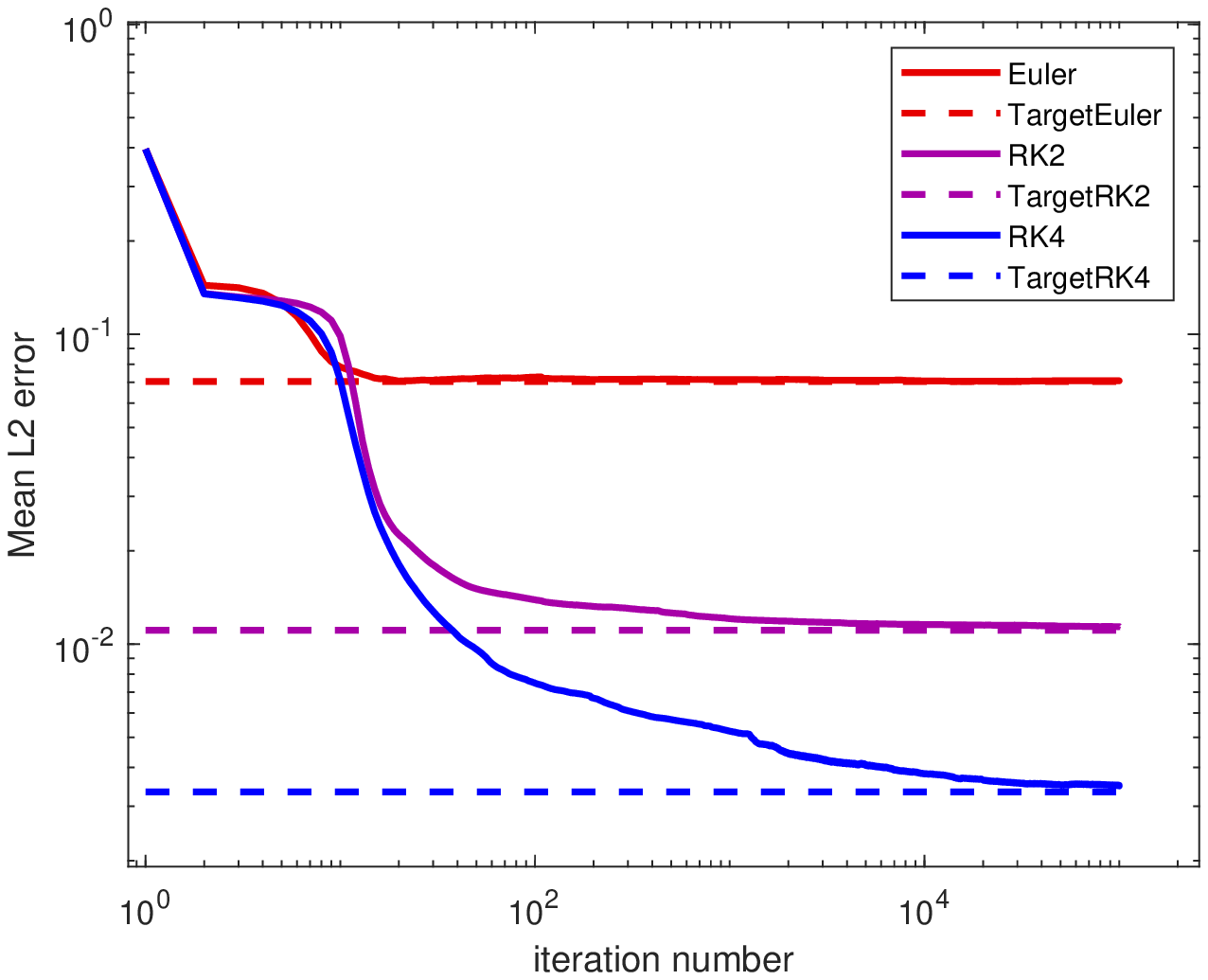}
%\caption{}
}
\subfigure[Phase plane trajectory simulation]{
\includegraphics[width=0.33\linewidth]{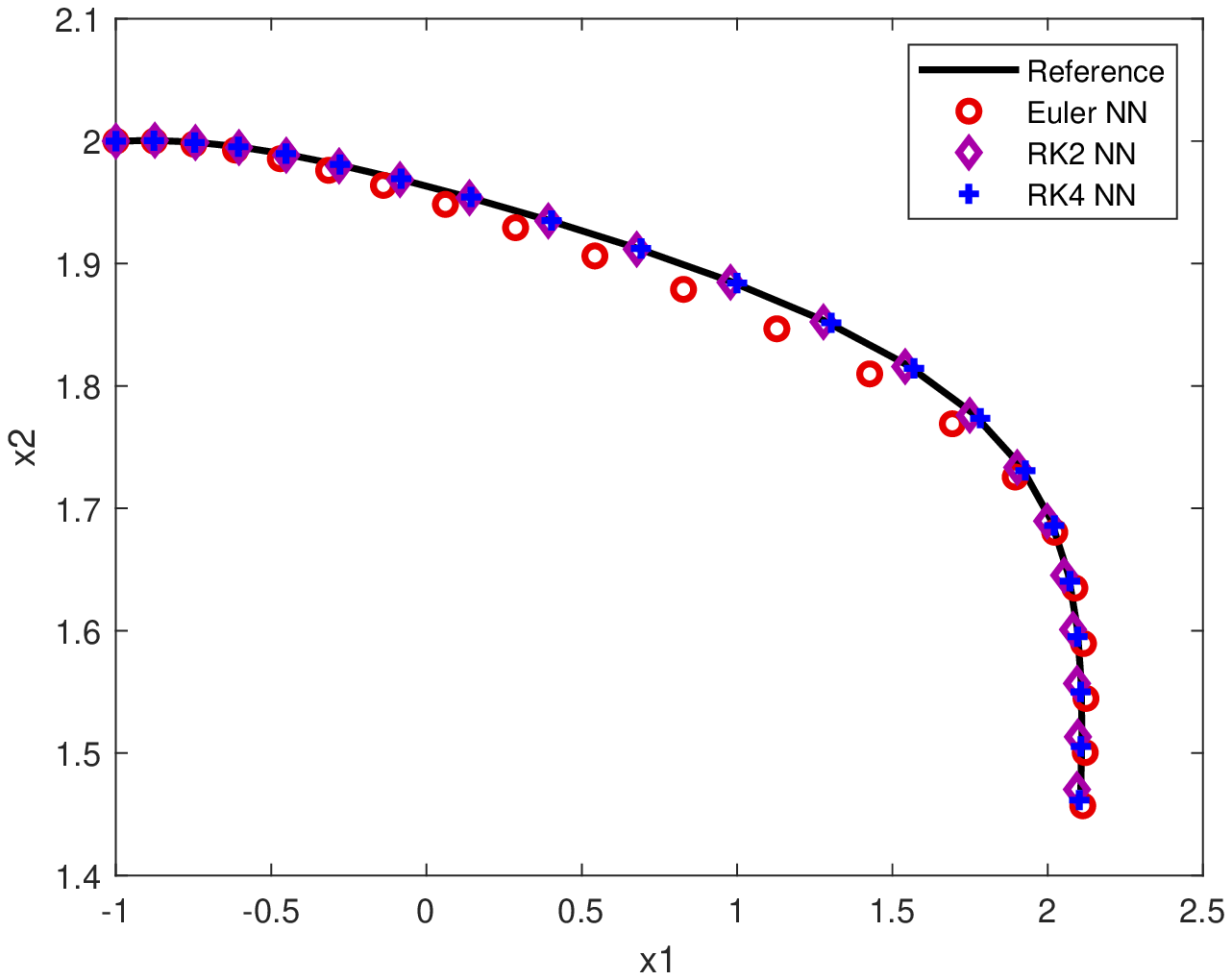}
%\caption{}
}
\caption{Example \ref{ex8} of the Fitzhugh–Nagumo equation
} 
\label{fig:ex8-result}
\end{figure}

Time lag of $\Delta= 0.05$ is applied and domain of interest is taken as $D=[-5, 5]\times [-5, 5]$. For the architecture study, maximum $L_{\infty}$ error of \eqref{error-Max-Linfy} and mean $L_2$ error of \eqref{error-Mean-L2} for each architecture setting are calculated and illustrated in Figure \ref{fig:ex8-result} of part (a) and part (b). The optimal network architecture picked has 2 hidden layers and 64 neurons per layer.

For the target study, network output and target errors of \eqref{error-Mean-L2} and \eqref{error-target} with training targets generated from forward Euler, Runge-Kutta2 and Runge-Kutta4 with mesh size $\Delta=0.05$ are displayed in Figure \ref{fig:ex8-result} part (c). ResNet errors match well with the target errors over iterations. 

The approximated curve starts at $(-1,2)$ and we run the simulation to $T=1.0$. In Figure \ref{fig:ex8-result} part (d) the three well trained ResNet solvers are applied to approximate the curve. All three ResNet output agree well with the reference solution.

%\begin{figure}[htbp]
%\centering
%%\includegraphics[width=0.45\linewidth]{fig/phase-portrait-demo-ex6.eps}
%\includegraphics[width=0.45\linewidth]{fig/phase-portrait-demo-ex7.eps}
%\caption{phase portrait for example 8 (left) and 6 (right)}
%\end{figure}

%%%%%%%%%%%%%%%%%%%%%%%%%%%%%%%%%%%%%%%%%%%%%%%%%%
\begin{example}\label{ex9} {\bf \emph {Genetic toggle switch}}
\end{example}

In this example, we consider the following nonlinear differential-algebraic equation
\begin{fleqn}
\be
 \left\{\begin{aligned}
     \dot{x_1} &= \frac{\alpha_1}{1+x_2^{\beta}}-x_1,\\
     \dot{x_2} &= \frac{\alpha_2}{1+z\gamma}-x_2,\\
     z&=\frac{x_1}{(1+[IPDG]/K)^{\eta}}.
 \end{aligned}
\right.
\ee
\end{fleqn}
which is used to model a genetic toggle switch in Escherichia coli. It is composed of two repressors and two constitutive promoters, where each promoter is inhibited by the repressor that is transcribed by the opposing promoter. Details of experimental measurement can be found in \cite{Chartrand-2011}. 

\begin{figure}[htbp]
\centering
\subfigure[Architecture study with $L_{\infty}$ error]{
\includegraphics[width=0.31\linewidth]{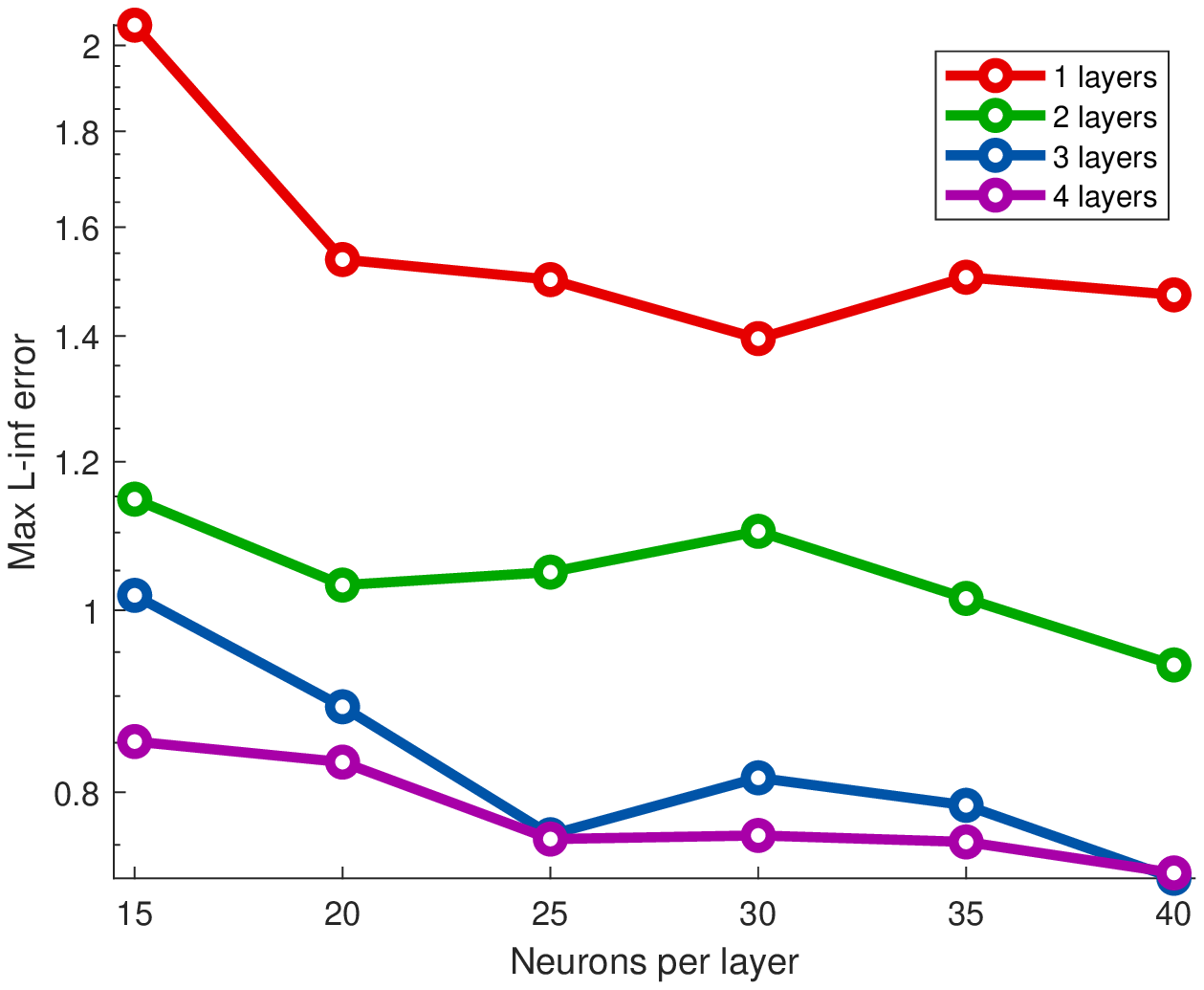}
}
\subfigure[Architecture study with $L_2$ error]{
\includegraphics[width=0.31\linewidth]{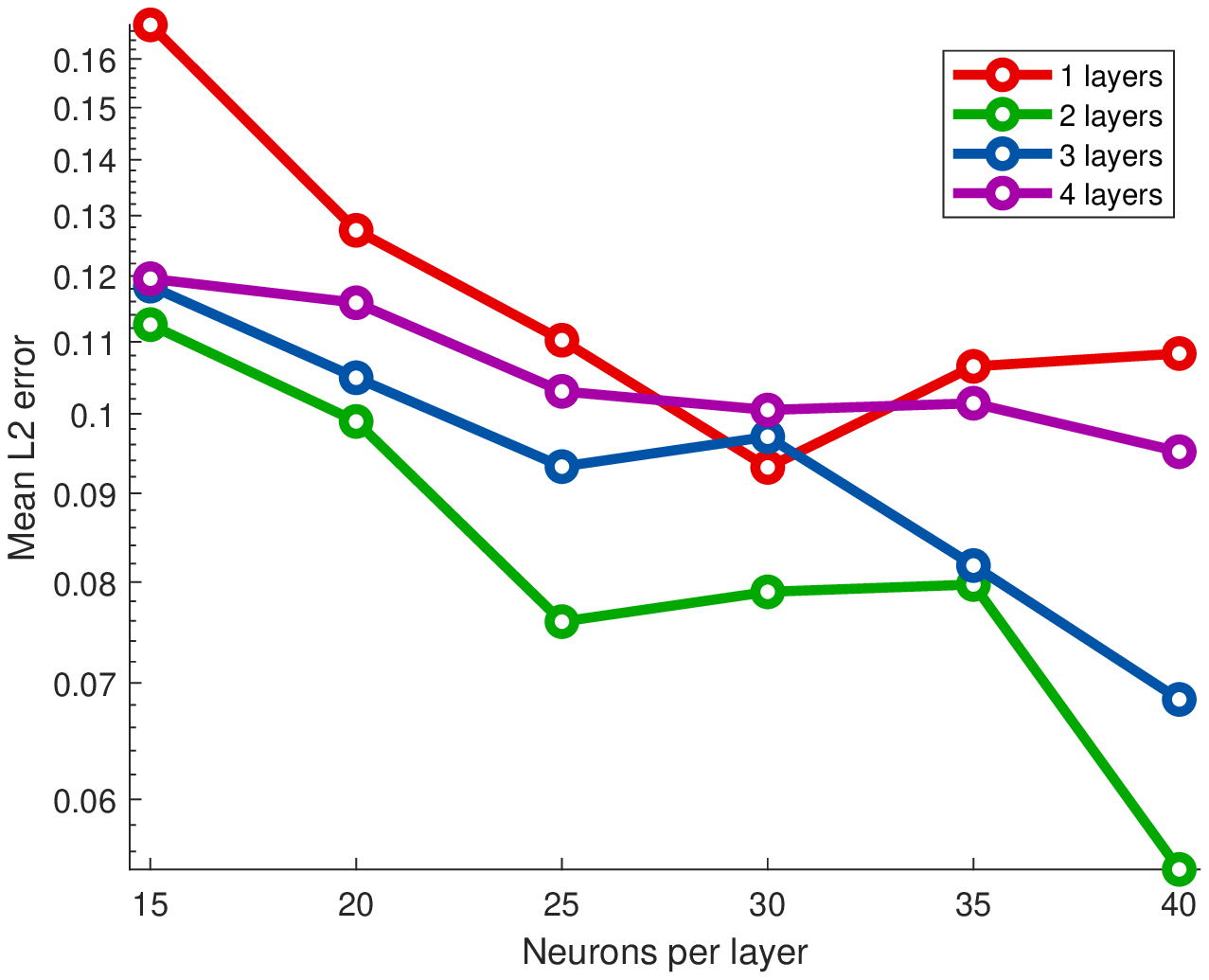}
}\quad
\subfigure[Target study on the system]{
\includegraphics[width=0.31\linewidth]{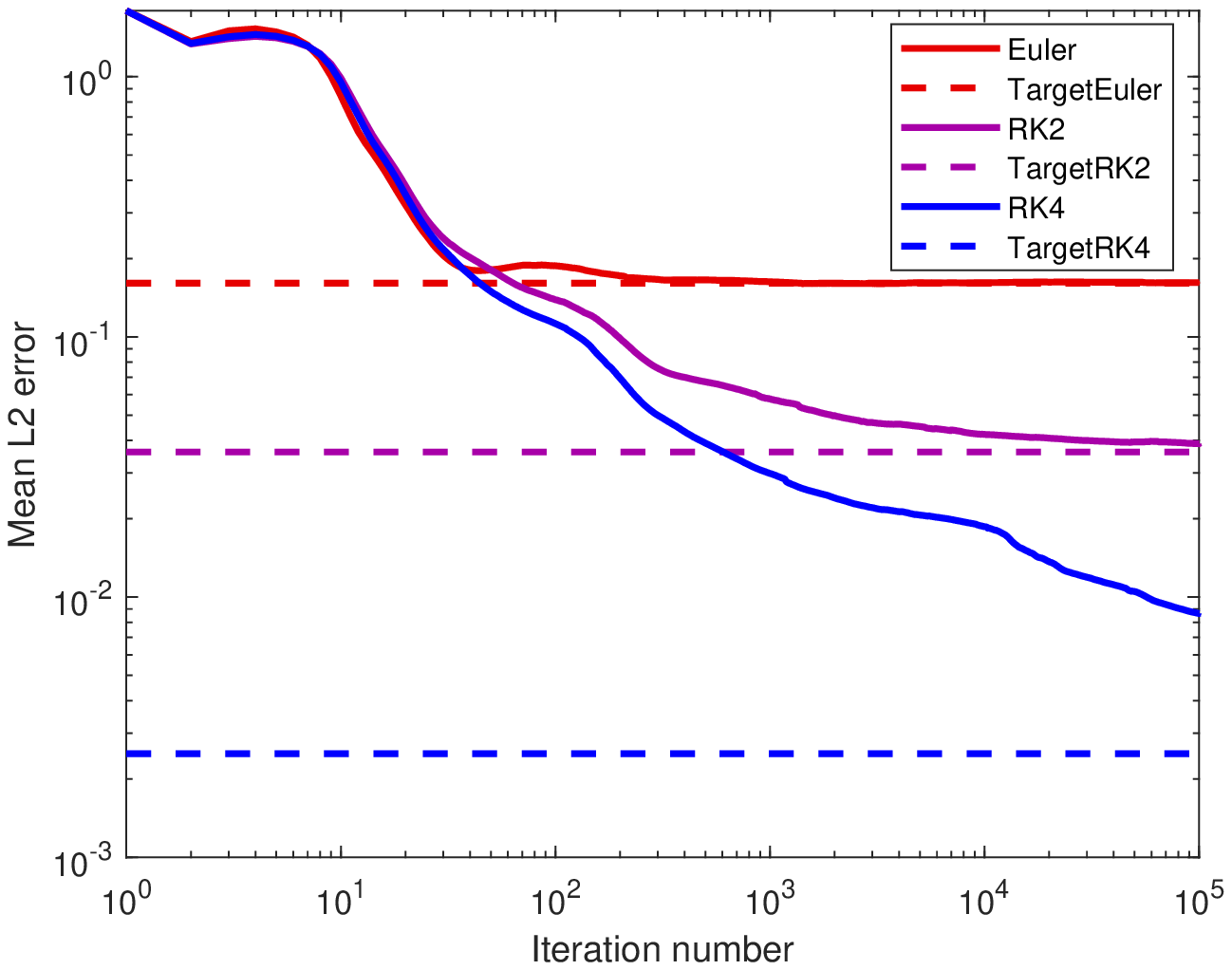}
}
\quad
\subfigure[Plase plane trajectory simulation]{
\includegraphics[width=0.31\linewidth]{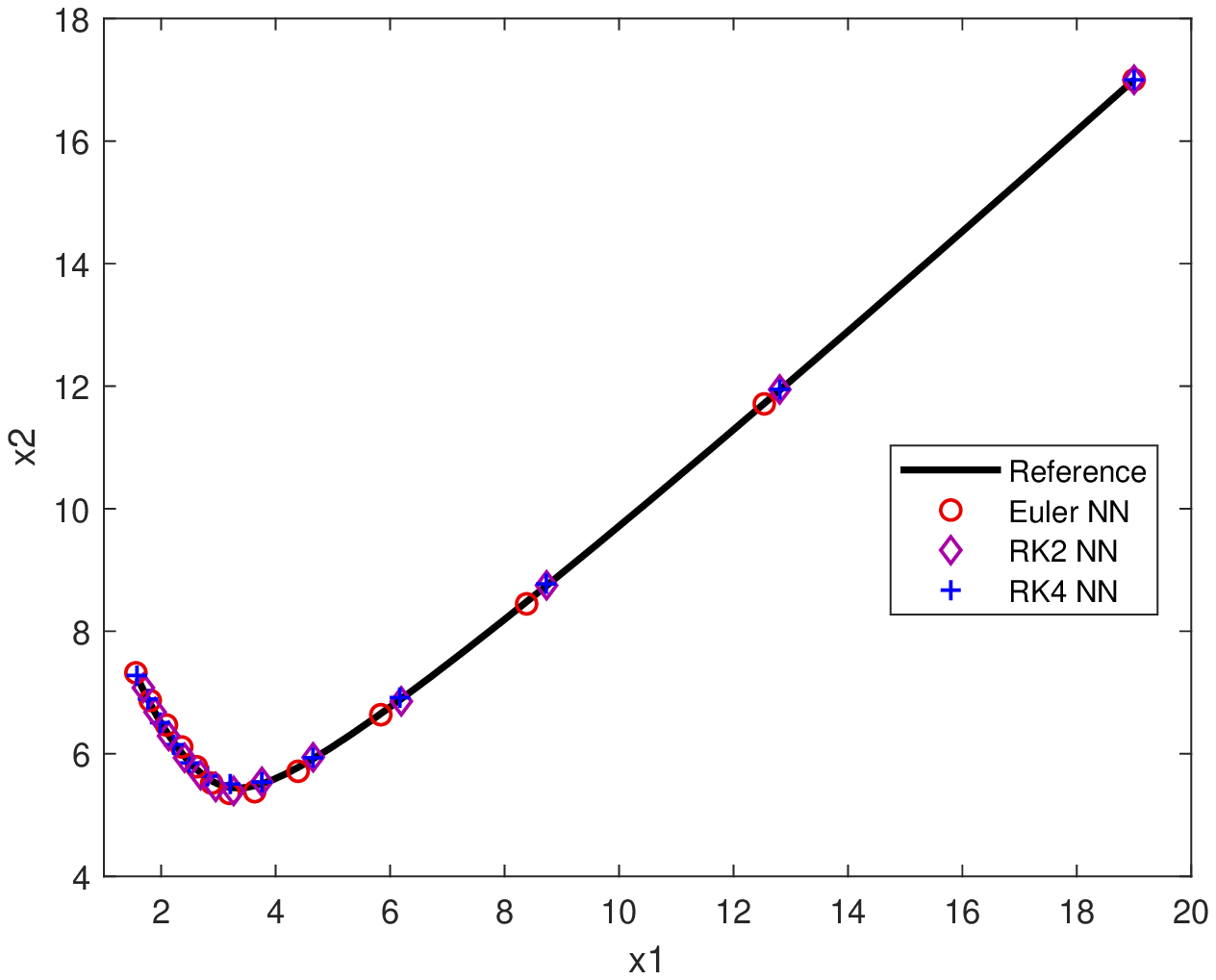}
}
\caption{Example \ref{ex9} of Genetic switch model
} 
\label{fig:ex9-results}
\end{figure}

Variables $x_1$ and $x_2$ denote the concentration of the two repressors. Coefficients $\alpha_1$ and $\alpha_2$ are the effective rates of the synthesis of the repressors. We have $\beta$ and $\gamma$ represent cooperativity of repression of the two promoters, respectively. We have $[IPTG]$ as the concentration of IPTG, the chemical compound that induces the switch, with $K$ the dissociation constant of IPTG.
In this example, we take $\alpha_1=156.25$, $\alpha_2=15.6$, $\gamma=1$, $\beta=2.5$, $K=2.9618\times 10^{-5}$ and $[IPDG]=10^{-5}$.

Domain of interested is $D =[0, 20]\times[0, 20]$. We observe that because points in $D$ can have a large magnitude, networks tend to fail during training. To remedy this, we scale the training set down by a factor of twenty, so that the network will be effectively trained on the domain $D'=[0,1]^2$. 

We conclude that 2 layers of 40 neurons is our optimal architecture. Target study results are displayed for in Figures \ref{fig:ex9-results}(c) and (d). Trajectory study is done with initial condition of $\vx = [19,17]$ for a total time interval of $T = 5.0$. For clarity, we only display every fourth network output.

%%%%%%%%%%%%%%%%%%%%%%%%%%%%%%%%%%%%%%%%%%%%%%%%%%
\begin{example}\label{ex10} {\bf \emph {Nonlinear electric network}}
\end{example}

In this example we consider another nonlinear differential-algebraic equation modeling a nonlinear electric network \cite{Pulch-2013}
\begin{fleqn}
\be
 \left\{\begin{aligned}
     \dot{x_1} &= v_2/C,\\
     \dot{x_2} &= x_1/L,\\
     0&=v_1-(G_0-G_{\infty})U_0~ \tanh(x_1/U_0)-G_{\infty}x_1,\\
     0&=v_2+x_2+v_1.
 \end{aligned}
\right.
\ee
\end{fleqn}
Here $x_1$ represents the node voltage and $x_2$, $v_1$ and $v_2$ are branch currents. Following \cite{Pulch-2013}, the physical parameters are specified as $C = 10^{-2}$, $L = 1$, $U_0 = 1$, $G_0 =-0.1$ and $G_{\infty} = 0.25$. All learning data pairs are taken from the domain of interest of $D = [-2, 2]\times [-0.2, 0.2]$. Time lag is taken as $\Delta =0.05$.

\begin{figure}[htbp]
\centering
\subfigure[Architecture study with $L_{\infty}$ error]{
\includegraphics[width=0.33\linewidth]{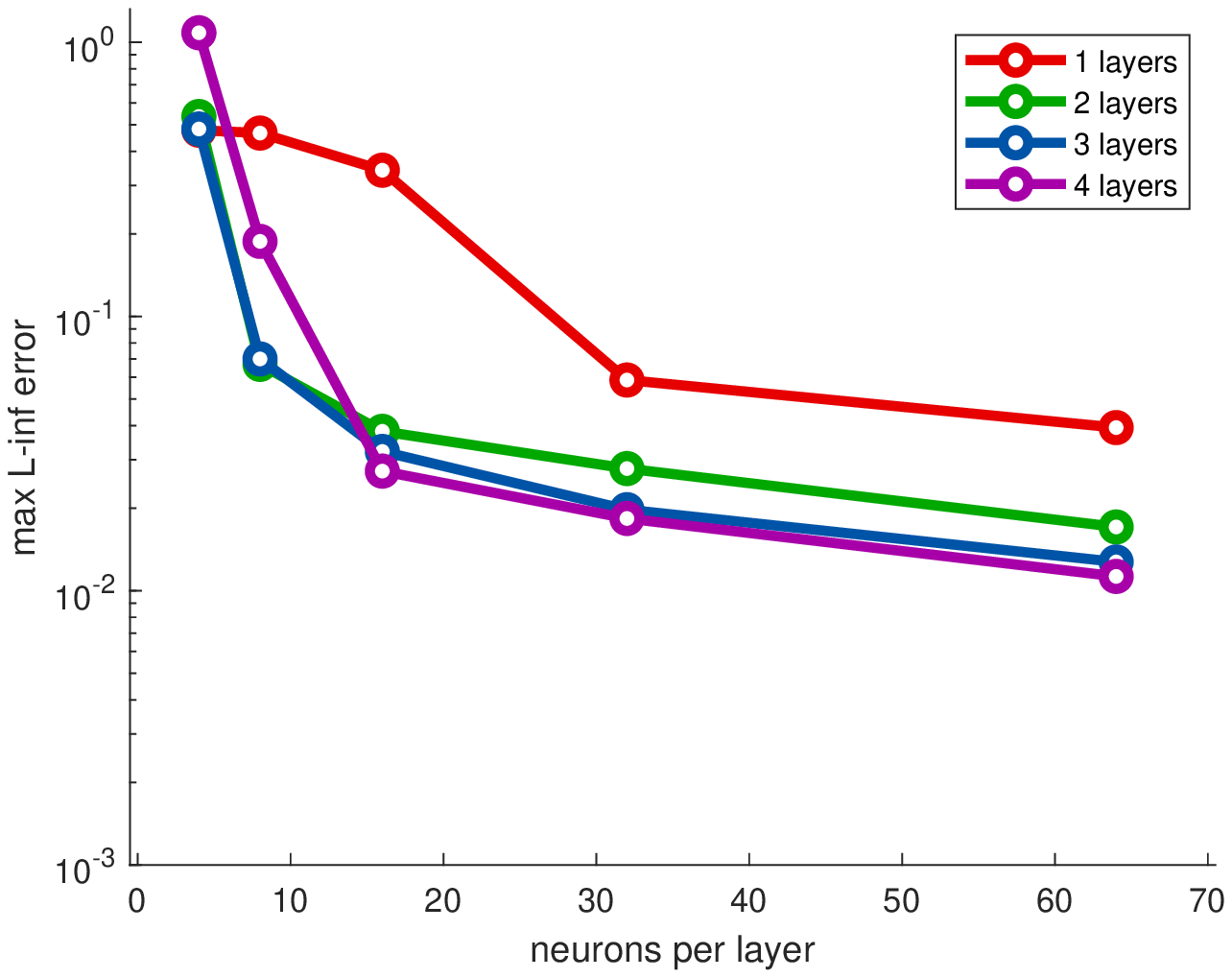}
%\caption{}
}
\subfigure[Architecture study with $L_{2}$ error]{
\includegraphics[width=0.33\linewidth]{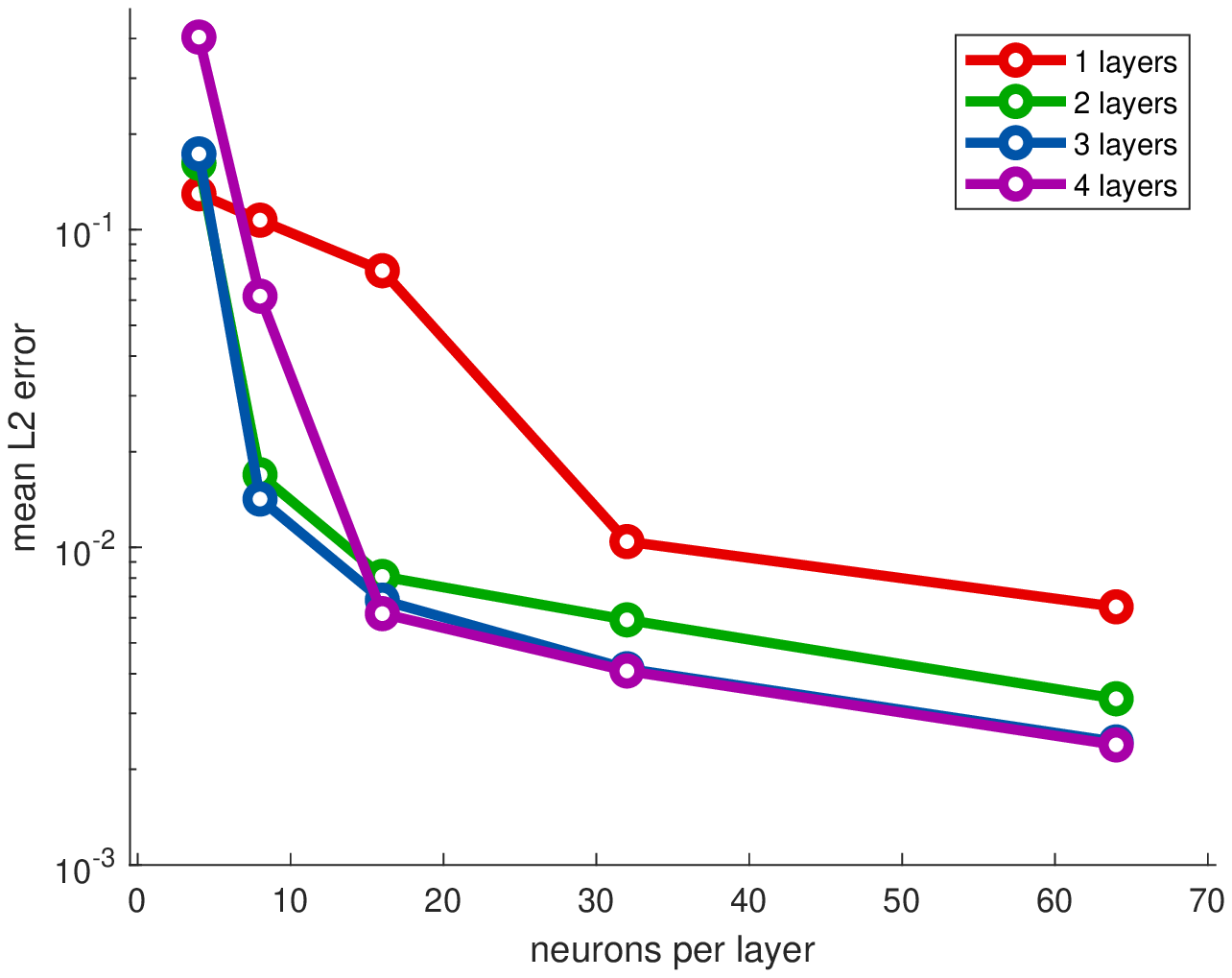}
%\caption{}
}
\quad
\subfigure[Target study on the system]{
\includegraphics[width=0.33\linewidth]{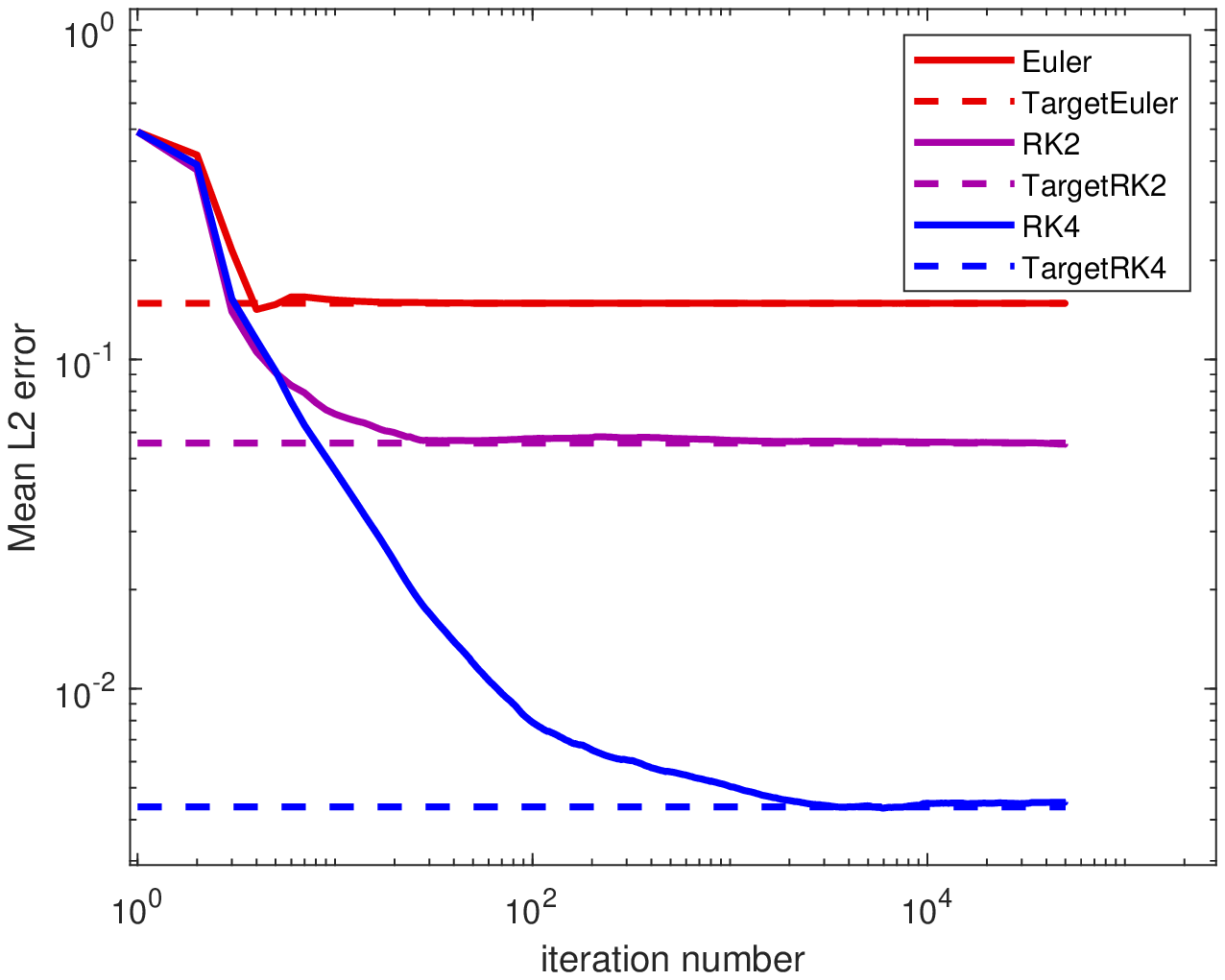}
%\caption{}
}
\subfigure[Phase plane trajectory simulation]{
\includegraphics[width=0.33\linewidth]{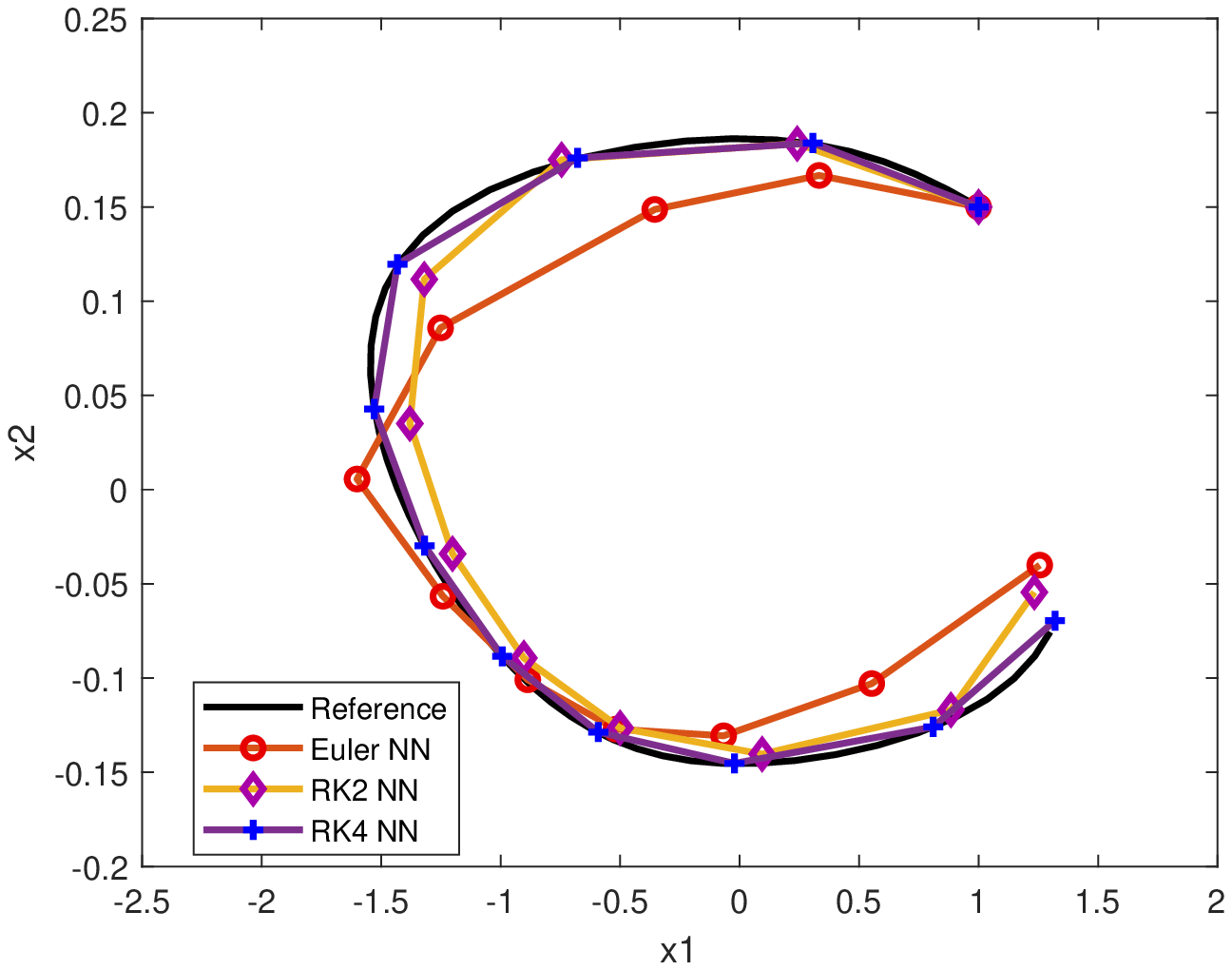}
%\caption{}
}
\caption{Example \ref{ex10} of the nonlinear electric network
} 
\label{fig:ex10-result}
\end{figure}
For the architecture study, maximum $L_{\infty}$ error of \eqref{error-Max-Linfy} and mean $L_2$ error of \eqref{error-Mean-L2} for each architecture setting are calculated and illustrated in Figure \ref{fig:ex10-result} of part (a) and part (b). The optimal network has 2 hidden layers and 64 neurons per layer.

For the target study, network output and target errors of \eqref{error-target} and \eqref{error-Mean-L2} with training targets generated from forward Euler, Runge-Kutta2 and Runge-Kutta4 with mesh size $\Delta=0.05$ are displayed in Figure \ref{fig:ex10-result} part (c). ResNet errors match well with the target errors over iterations. Target errors tend to be consistent to the analyzed error orders of $O(\Delta^2)$, $O(\Delta^3)$ and $O(\Delta^5)$ of the three finite difference schemes, even though the ODE system is nonlinear. 

The displayed trajectory starts at $\vx=(1, 0.15)$ and runs up to $T=0.5$.
In Figure \ref{fig:ex10-result} part (d) the three well trained ResNet solvers are applied to approximate the curve. The trajectory changes quickly with $\Delta=0.05$. ResNet solver from forward Euler gives large error and ResNet from Runge-Kutta4 agrees the best with the reference solution.

\section{Conclusions}
\label{S:5}

In this article, we consider the integral or weak formulation of ordinary differential equations and apply residual neural networks solving the ODEs. Specifically we investigate the optimal choice of hidden layers and neurons per layer as the ResNet architecture study. We also investigate the accuracy of ResNet solvers approximating the ODE solutions. Numerical tests show the accuracy of ResNet solver is dominated by the quality of the training target. Sequence of numerical examples verify the ResNet solver can be as accurate as any high order one step method, even the ResNet is implemented similarly to the first order forward Euler scheme. 



%% References


%% (1) according to the order of show up 
%% (2) according to the order of show up
\bibliographystyle{elsarticle-num}
%% (3) according to author's last name alphabetic order and full name showed up
%\bibliographystyle{plain}

\bibliography{Accuracy_and_architecture_studies_of_DNN}






% \end{thebibliography}

\end{document}